\documentclass[a4paper,10pt]{article}
\usepackage[utf8]{inputenc}
\usepackage[width=15cm,height=22cm]{geometry}
\usepackage{xcolor}
\usepackage{amsmath,amsfonts,amssymb,amsthm}
\usepackage[hidelinks,pdfusetitle]{hyperref}
\usepackage[capitalise]{cleveref}
\usepackage[usenames,dvipsnames]{pstricks}
\usepackage{epsfig}
\usepackage[boxed]{algorithm2e}
\newtheorem{thm}{Theorem}
\newtheorem{proposition}{Proposition}[section]
\newtheorem{definition}[proposition]{Definition}
\newtheorem{lemma}[proposition]{Lemma}
\newtheorem{corollary}[proposition]{Corollary}
\newtheorem{rk}[proposition]{Remark}
\numberwithin{equation}{section}

\newcommand*\samethanks[1][\value{footnote}]{\footnotemark[#1]}
\title{Unexpected quadratic behaviors for the small-time local null controllability of scalar-input parabolic equations}
\author{Karine Beauchard\texorpdfstring{\thanks{Univ Rennes, CNRS, IRMAR - UMR 6625, F-35000 Rennes, France}}{}, Fr\'ed\'eric Marbach\texorpdfstring{\samethanks}{}}

\newcommand{\C}{\mathbb{C}}
\newcommand{\N}{\mathbb{N}}
\newcommand{\R}{\mathbb{R}}
\newcommand{\dd}{\,\mathrm{d}}

\newcommand{\sinc}{\mathrm{sinc}}
\newcommand{\supp}{\mathrm{supp~}}

\newcommand{\gmu}{\mu}
\newcommand{\glambda}{\Gamma'[0]}

\newcommand{\omk}{\overline{\omega}_k}

\begin{document}

\maketitle

\begin{abstract}
 We consider scalar-input control systems in the vicinity of an equilibrium, at which the linearized systems are not controllable. For finite dimensional control systems, the authors recently classified the possible quadratic behaviors. Quadratic terms introduce coercive drifts in the dynamics, quantified by integer negative Sobolev norms, which are linked to Lie brackets and which prevent smooth small-time local controllability for the full nonlinear system.
 
 In the context of nonlinear parabolic equations, we prove that the same obstructions persist. More importantly, we prove that two new behaviors occur, which are impossible in finite dimension. First, there exists a continuous family of quadratic obstructions quantified by fractional negative Sobolev norms or by weighted variations of them. Second, and more strikingly, small-time local null controllability can sometimes be recovered from the quadratic expansion. We also construct a system for which an infinite number of directions are recovered using a quadratic expansion.
 
 As in the finite dimensional case, the relation between the regularity of the controls and the strength of the possible quadratic obstructions plays a key role in our analysis.
\end{abstract}

\bigskip
\bigskip
\setcounter{tocdepth}{1}
\tableofcontents

\newpage

\section{Introduction and main results}

The goal of this work is to illustrate possible behaviors for parabolic scalar-input control systems, stemming from the analysis of their second-order expansions. Some of these quadratic behaviors are already present in finite dimension (see \cite{BMJDE}, where the authors classified the possible quadratic behaviors for scalar-input control systems in finite dimension, or \cite{FM_CMLS} for a short survey in French by the second author). Others are new and specific to control systems in infinite dimension.

\subsection{Description of the control system}

We present our results in the simple setting of a scalar-input control system governed by a nonlinear heat equation set on the line segment $x \in (0,\pi)$. We consider the following nonlinear control system:
\begin{equation} \label{eq:z}
\left\{
\begin{aligned}
& \partial_t z(t,x) - \partial_{xx} z(t,x) = u(t) \Gamma[z(t)] (x),  \\
& \partial_x z (t,0) =  \partial_x z(t,\pi) = 0, \\
& z(0,x) = z_0(x),
\end{aligned}
\right.
\end{equation}
where $z$ is the state, $u$ is a scalar control, and $\Gamma$ is an appropriate nonlinearity.
We will be interested in the notion of small-time local null controllability: given a small time $T > 0$ and an initial data~$z_0$ sufficiently small, does there exist a small control $u$ such that $z(T) = 0$?

\begin{rk}
 The abstract system \eqref{eq:z} is not intended to model the behavior of a real-world physical system. Even so, we expect that most of the techniques and methods we introduce in the sequel could be applied or extended to other more complex or realistic classes of systems. We chose this abstract system as it lightens the computations and reduces the amount of technicalities in the proof (see also \cref{sec:wsio} for an example of difficulties to be expected for other systems).
\end{rk}

\subsection{Notations and functional settings}

\subsubsection{Functional setting in space}

We consider the Lebesgue space $L^2(0,\pi)$, equipped with its usual scalar product $\langle \cdot, \cdot \rangle$. Let $\Delta_N$ be the Neumann-Laplacian operator on $(0,\pi)$
\begin{equation} \label{eq:neumann-laplacian}
 D(\Delta_N):=\{ f \in H^2(0,\pi) ; \enskip f'(0)=f'(\pi)=0 \}, \quad \Delta_N f = - \partial_{xx} f
\end{equation}
and $(\varphi_k)_{k\in\N}$, the orthonormal basis of $L^2(0,\pi)$ of its eigenfunctions
\begin{equation}
 \varphi_k(x) := 
 \frac{1}{\sqrt{\pi}}
 \left\lbrace \begin{aligned}
  & 1 && \text{ if } k=0, \\
  & \sqrt{2} \cos (k x) && \text{ if } k \in \N^*.
  \end{aligned} \right.
\end{equation}
We introduce the Sobolev space $H^1_N(0,\pi)$ 
\begin{equation}
 H^1_N(0,\pi) := \left\{  f \in L^2(0,\pi) ; \enskip 
\|f\|_{H^{1}_N(0,\pi)}^2 = \sum_{k=0}^\infty | (1+k) \langle f , \varphi_k \rangle |^2 < \infty \right\}
\end{equation}
and its dual space $H^{-1}_N(0,\pi)$, which is equipped with the norm
\begin{equation}
 \|f\|_{H^{-1}_N(0,\pi)}^2 = \left( \sum_{k=0}^\infty \left| (1+k)^{-1} ( f , \varphi_k )_{H^{-1}_N,H^1_N}\right|^2 \right)^{1/2}.
\end{equation}

\subsubsection{Assumptions on the nonlinearity and regularity of solutions} 

Throughout this work, we will assume that there exists $C_\Gamma > 0$ such that the nonlinearity $\Gamma$ appearing in \eqref{eq:z} satisfies
\begin{equation} \label{def:Gamma}
 \begin{gathered}
  \Gamma \in C^2 \left( H^1_N(0,\pi); H^{-1}_N(0,\pi) \right)
  \quad \text{and} \quad
  \\
  \forall z \in H^1_N(0,\pi), 
  \quad \| \Gamma'[z] \|_{\mathcal{L}(H^1_N; H^{-1}_N)}
  + \| \Gamma''[z] \|_{\mathcal{L}(H^1_N \times H^1_N; H^{-1}_N)} \leq C_\Gamma,
 \end{gathered}
\end{equation}
so that, for every control $u \in L^\infty(0,T)$, system \eqref{eq:z} is locally well-posed (see \cref{Prop:WP+source} below) in the space
\begin{equation} \label{Z}
 Z := C^0\left([0,T]; L^2(0,\pi)\right) \cap L^2\left((0,T);H^1_N(0,\pi)\right),
\end{equation}
which we endow with the norm
\begin{equation}
 \|z\|_{Z}^2 := \|z\|_{C^0([0,T]; L^2(0,\pi))}^2 + \|z\|_{L^2((0,T);H^1_N(0,\pi))}^2
\end{equation}
and its solution $z$ has a $C^2$-dependence with respect to the control $u$.
Eventually, we will use as a shorthand notation
\begin{equation}
 \mu := \Gamma[0] \in H^{-1}_N(0,\pi).
\end{equation}

\begin{rk}
 The regularity assumption \eqref{def:Gamma} includes bilinear systems as a particular case. Indeed, for any $\lambda, \mu \in H^{-1}_N(0,\pi)$, the map $\Gamma : z \mapsto \mu + \lambda z$ satisfies \eqref{def:Gamma}. Moreover, even when $\Gamma$ is such an affine map, the control system is already fully nonlinear since the  source is $u \Gamma[z]$. 
\end{rk}



\subsubsection{Functional setting in time}

For $T>0$ and $f \in L^\infty(0,T)$, we consider the iterated primitives of $f$ defined by induction as
\begin{equation} \label{primitive_n}
 f_0 := f, \quad f_{n+1}(t):=\int_0^t f_n(\tau) \dd\tau.
\end{equation}
%
Implicitly, when required, we identify a function $f \in L^\infty(0,T)$ with its extension by zero to the real line, which allows us to consider
\begin{itemize}
\item its non-unitary Fourier transform $\widehat{f}$, defined for $\xi \in \R$ as
\begin{equation} \label{eq:fourier}
 \widehat{f}(\xi) := \int_\R f(t) e^{-i \xi t} \dd t = \int_0^T f(t) e^{-i\xi t} \dd t,
\end{equation}
\item its (negative)  fractional Sobolev norm in $H^{-s}(\R)$, 
for any $s \in (0,1]$,
\begin{equation} \label{def:normeH(-s)}
\|f\|_{H^{-s}(\R)} := \left( \frac1{2\pi} \int_{\R} \frac{|\widehat{f}(\xi)|^2}{(1+\xi^2)^s} \dd\xi\right)^{\frac12}.
\end{equation}
\end{itemize}
%
We will also consider, for any integer $m \in \N$ the usual integer-order Sobolev spaces 
\begin{equation}
 H^{m}(0,T):=\left\{ f \in L^2(0,T) ; \enskip f^{(k)} \in L^2(0,T) \text{ for } k=0,...,m \right\},
\end{equation}
equipped with the norm
\begin{equation}
 \|f\|_{H^m(0,T)} := \left(  \sum_{k=0}^m \int_0^T |f^{(k)}(t)|^2 \dd t \right)^{\frac12},
\end{equation}
and $H^m_0(0,T)$, which is the adherence of $C^\infty_c(0,T)$ for the topology of $\|.\|_{H^m(0,T)}$.
For $a \in [0,\infty)$ the (positive) fractional Sobolev space $H^{a}(0,T)$ defined by interpolation and equipped with the norm $\|.\|_{H^a(0,T)}$.

\subsection{Controllability stemming from the linear order}

We start by studying the linearized system of \eqref{eq:z} around the null equilibrium $(z,u) = (0,0)$:
\begin{equation} \label{eq:z-linear}
\left\{
\begin{aligned}
& \partial_t z(t,x) - \partial_{xx} z(t,x) = u(t) \gmu(x), \\
& \partial_x z (t,0) =  \partial_x z(t,\pi) = 0, \\
& z(0,x) = z_0(x).
\end{aligned}
\right.
\end{equation}
The controllability of systems such as \eqref{eq:z-linear} has been extensively studied. We refer in particular to \cite{MR0335014} and \cite{MR0510972} for the introduction of the \emph{moment method}. The assumption that $\langle \gmu, \varphi_k \rangle \neq 0$ for all $k\in\N$ is obviously necessary for the linearized system to be null controllable (otherwise the component $\langle z(t), \varphi_k\rangle$ of the state would evolve freely). Moreover, in order for the linearized system to be small-time null controllable, one must add the assumption that the sequence $\langle \gmu, \varphi_k\rangle$ does not decay too fast (see \cref{sec:linear-cost}).

\begin{thm} \label{thm:linear}
 Let $\gmu \in H^{-1}_N(0,\pi)$ such that $\langle \gmu, \varphi_k \rangle \neq 0$ for all $k\in \N$ and satisfying the decay assumption \eqref{eq:mu-b}. The linear system \eqref{eq:z-linear} is small-time null controllable with $L^2$ controls.
\end{thm}

When dealing with nonlinear behaviors, especially in infinite dimension, the regularity of the controls plays a crucial role. In fact, and quite surprisingly, the regularity of the controls already plays an important role for control systems in finite dimension (see \cite{BMJDE}). We define the following more precise notions, stressing the regularity imposed on the controls.

\begin{definition}[Small-time local null controllability]
 Let $\Gamma$ be such that \eqref{def:Gamma} holds. For $m \in \N^*$, we say that system \eqref{eq:z} is $H^m$-STLNC (respectively $H^m_0$-STLNC) when, for every $T,\eta > 0$, there exists $\delta > 0$ such that, for every $z_0 \in L^2(0,\pi)$ with $\|z_0\|_{L^2} \leq \delta$, there exists $u \in H^m(0,T)$ (resp. $u \in H^m_0(0,T)$) with $\|u\|_{H^m} \leq \eta$ such that the solution $z \in Z$ to \eqref{eq:z} satisfies $z(T) = 0$.
\end{definition}

\begin{definition}[Smooth small-time local null controllability]
 Let $\Gamma$ be such that \eqref{def:Gamma} holds. We say that system \eqref{eq:z} is Smoothly-STLNC, when it is $H^m_0$-STLNC for every $m \in \N^*$.
\end{definition}

Under appropriate assumptions, the smooth small-time null controllability of the linearized system around the null equilibrium implies that the nonlinear system is Smoothly-STLNC. This was also the case in finite dimension (see \cite[Theorem 1]{BMJDE}). Although the following theorem is quite classical, we include a proof in \cref{sec:sstlc-linear} to highlight that we can indeed construct regular controls.

\begin{thm} \label{thm:sstlc-linear}
 Let $\Gamma$ satisfying \eqref{def:Gamma}, $\langle \gmu, \varphi_k \rangle \neq 0$ for all $k\in \N$ and the decay assumption \eqref{eq:mu-b}. Then, the nonlinear system \eqref{eq:z} is Smoothly-STLNC with a linear cost. 
 
 \medskip \noindent
 More precisely, for every $m \in \N^*$, there exist constants $C_\mathfrak{L}, \delta_\mathfrak{L} > 0$ and a continuous map $\mathfrak{L} : \left\{ z_0 \in L^2(0,\pi); \enskip \|z_0\|_{L^2} \leq \delta_\mathfrak{L} \right\} \to H^m_0(0,T)$ such that, for every $z_0 \in L^2(0,\pi)$ with $\|z_0\|_{L^2} \leq \delta_\mathfrak{L}$, the solution $z \in Z$ to \eqref{eq:z} with a control $u := \mathfrak{L}(z_0)$ satisfies $z(T) = 0$. Moreover, the control and the trajectory satisfy the estimate
 \begin{equation} \label{eq:zu-stlc-cost}
   \|z\|_{Z} + \|u\|_{H^m(0,T)} \leq C_\mathfrak{L} \|z_0\|_{L^2(0,\pi)}.
 \end{equation}
\end{thm}

\begin{rk} \label{rk:sstlc-linear-equiv}
 In fact, the small-time null-controllability of the linearized system is a necessary condition for the Smooth-STLNC of the full nonlinear system \emph{with linear cost} (see \cref{sec:sstlc-linear-equiv}).
\end{rk}

Generally speaking, one can only deduce local controllability results for the nonlinear system from the controllability of the linearized system. However, for some particular nonlinearities, it is sometimes possible to obtain global control results (see e.g.\ \cite{MR1791879} where the authors prove a global control result for a nonlinear parabolic equation despite the presence of a strong nonlinearity which would make the solution blow up in the absence of control).

\subsection{Obstructions caused by quadratic integer drifts}

When the linearized system misses some directions, a natural question is whether a quadratic expansion can help to recover controllability along the lost directions. As a representative situation, we consider the case when the linearized system misses one direction: the first one. This choice also avoids technicalities explained in Section \ref{sec:wsio}. This corresponds to the assumption that 
\begin{equation} \label{lost_direction}
 \langle \mu, \varphi_0 \rangle = 0.
\end{equation}
To study the quadratic behavior of the system along this lost direction, we introduce, for $j \in \N^*$, the sequence $(c_j)_{j\in\N^*} \in \R^{\N^*}$ defined by
\begin{equation} \label{Def:cj}
 c_j := \langle \gmu , \varphi_j \rangle \langle \glambda \varphi_j, \varphi_0 \rangle.
\end{equation}
Similarly as in finite dimension (see \cite[Theorems 2 and 3]{BMJDE}), Lie bracket considerations lead to obstructions related to quadratic coercive drifts, quantified by integer-order negative Sobolev norms. 
This means that the component $\langle z(t),\varphi_0 \rangle$ ineluctably moves in one direction, for instance it increases and therefore cannot reach values $\leqslant \langle z(0),\varphi_0 \rangle$, which prevents controllability.
Moreover, $\langle z(t) - z(0),\varphi_0\rangle $ behaves like the square of an appropriate norm of the control, see Remark \ref{Rk1.6} below.

The particular case of the first obstruction was encountered by Coron in \cite{MR2193655} and Morancey and the first author in \cite{MR3167929}, in the context of a bilinear Schr\"odinger equation. The following theorem, proved in \cref{sec:obs-integer}, shows that any integer-order quadratic obstruction is possible and has consequences for the controllability of the full nonlinear system.

\begin{thm} \label{thm:obs-integer}
 Let $n \in \N^*$, $\Gamma$ satisfying \eqref{def:Gamma} and \eqref{lost_direction}. Assume that there exists $n$ different directions $\varphi_j$ such that $\langle \mu,\varphi_j\rangle \neq 0$ and that the sequence $c_j$ defined in~\eqref{Def:cj} satisfies the assumptions
\begin{align} 
 \label{K_C2n}
 & \sum_{j=1}^\infty j^{4n}|c_j| < + \infty,
 \\
 \label{Ad_2l-1}
 & \sum_{j=1}^\infty j^{2(2\ell-1)} c_j = 0, \quad \text{ for every } \ell \in \{1,...,n-1\},
 \\
 \label{Ad_2n-1}
 a := & \sum_{j=1}^\infty j^{2(2n-1)} c_j \neq 0.
\end{align}
Then the system (\ref{eq:z}) is not $H^{2n+2}$-STLNC.

\bigskip \noindent
More precisely, for every $\varepsilon > 0$, there exists $T^*>0$ such that, for every $T \in (0,T^*)$, there exists $\eta>0$ such that, for every $\delta \in [-1,1]$, for every $u \in H^{2n+2}(0,T)$ with $\|u\|_{H^{2n+2}(0,T)} \leq \eta$, if the solution of \eqref{eq:z} with initial data $z_0= \delta \varphi_0$ satisfies
\begin{equation} \label{eq:return} 
 \forall j \in \N^*, \quad \langle z(T),\varphi_j\rangle = 0,
\end{equation} 
then 
\begin{equation} \label{eq:DRIFT_HN}
 \left| \langle z(T),\varphi_0\rangle - \delta + a (-1)^n \|u_n\|_{L^2(0,T)}^2 \right| 
 \leq \varepsilon \left( |\delta| + \|u_n\|_{L^2(0,T)}^2 \right).
\end{equation}
\end{thm}

\begin{rk} \label{Rk1.6}
 At an heuristic level, the estimate \eqref{eq:DRIFT_HN} corresponds to the fact that, in the asymptotic of small controls in $H^{2n+2}(0,T)$, one has
 \begin{equation} \label{trop_vague}
   \langle z(T),\varphi_0\rangle \approx \langle z_0, \varphi_0 \rangle - a (-1)^n \|u_n\|_{L^2(0,T)}^2.
 \end{equation}
 The approximate equality \eqref{trop_vague} indicates that the quadratic terms induce a drift in the dynamics of the system, which is quantified by the $L^2(0,T)$ norm of $u_n$. In particular, initial states for which $\langle z_0,\varphi_0 \rangle$ has the same sign as $-a(-1)^n$ cannot be driven to zero. 
\end{rk}

\begin{rk}
Under assumption (\ref{K_C2n}), the series considered in (\ref{Ad_2l-1}) and (\ref{Ad_2n-1}) converge.
Under appropriate regularity assumptions on $\Gamma$, these two relations may be rewritten in term of the Lie brackets 
\begin{align}
\langle [\text{ad}_{f_0}^{2\ell-1}(f_1), f_1](0) , \varphi_0 \rangle & =0, \quad \text{ for every } \ell \in \{1,...,n-1\},
\\
\langle [\text{ad}_{f_0}^{2n-1}(f_1),f_1](0) , \varphi_0 \rangle & \neq 0,
\end{align}
where $f_0 := \Delta_N$ and $f_1 := \Gamma$. These Lie bracket conditions are exactly those that appear for finite dimensional systems $x'=f_0(x)+uf_1(x)$ in \cite[Theorems 2 and 3]{BMJDE} .

Here, the notation $\text{ad}_{f_0}^j$ refers to the usual definition of iterated Lie brackets of vector fields and is used formally in infinite dimension.
For smooth vector fields $X, Y \in C^\infty(\R^d,\R^d)$, $[X,Y]$ is the smooth vector field on $\R^d$ defined by 
$[X,Y](x)=Y'(x)X(x)-X'(x)Y(x)$ for every $x\in\R^d$ and we define by induction on $j \in\N$, $\text{ad}_{X}^0(Y):=Y$ and $\text{ad}_{X}^{j+1}(Y):=[X,\text{ad}_X^j(Y)]$.

These formal computations become rigorous for instance when $\Gamma[z]=\mu+\Gamma'[0]z$, $\mu \in D(\Delta_N^{2n-1})$, and $\Gamma'[0]$ maps $D(\Delta_N^j)$ into itself for any $j\in\{1,\dots,n\}$.
Indeed, then for any $j \in \{1,\dots,n\}$ and $z\in D(\Delta_N^j)$, the following equality holds in $L^2(0,\pi)$: 
\begin{equation}
\text{ad}_{\Delta_N}^j (\Gamma)(z)=\text{ad}_{\Delta_N}^j \left( \Gamma'(0) \right) z +  (-1)^j \Delta_N^j \mu,
\end{equation}
where the first bracket in the right hand side is a commutator of operators.
Thus
\begin{equation} \label{crohet_interm_1}
[\text{ad}_{\Delta_N}^{2\ell-1}(\Gamma), \Gamma ](0) =  \Gamma'[0]  (-1)^{2\ell-1} \Delta_N^{2\ell-1} \mu - \text{ad}_{\Delta_N}^{2\ell-1}\left( \Gamma'(0) \right) \mu.
\end{equation}
Taking into account that $\Delta_N \varphi_0=0$, we obtain
$\langle \text{ad}_{\Delta_N}^{2\ell-1}\left( \Gamma'(0) \right) \mu , \varphi_0 \rangle = \langle  \Gamma'(0) \Delta_N^{2\ell-1}\mu,\varphi_0\rangle $
and we deduce from (\ref{crohet_interm_1}) that
$\langle [\text{ad}_{\Delta_N}^{2\ell-1}(\Gamma), \Gamma ](0) ,\varphi_0\rangle = - 2 \langle \Gamma'[0]  \Delta_N^{2\ell-1} \mu,\varphi_0\rangle$.
Finally, for any $\ell \in \{1,\dots,n\}$, one has
\begin{equation}
 \begin{split}
\sum_{j=1}^\infty j^{2(2\ell-1)} c_j 
& = \sum_{j=1}^\infty j^{2(2\ell-1)} \langle \gmu , \varphi_j \rangle \langle \glambda \varphi_j, \varphi_0 \rangle 
=  \sum_{j=1}^\infty \langle \Delta_N^{2\ell-1} \gmu , \varphi_j \rangle \langle  \varphi_j, \glambda^* \varphi_0 \rangle
\\ & = \langle \Delta_N^{2\ell-1} \gmu , \glambda^* \varphi_0 \rangle 
= \langle \glambda \Delta_N^{2\ell-1} \gmu ,  \varphi_0 \rangle 
 = - \frac{1}{2} \langle [\text{ad}_{\Delta_N}^{2\ell-1}(\Gamma), \Gamma ](0) , \varphi_0 \rangle.
 \end{split}
\end{equation}
\end{rk}

\begin{rk}
 For control-affine systems in finite dimension, we proved in \cite[Theorem 3]{BMJDE} that the optimal norm for the smallness assumption on $u$ is the $W^{2n-3,\infty}$ one. 
 This norm is optimal in the following sense:
 \begin{itemize}
 \item there exists finite dimensional systems $x'=f(x,u)$ that are not $W^{2n-3,\infty}$-STLC because of such a drift, 
 but for which $W^{2n-4,\infty}$-STLC holds: 
 STLC is possible with controls small in $W^{2n-4,\infty}(0,T)$ but large in $W^{2n-3,\infty}(0,T)$, typically oscillating controls,
 \item in such cases, the cubic (or higher order) term is responsible for the controllability: it dominates the quadratic term in this oscillation regime.
 \end{itemize}
 The $H^{2n+2}$-norm used in Theorem \ref{thm:obs-integer}, for the nonlinear heat equation, is not the optimal one, but allows a lighter exposition.
 Here, the optimal norm is an open problem.
 To solve it, one would need a sharp quantification of the cubic remainder in $\langle z(t),\varphi_0\rangle$. 
\end{rk}

\begin{rk}
 Assumption \eqref{K_C2n} could probably also be softened. The convergence of this sum is indeed not necessary to define the drift amplitude $a$ in condition \eqref{Ad_2n-1}. However, this regularity avoids technical complications in the proofs when estimating the remainders. Hence, we will stick with it since our main goal is to highlight phenomenons and not to provide an optimal general framework.
\end{rk}

\subsection{Obstructions caused by quadratic fractional drifts}

In the context of parabolic equations, which are control systems in infinite dimension, a whole new continuous family of drifts can occur, quantified by fractional-order negative Sobolev norms. A fractional drift quantified by the $H^{-5/4}$ norm of the control had already been observed by the second author for a Burgers equation (see \cite{2015arXiv151104995M}). The following theorem, proved in \cref{sec:obs-fractional}, is the first main result of our work and proves that any negative fractional drift is possible.

\begin{thm} \label{thm:obs-fractional}
 Let $n \in \N$, $s \in (0,1)$, $a \in \R^*$, $\alpha > 4s-1$ and $\Gamma$ be such that \eqref{def:Gamma} and \eqref{lost_direction} hold. Assume that the coefficients $c_j$ defined in \eqref{Def:cj} satisfy
 \begin{align} \label{cj_asympt}
  & c_j = \frac{a}{j^{4n-1+4s}} + O\left( \frac{1}{j^{4n+\alpha}} \right),
  \\
  \label{Ad_2l-1_frac}
  & \sum_{j=1}^\infty j^{2(2\ell-1)} c_j = 0, 
  \quad \text{ for every } \ell \in \{1,...,n\}.
 \end{align}
 Then the system (\ref{eq:z}) is not $H^{2n+2s+\frac{3}{2}}$-STLNC.

\bigskip \noindent
More precisely, there exists a constant $\gamma(s) > 0$ (see \eqref{gamma}) such that, for every $\varepsilon > 0$, there exists $T^*>0$ such that, for every $T \in (0,T^*)$, there exists $\eta>0$ such that, for every $\delta \in [-1,1]$, for every $u \in H^{2n+2s+\frac32}(0,T)$ with $\|u\|_{H^{2n+2s+\frac{3}{2}}(0,T)} \leq \eta$, if the solution of \eqref{eq:z} with initial data $z_0= \delta \varphi_0$ satisfies the condition \eqref{eq:return}, then 
\begin{equation} \label{eq:DRIFT_HS}
 \left| \langle z(T),\varphi_0\rangle - \delta - a \gamma(s) (-1)^n \|u_n\|_{H^{-s}(\R)}^2 \right| 
 \leq \varepsilon \left( |\delta| + \|u_n\|_{H^{-s}(\R)}^2 \right).
\end{equation}
\end{thm}

\begin{rk}
 As in the integer-order obstructions, the following comments can be made.
 \begin{itemize}
 
 \item At an heuristic level, the estimate \eqref{eq:DRIFT_HS} corresponds to the fact that, in the asymptotic of small controls in $H^{2n+2s+\frac{3}{2}}(0,T)$, one has
 \begin{equation} \label{trop_vague2}
   \langle z(T),\varphi_0\rangle \approx \langle z_0, \varphi_0 \rangle + a \gamma(s) (-1)^n \|u_n\|_{H^{-s}(\R)}^2.
 \end{equation}

 \item The conditions \eqref{Ad_2l-1_frac} can be interpreted in terms of Lie brackets between $\Delta_N$ and $\Gamma$.
 
 \item The smallness in $H^{2n+2s+\frac{3}{2}}$ of the controls is not the optimal assumption.
 
 \end{itemize}
\end{rk}


\begin{rk}
 Despite the resemblance between the integer-order and the fractional-order statements, we stress that the nature of the underlying cause might be different. Indeed, the integer-order obstructions occur when the weighted sum \eqref{Ad_2n-1} of the coefficients $c_j$ is non-zero, whereas the fractional-order obstructions depend only on the asymptotic behavior \eqref{cj_asympt} of the coefficients.
\end{rk}

\begin{rk} \label{rk:weights}
 We state \cref{thm:obs-fractional} with drifts quantified by fractional negative Sobolev norms to keep the statement easily understandable. However, as claimed in the abstract, depending on the asymptotic behavior of the sequence $c_j$, the drifts can be quantified by essentially arbitrary weighted fractional negative Sobolev norms. We refer to \cref{sec:weights} for examples and more precise statements. 
\end{rk}

\subsection{Controllability stemming from the quadratic order}

Finally, and even more strikingly, for parabolic equations, we can sometimes recover small-time local null controllability from the quadratic expansion.  This fact is most surprising, as it is never possible in finite dimension. Indeed, for finite dimensional control systems, if the linearized system misses one direction, then small-time local controllability can only be recovered thanks to cubic -- or higher order -- terms (thus, with a cubic or more control cost).

Up to our knowledge, the following result is the first example in which small-time controllability is restored at the quadratic order for a scalar-input system. In all previously known situations where the linearized system misses at least one direction and controllability is restored, either:

\begin{itemize}
 
 \item controllability is restored in small time but by means of a \emph{cubic expansion} with a vanishing quadratic term, 
 (see e.g.\ \cite{MR2060480} where the authors prove small-time local null controllability for a Korteweg-de-Vries system with a critical length thanks to cubic terms),

 \item controllability is restored by means of a quadratic expansion but only 
 \emph{in large time} (see e.g. \cite{MR2338431,MR2504039} where the authors obtain controllability in large time for Korteweg-de-Vries systems with critical lengths and
 \cite{MR2375753, MR2200740, MR3167929} where the authors obtain controllability of bilinear Schrödinger equation),

\item controllability is restored in \emph{large time} by means of 
the return method and possibly other technics
 (see \cite{MR1932962} about the Saint-Venant equation, 
and \cite{MR2144647, MR2200740} about bilinear Schrödinger equations
where the large time is due to quasi-static transformations,
and \cite{MR3208452} where a large time is needed to construct  the reference trajectory of the return method),

 \item controllability is restored \emph{with non-scalar controls} (see e.g.\ \cite{MR1380673} where the author prove small-time controllability for the Euler equation using boundary controls, which corresponds to an infinite number of scalar controls).

\end{itemize}
The following result, proved in \cref{sec:sstlc-quad}, is the second main result of our work.

\begin{thm} \label{thm:sstlc-quad}
 There exists a nonlinearity $\Gamma$ satisfying \eqref{def:Gamma} such that $\langle \gmu, \varphi_0 \rangle = 0$ and the nonlinear system~\eqref{eq:z} is smoothly small-time locally null controllable with \emph{quadratic cost}. 
 More precisely, for each $T > 0$ and $m \in \N^*$, there exists $C_{T,m} > 0$ such that, for each $z_0 \in L^2(0,\pi)$, there exists a control $u \in H^m_0(0,T)$ and a solution $z \in Z$ to \eqref{eq:z}, such that $z(T) = 0$ and 
 \begin{equation} \label{eq:quad.cost.N=1}
 \|u\|_{H^m} \leq C_{T,m} \Big( |\langle z_0,\varphi_0\rangle|^{\frac12} 
 + \|z_0-\langle z_0,\varphi_0\rangle \varphi_0\|_{L^2} \Big).
 \end{equation}
\end{thm}

\begin{rk}[Local vs. global]
 Although we are mostly interested in local controllability properties, \cref{thm:sstlc-quad} is stated as a global controllability result because the system we construct is homogeneous with respect to dilations, so that local and global notions are equivalent.
\end{rk}

\begin{rk}[Quadratic cost]
 \label{rk:quad.cost}
 The size estimate \eqref{eq:quad.cost.N=1} for the control is reminiscent of the fact that the controllability of the first mode $\langle z,\varphi_0\rangle$
 stems from the quadratic order, whereas the controllability of the other modes stems from the linear order (see \cref{sec:quad.cost} for more details).
 
 For small initial states, this cost estimate highlights that the constructed control is "more expansive" than a linear control (with respect to the size of $z_0$), but "less expansive" than a control strategy relying on a cubic (or higher order) expansion.
 
 In finite dimension, there is no system for which the linearized system misses a direction and for which small-time local controllability is recovered with quadratic cost (we refer to \cite{BMJDE} for more precise statements).
\end{rk}

\subsection{Recovering an infinite number of lost directions}

Eventually, we prove that not only a finite number of directions can be recovered at the quadratic order, but even an infinite number of lost directions. Our third main result is the following theorem, proved in \cref{sec:infini}, which,
up to our knowledge, is the first example of recovering an infinite number of lost directions thanks to a power series expansion method.

\begin{thm}
 \label{thm:infini}
 Let $s \in (0,\frac{1}{2})$. There exists a nonlinearity $\Gamma_s : H^1_N(0,\pi) \to H^{-1}_N(0,\pi)$, satisfying the regularity assumptions~\eqref{def:Gamma} and, for every $k\in \N$,
 \begin{equation}
  \label{lost}
  \langle \Gamma_s[0], \varphi_{2k+1} \rangle = 0,
 \end{equation}
 and for which the nonlinear system \eqref{eq:z} is null controllable in small-time with quadratic cost. 
 
 \medskip
 
 More precisely, for each $T > 0$, there exists $C_T > 0$, such that, for each $z_0 \in H^1_N(0,\pi)$, there exists a control $u \in H^{-s}(0,T)$ and a solution $z \in Z$ to \eqref{eq:z}, such that $z(T) = 0$ and
 \begin{equation}
  \label{eq:quad.cost}
  \| u \|_{H^{-s}(0,T)}
  \leq 
  \|(\langle z_0,\varphi_{2k+1}\rangle)_{k\in\N}\|_{\ell^1}^{\frac12}
  + C_T \| (\langle z_0,\varphi_{2k}\rangle)_{k\in\N} \|_{\ell^2}
  .
 \end{equation}
\end{thm}

\begin{rk}[Lost directions]
 Conditions \eqref{lost} imply that an infinite number of directions are "lost" by the linearized system \eqref{eq:z-linear} and "recovered" through the nonlinearity.
\end{rk}

\begin{rk}[Control cost]
A surprising feature of the cost estimate \eqref{eq:quad.cost} is that the cost of control for the odd modes does not blow up as $T \to 0$. This is reminiscent of the fact that we construct a system which behaves really nicely with respect to controllability with $H^{-s}$ controls. This surprising feature is discussed in \cref{sec:cost1}.
\end{rk}

\begin{rk}[Initial state regularity]
 Our result assumes that the initial state is in $H^1_N(0,\pi)$. Of course, if the initial state is only in $L^2(0,\pi)$, one can use a null control on a small time interval in order to let the smoothing effect of the heat equation take place. The main consequence is that the $\ell^1$ norm of the odd coefficients of the regularized initial data scales like $T^{-\frac14}$, which deteriorates the constant control cost we obtained for slightly smoother initial data.
\end{rk}

\begin{rk}[Control regularity]
 The possibility to restore small-time null controllability using more regular controls, for the nonlinearity $\Gamma_s$ we construct, is unlikely. The existence of a scalar-input nonlinear system for which smooth small-time controllability could be restored with quadratic cost, although an infinite number of directions are lost at the linear order is an open question. 
\end{rk}

\subsection{Examples}

The goal of this section is to propose examples of nonlinearities $\Gamma$ to which \cref{thm:obs-integer} and \cref{thm:obs-fractional} apply.
In particular, we want to show that every situation can happen already with affine nonlinearities $\Gamma[z]=\mu+\lambda z$, where $\mu, \lambda \in L^2(0,\pi)$ are functions, visualize the different assumptions in this case and understand the genericity of the different situations. For such affine nonlinearities, definition \eqref{Def:cj} simplifies to 
\begin{equation}
\forall j \in \mathbb{N}, \quad c_j=\langle \mu , \varphi_j \rangle \langle \varphi_j , \lambda \rangle.
\end{equation}

We first focus on the first integer obstruction $n=1$ in Theorem \ref{thm:obs-integer}.
Among the functions $(\lambda,\mu)$ such that the series $\sum j^4 |c_j|^2$ converges, then, generically, the sum $\sum j^2 c_j$ does not vanish, 
and we observe a drift quantified by $\|u_1\|_{L^2(0,T)}^2$.
This is for instance the case in the following example.

\paragraph{Example 1.} We consider $\mu(x)=\rho_1(x)$ and $\lambda(x)=\rho_2(x)$ where
\begin{equation} \label{def:rho12}
\rho_1(x)=x^2-\frac{\pi^2}{3} \quad \text{ and } \quad \rho_2(x)=\frac{x^4}{4}-\pi\frac{x^3}{3}+\frac{\pi^4}{30}.
\end{equation}
Note that $\rho_2 \in D(\Delta_N)$, whereas $\rho_1 \notin D(\Delta_N)$, thus $\langle \rho_2,\varphi_j\rangle$ are expected to decrease faster than $\langle \rho_1,\varphi_j\rangle$.
Indeed, easy computations show that
\begin{equation} \label{Fourier_rho1}
\langle \rho_1,\varphi_0\rangle=0  \text{ and for every } j\in\mathbb{N}^*, \quad \langle \rho_1 , \varphi_j\rangle=2\sqrt{2\pi}\frac{(-1)^j}{j^2},
\end{equation}
\begin{equation} \label{Fourier_rho2}
\langle \rho_2 , \varphi_0\rangle=0 \text{ and for every } j\in\mathbb{N}^*, \quad \langle \rho_2 , \varphi_j\rangle=-2\sqrt{2\pi}\frac{[2(-1)^j+1]}{j^4}.
\end{equation}
Therefore (\ref{lost_direction}) holds because $\langle \mu,\varphi_0\rangle=\langle\rho_1,\varphi_0\rangle=0$ 
and (\ref{K_C2n}) holds with $n=1$ because $c_j$ behaves asymptotically like $\frac{1}{j^6}$. 
Moreover, (\ref{Ad_2n-1}) holds with $n=1$ because
\begin{equation}\sum_{j=1}^{\infty} j^2 c_j = -(2\sqrt{2\pi})^2 \sum_{j=1}^{\infty} j^2 \frac{(-1)^j}{j^2} \frac{[2(-1)^j+1]}{j^4} < 0.
\end{equation}
Thus Theorem \ref{thm:obs-integer} applies: there is a drift quantified by $\| u_1 \|_{L^2(0,T)}^2$.

\bigskip

By modifying a bit this example, we may easily construct another example for which (\ref{Ad_2n-1}) does not hold anymore for $n=1$ and the serie involved in (\ref{Ad_2n-1}) for $n=2$ is not convergent.
Then, we are in between of the first integer obstruction ($n=1$) and the second integer obstruction ($n=2$), in the domain of fractional obstructions treated by Theorem  \ref{thm:obs-fractional}.

\paragraph{Example 2.} We consider $\mu(x)=\rho_1(x)$ and $\lambda(x)=\rho_2(x)-\alpha \varphi_1(x)$ where 
\begin{equation}\alpha = 2\sqrt{2\pi} \sum_{j=1}^{\infty}  \frac{(-1)^j [2(-1)^j+1]}{j^4}.
\end{equation}
Then (\ref{lost_direction}) holds because $\langle \mu,\varphi_0\rangle=\langle\rho_1,\varphi_0\rangle=0$,
(\ref{Ad_2l-1_frac}) holds with $n=1$ by choice of $\alpha$.
Thanks to (\ref{Fourier_rho1}) and (\ref{Fourier_rho2}), the assumption (\ref{cj_asympt}) holds with $n=1$ and $s=\frac{3}{4}$.
Thus Theorem \ref{thm:obs-fractional} applies: there is a drift quantified by $\|u_1\|_{H^{-3/4}(\mathbb{R})}^2$.

\bigskip

In order to observe the second integer obstruction, we need smoother functions $\lambda, \mu$ so that the series involved in (\ref{Ad_2n-1}) with $n=2$ does converge.
This is the case in the following example.

\paragraph{Example 3.} We consider $\mu=\rho_2$ and $\lambda=\rho_3+\beta\varphi_1+\gamma\varphi_2$ where 
$\rho_3$ solves $\rho_3''=\rho_2$ on $(0,\pi)$, $\rho_3'(0)=\rho_3'(\pi)=0$ and $\beta,\gamma\in\mathbb{R}$.
Easy computations show that
\begin{equation}
\forall j \in\N^*, \quad \langle \rho_3,\varphi_j\rangle=-\frac{1}{j^2} \langle \rho_2,\varphi_j\rangle=2\sqrt{2\pi}\frac{[2(-1)^j+1]}{j^6}.
\end{equation}
Then (\ref{lost_direction}) holds because $\langle \mu,\varphi_0\rangle=\langle\rho_2,\varphi_0\rangle=0$,
and (\ref{K_C2n}) holds with $n=2$ because $c_j$ behaves asymptotically like $\frac{1}{j^{10}}$. 
For appropriate choices of $(\beta,\gamma)\in\R^2$, namely
\begin{equation}
-2\sqrt{2\pi} \sum_{j=1}^\infty \frac{[2(-1)^j+1]^2}{j^8} + \beta - \frac{3}{4} \gamma = 0
\quad \text{ and } \quad
-2\sqrt{2\pi} \sum_{j=1}^\infty \frac{[2(-1)^j+1]^2}{j^4} + \beta - 48 \gamma  \neq 0
\end{equation}
then (\ref{Ad_2l-1}) and (\ref{Ad_2n-1}) hold with $n=2$, thus Theorem \ref{thm:obs-integer} applies: there is a drift quantified by $\| u_2 \|_{L^2(0,T)}^2$.
For a different choice of $(\beta,\gamma)$, namely
\begin{equation}
-2\sqrt{2\pi} \sum_{j=1}^\infty \frac{[2(-1)^j+1]^2}{j^8} + \beta - \frac{3}{4} \gamma = 0
\quad \text{ and } \quad
-2\sqrt{2\pi} \sum_{j=1}^\infty \frac{[2(-1)^j+1]^2}{j^4} + \beta - 48 \gamma  = 0
\end{equation}
Theorem \ref{thm:obs-fractional} applies: there is a drift quantified by $\|u_2\|_{H^{-3/4}(\mathbb{R})}^2$.

\bigskip

Iterating this construction, we may construct functions $\lambda, \mu$, relying on the function $\rho_3$ such that Theorem \ref{thm:obs-fractional} applies with $n=2$ and $s=\frac{1}{4}$. 
Then there is a drift quantified by $\|u_3\|_{H^{-1/4}(\mathbb{R})}^2$.

\section{Smooth controllability stemming from the linear order}
\label{sec:sstlc-linear}

The goal of this section is to prove \cref{thm:sstlc-linear}. The idea that small-time controllability of the linearized system implies controllability of the nonlinear system is quite classical. However, there are two difficulties here. First, we are seeking a null controllability result and we only assumed the controllability of the linearized system at the null equilibrium (and not around any state near the equilibrium). This prevents us from using classical fixed-point methods and requires a specific powerful method. We will use the \emph{source term method} introduced by Liu, Takahashi and Tucsnak in \cite{MR3023058} in the context of a fluid-structure system. Second, we wish to build regular controls, even for the nonlinear system. This will require that we adapt accordingly the source term method.

\subsection{Classical well-posedness results}

We recall, for the sake of completeness, usual well-posedness results for the class of parabolic equations we are looking at, under the regularity assumption \eqref{def:Gamma} on $\Gamma$.

\begin{lemma} \label{thm:wp-linear}
 Let $T > 0$. There exists $C_T > 0$, which is a non-decreasing function of $T$, such that, for any $f \in L^2((0,T);H^{-1}_N(0,\pi))$ and any $z_0 \in L^2(0,\pi)$, there exists a unique solution $z \in Z$ to
 \begin{equation} \label{eq:z-f}
 \left\{
 \begin{aligned}
 & \partial_t z(t,x) - \partial_{xx} z(t,x) = f(t,x), \\
 & z_x(t,0) = z_x(t,\pi) = 0, \\
 & z(0,x) = z_0(x).
 \end{aligned}
 \right.
 \end{equation}
 Moreover, it satisfies the estimate
 \begin{equation} \label{eq:wp-linear}
 \|z\|_Z
 \leq  
 C_T \left(\|z_0\|_{L^2} + \|f\|_{L^2((0,T); H^{-1}_N(0,\pi))}\right).
 \end{equation}
\end{lemma}

\begin{proof}
 The proof of this statement is classical. We refer for example to \cite{MR710486}. The fact that $C_T$ depends on $T$ may seem unusual at first glance. However, for the Neumann-Laplacian~\eqref{eq:neumann-laplacian}, the eigenvector $\varphi_0$ is associated with the eigenvalue $0$. Thus, the contribution of this first mode to the $L^2((0,T);H^1_N(0,\pi))$ norm of $z$ is not bounded as $T \to +\infty$.
\end{proof}

\begin{lemma} \label{Prop:WP+source}
 Let $\Gamma$ satisfying (\ref{def:Gamma}) and $T > 0$. 
 There exist constants $C,\eta > 0$ such that, 
 for every $z_0 \in L^2(0,\pi)$ and 
 $u \in L^\infty(0,T)$ with $\|u\|_{L^\infty} \leq \eta$, there exists a unique solution $z \in Z$ to
 \begin{equation} \label{eq:z+source}
 \left\{
 \begin{aligned}
 & \partial_t z(t,x) - \partial_{xx} z(t,x) = u(t) \Gamma[z(t)](x), \\
 & \partial_x z (t,0) =  \partial_x z(t,\pi) = 0, \\
 & z(0,x) = z_0(x).
 \end{aligned}
 \right.
 \end{equation}
 Moreover, this solution satisfies
 \begin{equation} \label{wp-estimate}
   \|z\|_{Z}
   \leq C \left( \|z_0\|_{L^2}  + \|u\|_{L^\infty(0,T)} \right).
 \end{equation}
\end{lemma}

\begin{proof}
 Let $T > 0$, $\Gamma$ satisfying (\ref{def:Gamma}), 
 $z_0 \in L^2(0,\pi)$ and $u \in L^\infty(0,T)$. We construct a map $\mathcal{F} : Z \to Z$ by associating, to any $z \in Z$, the value $\mathcal{F}(z) := h$, where $h$ is the solution to
 \begin{equation}
 \left\{
 \begin{aligned}
 & \partial_t h(t,x) - \partial_{xx} h(t,x) = 
 u(t) \Gamma[z(t)](x), \\
 & \partial_x h (t,0) =  \partial_x h(t,\pi) = 0, \\
 & h(0,x) = z_0(x).
 \end{aligned}
 \right.
 \end{equation}
 From the regularity assumption \eqref{def:Gamma} on $\Gamma$, for every $z_1, z_2 \in H^1_N(0,\pi)$, 
 \begin{equation} 
  \| \Gamma[z_1] - \Gamma[z_2] \|_{H^{-1}_N(0,\pi)}
  \leq C_\Gamma \|z_1-z_2\|_{H^1_N(0,\pi)}. 
 \end{equation}
 Then, by \cref{thm:wp-linear}, the map $\mathcal{F}$ is well-defined and moreover, for every $z_1, z_2 \in Z$,
 \begin{equation}
  \| \mathcal{F}(z_1)-\mathcal{F}(z_2) \|_{Z} 
  \leq 
  C_T C_\Gamma \|u\|_{L^\infty(0,T)} \|z_1-z_2\|_{Z}.
 \end{equation}
 For $\|u\|_{L^\infty(0,T)} \leq \eta := (2C_TC_\Gamma)^{-1}$, this proves that $\mathcal{F}$ is a contraction mapping on $Z$. Thanks to the Banach fixed point theorem, it admits a unique fixed point. Using once more \cref{thm:wp-linear}, we obtain estimate \eqref{wp-estimate} with $C := 2C_T (\|\mu\|_{H^{-1}_N(0,\pi)}+1)$.
\end{proof}



\subsection{Smooth resolution of moment problems}

To study the linear problem and obtain estimates useful for the nonlinear problem, we will use the well-known \emph{moment method} introduced by Fattorini and Russel (see e.g. the seminal works \cite{MR0335014,MR0510972}). We start with the following results, concerning the solvability of moment problems with smooth controls. 

\begin{lemma}[Existence of biorthogonal families] \label{thm:bbgo}
 Let $m \in \N$. There exist $C_1, T_1 > 0$, such that, for any $T \in (0,T_1]$, there exists a family $(\psi^T_{k,j})_{k \in \N, 0 \leq j \leq m}$ of functions in $L^2(-\frac{T}{2},\frac{T}{2})$, such that, for any $k, k' \in \N$ and $0 \leq j, j' \leq m$,
 \begin{equation} \label{eq:biorth}
  \int_{-T/2}^{T/2} t^{j}\, e^{-(1+k^2)t}\, \psi^T_{k',j'}(t) \dd t = \delta_{j,j'} \delta_{k,k'},
 \end{equation}
 which moreover satisfies, for $k \in \N$ and $0 \leq j \leq m$,
 \begin{equation} \label{eq:psi-size}
 \| \psi^T_{k,j} \|_{L^2\left(-\frac{T}{2},\frac{T}{2}\right)} 
 \leq 
 C_1 e^{k C_1 + C_1/T}.
 \end{equation}
\end{lemma}

\begin{proof}
 This result is stated more generally in \cite[Theorem 1.5]{MR3262588} for any sequence of eigenvalues satisfying a set of appropriate assumptions. Here, the sequence of eigenvalues given by $\Lambda_k := 1+k^2$ for $k \in \N$ satisfies all the required assumptions. 
\end{proof}

\cref{thm:bbgo} was introduced to control systems of parabolic equations, which required being able to solve moment problems with polynomial terms. Here, we consider a single scalar parabolic equation, but we wish to build regular controls and estimate their size.

\begin{proposition}[Solvability of moment problems] \label{thm:moments}
 Let $m \in \N$ and $\eta_1 > \frac{1}{2}$. There exist $M_1, T_1 > 0$, such that the following property holds. Let $T \in (0,T_1]$ and define the normed vector space
 \begin{equation} \label{eq:DT-def}
  D_T := \left\{ d = (d_k)_{k \in \N} \in \R^\N, \enskip \|d\|_{D_T} := \sup_{k \in \N} |d_k| e^{\eta_1 T k^2} < + \infty \right\}. 
 \end{equation}
 There exists a continuous linear map $\mathfrak{L}^T_1 : D_T \to H^m_0(0,T)$ such that, for any sequence $d = (d_k)_{k \in \N} \in D_T$, the control $u := \mathfrak{L}^T_1(d)$ satisfies, for all $k \in \N$, the moment condition
 \begin{equation} \label{eq:u-dk}
  \int_0^T u(t) e^{-k^2(T-t)} \dd t = d_k
 \end{equation}
 and the size estimate
 \begin{equation} \label{eq:u-moments-size}
  \|u\|_{H^m(0,T)} \leq M_1 e^{M_1/T} \|d\|_{D_T}.
 \end{equation}
\end{proposition}

\begin{proof}
 Let $C_1, T_1 > 0$ be given by \cref{thm:bbgo}. Let $m \in \N$. First, there exists a constant $S_1 > 0$ such that, for any $T \in (0,T_1]$ and any $u \in H^m_0(0,T)$, one has
 \begin{equation} \label{eq:uhm-uml2}
  \|u\|_{H^m(0,T)} \leq S_1 \|u^{(m)}\|_{L^2(0,T)}.
 \end{equation}
 Let $\eta_1 > \frac{1}{2}$. We define $\alpha_1 := \frac{1}{2} (\eta_1 - \frac{1}{2})$ and 
 \begin{equation} \label{eq:m1}
  M_1 := \max \left\{ 
  1 + C_1 + \frac{(C_1+2m)^2}{4 \alpha_1}, \enskip 
  C_1 S_1 \left(1 + m! + \frac{T_1}{1 - e^{-\alpha_1 T_1}}\right)  \right\}.
 \end{equation}
 Let $T \in (0,T_1]$ and $d \in D_T$. To build a control $u := \mathfrak{L}^T_1(d) \in H^m_0(0,T)$, we look for its $m$-th derivative $u^{(m)}$ under the form
 \begin{equation} \label{eq:um-v}
  u^{(m)}(t) = v\left(\frac{T}{2}-t\right) e^{-\left(\frac{T}{2}-t\right) },
 \end{equation}
 where $v \in L^2(-\frac{T}{2},\frac{T}{2})$. Using \eqref{eq:um-v} and iterated integration by parts, we obtain that $u \in H^m_0(0,T)$ if and only if, for $0 \leq j \leq m-1$,
 \begin{equation} \label{eq:v-tauj}
  \int_{-T/2}^{T/2} \tau^j v(\tau) e^{-\tau} \dd \tau = 0.
 \end{equation}
 Integrating by parts and using (\ref{eq:v-tauj}),
 we obtain that \eqref{eq:u-dk} for $k = 0$ is equivalent to 
 \begin{equation} \label{eq:v-taum}
  \int_{-T/2}^{T/2} \tau^m v(\tau) e^{-\tau} \dd \tau = m!\, d_0
 \end{equation}
 and that \eqref{eq:u-dk} for $k \geq 1$ is equivalent to
 \begin{equation} \label{eq:v-expk}
 \int_{-T/2}^{T/2} v(\tau) e^{-(1+k^2) \tau} \dd \tau = 
 (-1)^m k^{2m} e^{k^2 T /2} d_k.
 \end{equation}
 We set
 \begin{equation} \label{eq:v-sum}
  v(t) := m! \, d_0 \psi^T_{0,m}(t) + (-1)^{m} \sum_{k \geq 1} k^{2m} e^{k^2 T/2} d_k \psi^T_{k,0}(t).
 \end{equation}
 Thanks to the size estimate of the biorthogonal family \eqref{eq:psi-size} and the decay of $d_k$ \eqref{eq:DT-def}, the above serie converges in $L^2(-T/2,T/2)$ and
 \begin{equation} \label{eq:v-size1}
  \|v\|_{L^2} \leq m!\, |d_0| C_1 e^{C_1/T}
  + C_1 \sum_{k\geq 1} |d_k| k^{2m} e^{k^2 T/2} e^{kC_1 + C_1/T}.
 \end{equation}
  Thanks to the biorthogonality condition \eqref{eq:biorth}, the relations \eqref{eq:v-tauj}, \eqref{eq:v-taum} and \eqref{eq:v-expk} are satisfied. Thus $u \in H^m_0(0,T)$ and solves the moment problem \eqref{eq:u-dk} for $k \in \N$. 
  Since $d \in D_T$, using the definition of $\|d\|_{D_T}$ in \eqref{eq:DT-def} and the relations $k^{2m} \leq e^{2mk}$ and $2 \alpha_1 = \eta_1 - \frac{1}{2}$, we obtain from \eqref{eq:v-size1} that
 \begin{equation} \label{eq:v-tech1}
  \|v\|_{L^2} \leq C_1 e^{C_1/T} \left( m!\, +
  \sum_{k\geq 1} e^{(C_1+2m)k} e^{-2 \alpha_1 k^2 T} \right) \|d\|_{D_T}.
 \end{equation}
 For any $k \in \N$,
 \begin{equation} \label{eq:v-tech2}
  \exp \left( (C_1+2m) k \right) 
  \leq 
  \exp \left( \alpha_1 k^2 T  
  + \frac{(C_1+2m)^2}{4 T \alpha_1}
   \right).
 \end{equation}
 Moreover, for $T \in (0,T_1]$,
 \begin{equation} \label{eq:v-tech3}
  \sum_{k \geq 1} e^{-\alpha_1 T k^2} \leq \frac{1}{T} \frac{T_1}{1- e^{-\alpha_1 T_1}}
  \leq e^{1/T} \frac{T_1}{1- e^{-\alpha_1 T_1}}.
 \end{equation}
 Gathering \eqref{eq:uhm-uml2} \eqref{eq:v-tech1}, \eqref{eq:v-tech2}, and \eqref{eq:v-tech3} proves the size estimate \eqref{eq:u-moments-size} with the claimed constant $M_1$ defined in \eqref{eq:m1}.
\end{proof}

\subsection{Cost of controllability for the linearized system}
\label{sec:linear-cost}

The first step of the source-term method is to compute an estimate of the cost of controllability for the linearized system \eqref{eq:z-linear}. Roughly speaking, the cost of controllability is the minimal size of the controls one must use to drive an initial state to zero in a given time. This topic has received much attention. In the particular case of \eqref{eq:z-linear}, we refer to the recent work \cite{MR3619248} and the references therein. 

In order for the linear system \eqref{eq:z-linear} to be small-time null controllable, it is necessary to assume that the coefficients $|\langle \gmu, \varphi_k\rangle| \neq 0$ do not decay too fast. More precisely, that
\begin{equation} \label{eq:hyp-mu-T0}
 T_\gmu := \limsup_{k \to +\infty} \frac{- \log |\langle \gmu, \varphi_k \rangle|}{k^2} = 0.
\end{equation}
When $T_\gmu$ is positive, then it is the minimal time of null controllability for the linearized system (see \cite{MR3250361,MR3619248} for more details). In the sequel, we always assume that \eqref{eq:hyp-mu-T0} holds as we are interested in obstructions to controllability caused by the quadratic  properties of the system. 

In order for the linear system \eqref{eq:z-linear} to be controllable with the usual control cost for the heat equation of the form $e^{C/T}$, we must add a stronger assumption than \eqref{eq:hyp-mu-T0}. We assume
\begin{equation} \label{eq:mu-b}
 b_\gmu := \limsup_{k \to +\infty} \frac{- \log |\langle \gmu, \varphi_k\rangle|}{k} < + \infty.
\end{equation}
Assumption \eqref{eq:mu-b} implies \eqref{eq:hyp-mu-T0} and is satisfied for a very wide class of $\gmu \in H^{-1}_N(0,\pi)$.

The following result is quite classical. We include a proof for the sake of completeness and because it is not so frequent to build regular controls in the dissipative case. For time-reversible systems, some authors studied the behavior of the HUM operator (see e.g.\ \cite{MR2486082,MR2649987}) or developed methods to obtain regular controls from the HUM method (see e.g.\ \cite{MR2679646}).

\begin{proposition} \label{thm:linear-cost}
 Let $\gmu \in H^{-1}_N(0,\pi)$. Assume that, for every $k \in \N$, $\langle \gmu, \varphi_k \rangle \neq 0$ and that $\gmu$ satisfies~\eqref{eq:mu-b}. Let $m \in \N$. There exists $M_2 > 0$ such that, for any $T > 0$, there exists a continuous linear map $\mathfrak{L}^T_2 : L^2(0,\pi) \to H^m_0(0,T)$ such that, for any $z_0 \in L^2(0,\pi)$, the solution $z \in Z$ to the linear system \eqref{eq:z-linear} with control $u := \mathfrak{L}^T_2(z_0)$ satisfies $z(T) = 0$ and
 \begin{equation} \label{eq:cost}
  \|u\|_{H^m(0,T)} \leq M_2 e^{M_2/T} \|z_0\|_{L^2(0,\pi)}.
 \end{equation}
\end{proposition}

\begin{proof}
 From the decay assumption \eqref{eq:mu-b}, there exists $k_b \in \N$ such that, for $k \geq k_b$, one has, for any $\tau > 0$,
 \begin{equation} \label{eq:muk-large}
  \frac{1}{|\langle \gmu, \varphi_k \rangle|} \leq e^{(b_\gmu+1)k}
  \leq e^{(b_\gmu+1)^2/\tau} e^{k^2 \tau / 4}.
 \end{equation}
 Moreover, for $k \leq k_b$, one has
 \begin{equation} \label{eq:muk-small}
  \frac{1}{|\langle \gmu, \varphi_k \rangle|} 
  \leq 
  B_\gmu := \sup_{k' \leq k_b} 
  \frac{1}{|\langle \gmu, \varphi_{k'} \rangle|} 
  < + \infty.
 \end{equation}
 Let $m \in \N$ and $M_1, T_1 > 0$ be given by \cref{thm:moments} for the exponent $\eta_1 = 3/4$. We set $M := \max \{ M_1, M_1 B_\gmu, M_1 + (1+b_\gmu)^2 \}$ and $M_2 := \max \{ M, M e^{M/T_1} \}$.

 \bigskip
 
 Let $T > 0$ and $z_0 \in L^2(0,\pi)$. Let $\tau := \min \{ T, T_1 \}$. Each component of the state $z(\tau)$ can be computed explicitly as
 \begin{equation}
   \langle z(\tau), \varphi_k \rangle = \langle z_0, \varphi_k \rangle e^{-k^2 \tau}
   + \langle \gmu, \varphi_k \rangle \int_0^\tau u(t) e^{-k^2(\tau-t)}  \dd t.
 \end{equation}
 We construct a sequence $(d_k)_{k \in \N}$ as
 \begin{equation}
  d_k := - \frac{\langle z_0, \varphi_k \rangle}{\langle \gmu, \varphi_k \rangle} e^{-k^2 \tau}.
 \end{equation}
 Since $|\langle z_0, \varphi_k \rangle| \leq \|z_0\|_{L^2}$, we obtain thanks to \eqref{eq:muk-large} and \eqref{eq:muk-small}, for any $k \in \N$,
 \begin{equation} \label{eq:sizedk}
  |d_k| \leq \|z_0\|_{L^2} \max \{ B_\gmu, e^{(1+b_\gmu)^2 / \tau} \} e^{-3k^2 \tau/4}.
 \end{equation}
 From \eqref{eq:sizedk}, $d \in D_\tau$ with $\eta_1 = 3/4$. We set $u = \mathfrak{L}^T_2(z_0) := \mathfrak{L}^{\tau}_1(d)$, which we extend by zero on $[\tau,T]$ if $T > \tau$. From the size estimate for the resolution of the moment problem \eqref{eq:u-moments-size} and \eqref{eq:sizedk}, we have 
 \begin{equation}
  \begin{split}
   \|u\|_{H^m(0,T)} 
   & \leq M_1 e^{M_1/\tau} \|d\|_{D_\tau} 
   \\
   & \leq M_1 e^{M_1/\tau} \|z_0\|_{L^2} \max \{ B_\gmu, e^{(1+b_\gmu)^2 / \tau} \}
   \\
   & \leq M e^{M/\tau} \|z_0\|_{L^2}
   \\
   & \leq M_2 e^{M_2 / T} \|z_0\|_{L^2}.
  \end{split}
 \end{equation} 
 This concludes the proof of the cost estimate \eqref{eq:cost}.
\end{proof}

\subsection{Controllability despite a source term}

The key point of the source term method is to prove that, if a linear system is null controllable (i.e. if one can use a control to drive a non-zero initial state back to zero), then one can also use a control to drive the state to zero despite a source term, provided that it vanishes quick enough near the final time, compared to the control cost in small time. We consider the forced version of the linear control system \eqref{eq:z-linear}: 
\begin{equation} \label{eq:z-linear-forced}
\left\{
\begin{aligned}
& \partial_t z(t,x) - \partial_{xx} z(t,x) = u(t) \gmu(x) + f(t,x), \\
& \partial_x z (t,0) =  \partial_x z(t,\pi) = 0, \\
& z(0,x) = z_0(x).
\end{aligned}
\right.
\end{equation}
Let $\gmu \in H^{-1}_N(0,\pi)$ satisfying \eqref{eq:mu-b}. Let $M_2$ be given by \cref{thm:linear-cost}.
Let $T > 0$, $q \in (1,\sqrt{2})$ and $p > q^2/(2-q^2)$. We define the weights
\begin{align}
 \label{eq:rho0}
\rho_0(t) & := M_2^{-p} \exp\left(-\frac{pM_2}{(q-1)(T-t)}\right), \\
 \label{eq:rhoF}
\rho_\mathcal{F}(t) & := M_2^{-1-p} \exp\left(-\frac{(1+p)q^2M_2}{(q-1)(T-t)}\right).
\end{align}
Then we define associated spaces for the source term, the state and the control
\begin{align}
\label{eq:F}
\mathcal{F} & := \left\{ f \in L^2((0,T);{H^{-1}_N(0,\pi)}), \enskip \int_0^T \|f(t)\|^2_{H^{-1}_N(0,\pi)} / \rho^2_\mathcal{F}(t) \dd t < + \infty \right\}, \\
\label{eq:Z}
\mathcal{Z} & := \left\{ z \in L^2((0,T);H^1_N(0,\pi)), \enskip \int_0^T \|z(t)\|_{H^1_N(0,\pi)}^2 / \rho^2_0(t) \dd t < + \infty \right\}, \\
\label{eq:U}
\mathcal{U} & := \left\{ u \in L^\infty(0,T), \enskip \sup_{t\in[0,T]} |u(t)|^2 / \rho^2_0(t) < + \infty \right\}.
\end{align}

\begin{proposition} \label{thm:source-control}
 Let $\gmu \in H^{-1}_N(0,\pi)$. Assume that, for every $k \in \N$, $\langle \gmu, \varphi_k \rangle \neq 0$ and that $\gmu$ satisfies \eqref{eq:mu-b}. Let $m \in \N^*$ and $T > 0$. There exists $C_3 > 0$ and a continuous linear map $\mathfrak{L}_3 : L^2(0,\pi) \times \mathcal{F} \to \mathcal{U} \cap H^m_0(0,T)$ such that, for any $z_0 \in L^2(0,\pi)$ and any $f \in \mathcal{F}$, the solution $z \in Z$ to \eqref{eq:z-linear-forced} with a control $u := \mathfrak{L}_3(z_0,f)$ satisfies $z(T) = 0$ and 
 \begin{equation} \label{eq:source-control}
  \|z\|_{\mathcal{Z}} 
  + \| z / \rho_0 \|_{C^0([0,T];L^2(0,\pi))} + \|u\|_{\mathcal{U}}
  + \|u\|_{H^m(0,T)} 
  \leq C_3 \left( \| z_0 \|_{L^2(0,\pi)} + 
   \|f\|_{\mathcal{F}} \right). 
 \end{equation}
\end{proposition}

\begin{rk}
 \cref{thm:source-control} is inspired by \cite[Proposition 2.3]{MR3023058}. However, we give a proof below because we introduced changes in the definitions of the functional spaces. Namely, we use forces with weaker regularity ($H^{-1}_N(0,\pi)$ instead of $H^1_N(0,\pi)$), we ask for stronger regularity on the state ($H^1_N(0,\pi)$ instead of $L^2(0,\pi)$) and stronger regularity on the constructed controls. Our proof follows the time decomposition scheme introduced in the original paper.
\end{rk}

\begin{proof}
 For $k \geq 0$, we define $T_k := T(1-q^{-k})$. On the one hand, we let $a_0 := z_0$ and, for $k \geq 0$, we define $a_{k+1} := z_f(T_{k+1}^-)$ where $z_f$ is the solution to
 \begin{equation}
  \left\{
  \begin{aligned}
   & \partial_t z_f  - \Delta_N  z_f   = f   \quad 
    \text{ on } (T_k,T_{k+1}), \\
   & z_f(T_k^+,.) = 0 
  \end{aligned}
  \right.
 \end{equation}
 From \cref{thm:wp-linear}, using the energy estimate \eqref{eq:wp-linear}, we have
 \begin{equation} \label{eq:ak1-estimate}
 \|a_{k+1}\|_{L^2(0,\pi)} 
 \leq \|z_f\|_{C^0([T_k,T_{k+1}], L^2(0,\pi))}
 \leq C_T \| f \|_{L^2((T_k,T_{k+1}),H^{-1}_N(0,\pi))}.
 \end{equation}
 On the other hand, for $k\geq0$, we also consider the control systems
 \begin{equation}
  \left\{
  \begin{aligned}
   & \partial_t z_u - \Delta_N z_u = u_k \gmu \quad \text{on } (T_k,T_{k+1}), \\
   & z_u(T_k^+) = a_k.
  \end{aligned}
  \right.
 \end{equation}
 Exceptionally, in this proof, the notation $u_k$ denotes a sequence of controls, and not the $k$-th primitive of a given function $u$ as in (\ref{primitive_n}). Using \cref{thm:linear-cost}, we define $u_k \in H^m_0(T_k,T_{k+1})$ as $u_k(t) := (\mathfrak{L}_2^{T_{k+1}-T_k}(a_k))(t-T_k)$. Hence, $z_u(T_{k+1}^-) = 0$ and, thanks to the cost estimate \eqref{eq:cost},
 \begin{equation} \label{eq:uk-estimate}
  \|u_k\|^2_{H^m(T_k,T_{k+1})} \leq M_2^2 e^{2M_2/(T_{k+1}-T_k)} \|a_k\|^2_{L^2(0,\pi)}.
 \end{equation}
 In particular, for $k = 0$,
 \begin{equation}
  \|u_0\|^2_{H^m(T_0,T_1)} \leq M_2^2 e^{2qM_2/(T(q-1))} \|z_0\|^2_{L^2(0,\pi)}.
 \end{equation}
 And, since $\rho_0$ is decreasing
 \begin{equation}
  \|u_0^{(m)}/\rho_0\|^2_{L^2(T_0,T_1)}
  + \|u_0/\rho_0\|^2_{L^\infty(T_0,T_1)}
  \leq (1+T) \rho_0^{-2}(T_1) M_2^2 e^{2qM_2/(T(q-1))} \|z_0\|^2_{L^2(0,\pi)}.
 \end{equation}
 For $k \geq 0$, since $\rho_\mathcal{F}$ is decreasing, combining \eqref{eq:ak1-estimate} and \eqref{eq:uk-estimate} yields
 \begin{equation}
  \|u_{k+1}\|^2_{H^m(T_{k+1},T_{k+2})} 
  \leq C_T^2 M_2^2 e^{2M_2/(T_{k+2}-T_{k+1})} \rho_\mathcal{F}^2(T_k) \| f / \rho_\mathcal{F} \|_{L^2((T_k,T_{k+1}),H^{-1}_N(0,\pi))}^2.
 \end{equation}
 In particular, since $\rho_0$ is decreasing and $m \geq 1$,
 \begin{equation}
 \begin{split}
  & \|u_{k+1}^{(m)}/\rho_0\|^2_{L^2(T_{k+1},T_{k+2})}
 + \|u_{k+1}/\rho_0\|^2_{L^\infty(T_{k+1},T_{k+2})} 
  \\ 
 & \leq C_T^2(1+T) M_2^2 e^{2M_2/(T_{k+2}-T_{k+1})} \rho_0^{-2}(T_{k+2}) \rho_\mathcal{F}^2(T_k) \| f / \rho_\mathcal{F} \|_{L^2((T_k,T_{k+1}),H^{-1}_N(0,\pi))}^2.
 \end{split}
 \end{equation}
 Using the definitions of the weights \eqref{eq:rho0} and \eqref{eq:rhoF}, we obtain
 \begin{equation}
  \begin{split}
 \|u_{k+1}^{(m)}/\rho_0\|^2_{L^2(T_{k+1},T_{k+2})}
 & + \|u_{k+1}/\rho_0\|^2_{L^\infty(T_{k+1},T_{k+2})} \\
 & 
 \leq C_T^2(1+T) \| f / \rho_\mathcal{F} \|_{L^2((T_k,T_{k+1}),H^{-1}_N(0,\pi))}^2.
  \end{split}
 \end{equation}
 As in the original proof, we can paste the controls $u_k$ for $k \geq 0$ together by defining
 \begin{equation}
  \mathfrak{L}_3 (z_0,f) := u_0(z_0) + \sum_{k\geq1} u_k(f).
 \end{equation}
 The concatenated control $u := \mathfrak{L}_3(z_0,f)$ remains $H^m_0$ even across the junctions because its derivatives vanish at each $T_k$. And we have the estimate
 \begin{equation}
  \begin{split}
   \|u^{(m)}/\rho_0\|^2_{L^2(0,T)}
  & + \|u/\rho_0\|^2_{L^\infty(0,T)} 
  \leq C_T^2(1+T) \| f \|_{\mathcal{F}}^2
  \\
  & + (1+T)  \rho_0^{-2}(T_1) M_2^2 e^{2qM_2/(T(q-1))} \|z_0\|^2_{L^2(0,\pi)}.
  \end{split}
 \end{equation}
 The state $z$ can also be reconstructed by concatenation of $z_f + z_u$, which are continuous at each junction thanks to the construction. Then we estimate the state. We use the energy estimate \eqref{eq:wp-linear} from \cref{thm:wp-linear} on each time interval. Hence
 \begin{equation}
 \|z_f\|^2_{C^0([T_k,T_{k+1}],L^2(0,\pi))}
 + \|z_f\|^2_{L^2((T_k,T_{k+1}),H^1_N(0,\pi))}
 \leq 
 C_T^2 \|f\|^2_{L^2((T_k,T_{k+1}),H^{-1}_N(0,\pi))}.
 \end{equation}
 and
 \begin{equation}
  \begin{split}
 \|z_u\|^2_{C^0([T_k,T_{k+1}],L^2(0,\pi))}
 & + \|z_u\|^2_{L^2((T_k,T_{k+1}),H^1_N(0,\pi))}
 \\ & \leq 
 C_T^2 \|a_k\|_{L^2}^2 + C_T^2 \|\gmu\|_{H^{-1}(0,\pi)}^2 \|u_k\|_{L^2(T_k,T_{k+1})}^2.
  \end{split}
 \end{equation}
 Proceeding similarly as for the estimate on the control, we obtain respectively
 \begin{equation}
   \|z_f/\rho_0\|^2_{C^0([0;T],L^2(0,\pi))}
 + \|z_f/\rho_0\|^2_{L^2((0,T),H^1_N(0,\pi))}
 \leq 
 C_T^2 M_2^{-2} \|f\|_{\mathcal{F}}^2
 \end{equation}
 and
  \begin{equation}
  \begin{split}
 \|z_u/&\rho_0\|^2_{C^0([0,T],L^2(0,\pi))}
 + \|z_u/\rho_0\|^2_{L^2((0,T),H^1_N(0,\pi))}
 \\ & \leq 
 C_T^2 (M_2^{-2} + \|\gmu\|_{H^{-1}(0,\pi)}^2) 
\|f\|_{\mathcal{F}}^2
\\
 & + C_T^2 \rho_0^{-2}(T_1) \left( 1 + \|\gmu\|_{H^{-1}(0,\pi)}^2 M_2^2 e^{2qM_2/(T(q-1))} \right) \|z_0\|_{L^2(0,\pi)}^2.
  \end{split}
 \end{equation}
 This concludes the proof of estimate \eqref{eq:source-control} for an appropriate choice of constant $C_3$.
 The estimate of $z/\rho_0$ in $C^0([0,T],L^2(0,\pi))$ implies that $z(T,\cdot) = 0$ because $\rho_0(T) = 0$.
\end{proof}

\subsection{Fixed-point argument for the nonlinear system}

We conclude the proof of \cref{thm:sstlc-linear} thanks to a fixed point argument. 

\bigskip

Let $\Gamma$ satisfying \eqref{def:Gamma}. Let $T > 0$ and $m \in \N^*$. Let $C_3 > 0$ be given by \cref{thm:source-control}. We define a small radius 
\begin{equation} \label{def:deltafl}
 \delta_{\mathfrak{L}} := M_2^{p-1} \left(8 C_\Gamma C_3^2\right)^{-1}
\end{equation}
and the associated ball of $L^2(0,\pi)$:
\begin{equation}
 B_{\delta_\mathfrak{L}} := \left\{ z_0 \in L^2(0,\pi); \enskip \|z_0\|_{L^2(0,\pi)} \leq \delta_{\mathfrak{L}} \right\}.
\end{equation}
Moreover, for any $r > 0$, we set
\begin{equation}
 \mathcal{F}_r := \left\{ f \in \mathcal{F}; \enskip \|f\|_{\mathcal{F}} \leq r \right\}.
\end{equation}
We construct a map $\mathcal{N} : B_{\delta_{\mathfrak{L}}} \times \mathcal{F}_{\delta_{\mathfrak{L}}} \to \mathcal{F}_{\delta_{\mathfrak{L}}}$ by setting, for $z_0 \in B_{\delta_{\mathfrak{L}}}$ and $f\in\mathcal{F}_{\delta_{\mathfrak{L}}}$,
\begin{equation}
 \mathcal{N}(z_0,f) := u (\Gamma[z]-\Gamma[0]),
\end{equation}
where $u := \mathfrak{L}_3(z_0,f)$ is given by \cref{thm:source-control} and $z$ is the associated trajectory to \eqref{eq:z-linear-forced} with initial data $z_0$, control $u$ and source $f$. 

\begin{itemize}
 \item \emph{First step}. For each $z_0 \in B_{\delta_{\mathfrak{L}}}$, the  application $\mathcal{N}(z_0,\cdot)$ maps $\mathcal{F}_{\delta_{\mathfrak{L}}}$ to itself. Indeed, thanks to \cref{thm:uz-estimate} (see below), the source-cost estimate \eqref{eq:source-control} and the definition of~$\delta_{\mathfrak{L}}$ in \eqref{def:deltafl}, one has
 \begin{equation}
  \begin{split}
   \| \mathcal{N}(z_0,f) \|_{\mathcal{F}}
   & \leq C_\Gamma M_2^{1-p} \|u\|_{\mathcal{U}} \|z\|_{\mathcal{Z}}
   \\
   & \leq C_\Gamma M_2^{1-p} C_3^2 
   \left( \|z_0\|_{L^2(0,\pi)} + \|f\|_{\mathcal{F}}\right)^2
   \\
   & \leq C_\Gamma M_2^{1-p} C_3^2 
   \cdot 4 \delta_{\mathfrak{L}}^2 \leq \frac12 \delta_{\mathfrak{L}}.
  \end{split}
 \end{equation}
 
 \item \emph{Second step}. For each $z_0 \in B_{\delta_{\mathfrak{L}}}$, the application $\mathcal{N}(z_0,\cdot)$ is a contraction on $\mathcal{F}_{\delta_{\mathfrak{L}}}$ with a uniform constant $\frac12$. Indeed, using \cref{thm:uz-estimate} and \cref{thm:source-control} once more, for $f_1,f_2 \in \mathcal{F}_{\delta_{\mathfrak{L}}}$,
  \begin{equation}
 \begin{split}
   & \|  \mathcal{N}(z_0,f_1) - \mathcal{N}(z_0,f_2) \|_{\mathcal{F}}
 \\
& = \| (u_1-u_2)(\Gamma[z_1]-\Gamma[0])   +
    \ u_2(\Gamma[z_2]-\Gamma[z_1])  \|_{\mathcal{F}}
\\  
& \leq C_\Gamma M_2^{1-p} \Big( 
\|u_1-u_2\|_{\mathcal{U}} \|z_1 \|_{\mathcal{Z}} 
+ \|u_2\|_{\mathcal{U}} \|z_1-z_2\|_{\mathcal{Z}} \Big) 
\\
 & \leq C_\Gamma C_3^2 M_2^{1-p} \left(2\|z_0\|_{L^2(0,\pi)} + \|f_1\|_{\mathcal{F}} + \|f_2\|_{\mathcal{F}}\right) \|f_1-f_2\|_{\mathcal{F}}
 \\
 & \leq C_\Gamma C_3^2 M_2^{1-p}  
 \cdot 4 \delta_{\mathfrak{L}} \|f_1-f_2\|_{\mathcal{F}}
 \leq \frac12 \|f_1-f_2\|_{\mathcal{F}}.
 \end{split}
 \end{equation}
 
 \item \emph{Third step}. Thanks to the Banach fixed point theorem, for any $z_0 \in B_{\delta_{\mathfrak{L}}}$, the application $\mathcal{N}(z_0,\cdot)$ admits a unique fixed point $f_{z_0} \in \mathcal{F}_{\delta_{\mathfrak{L}}}$. We define
 \begin{equation}
  \mathfrak{L}(z_0) := \mathfrak{L}_3(z_0, f_{z_0}) \in H^m_0(0,T).
 \end{equation}
 Since the application $\mathcal{N}(z_0,\cdot)$ is continuous and has a uniform contraction constant, this defines a continuous map $\mathfrak{L}$ on $B_{\delta_{\mathfrak{L}}}$.
 
 \item \emph{Fourth step}. We estimate the size of $\mathfrak{L}(z_0)$. Let $z_0 \in B_{\delta_{\mathfrak{L}}}$ and $r := \|z_0\|_{L^2(0,\pi)} < \delta_{\mathfrak{L}}$. For $f\in\mathcal{F}_r$, repeating the same estimates as in the first step yields
 \begin{equation}
  \| \mathcal{N}(z_0,f) \|_{\mathcal{F}}
  \leq 4 C_\Gamma C_3^2 M_2^{1-p} r^2
  \leq r.
 \end{equation}
 Hence the application $\mathcal{N}(z_0,\cdot)$ actually leaves $\mathcal{F}_r$ invariant. This ensures that $f_{z_0} \in \mathcal{F}_r$. We conclude using \eqref{eq:source-control} that
 \begin{equation}
  \| \mathfrak{L}(z_0) \|_{H^m(0,T)} \leq C_3 \left(\|z_0\|_{L^2(0,\pi)} + \|f_{z_0}\|_{\mathcal{F}} \right)
  \leq 2 C_3 \|z_0\|_{L^2(0,\pi)}.
 \end{equation}
\end{itemize}
Therefore, the nonlinear system is smoothly small-time locally null controllable, with a control cost which depends linearly on the size of the initial data. This concludes the proof of \cref{thm:sstlc-linear}.

\begin{lemma} \label{thm:uz-estimate}
 Let $\Gamma$ satisfying \eqref{def:Gamma}, $v \in \mathcal{U}$ and 
 $\zeta_1, \zeta_2 \in \mathcal{Z}$. Then
 \begin{equation} \label{eq:uz-estimate}
  \| v (\Gamma[\zeta_1] - \Gamma[\zeta_2]) \|_{\mathcal{F}} 
  \leq C_\Gamma M_2^{1-p} \|v\|_{\mathcal{U}} \|\zeta_1-\zeta_2\|_{\mathcal{Z}}.
 \end{equation}
\end{lemma}

\begin{proof}
 For almost every $t \in (0,T)$, using \eqref{eq:U} and \eqref{def:Gamma}, one has
 \begin{equation}
 \| v(t) (\Gamma[\zeta_1(t)] - \Gamma[\zeta_2(t)]) \|^2_{H^{-1}_N(0,\pi)} 
 \leq C_\Gamma^2 \rho^2_0(t)  \|v\|^2_{\mathcal{U}}  \|(\zeta_1-\zeta_2)(t)\|_{H^1_N(0,\pi)}^2.
 \end{equation}
 Hence, using \eqref{eq:F} and \eqref{eq:Z}
 \begin{equation}
 \| v  (\Gamma[\zeta_1] - \Gamma[\zeta_2]) \|_{\mathcal{F}}^2 \leq C_\Gamma^2 \sup_{t\in[0,T]} \frac{\rho_0^4(t)}{\rho^2_\mathcal{F}(t)} \cdot  \|v\|^2_{\mathcal{U}} \|\zeta_1-\zeta_2\|^2_{\mathcal{Z}}.
 \end{equation}
 The supremum is finite and bounded by $M_2^{1-p}$, provided that $p \geq q^2/(2-q^2)$, which we assumed. This concludes the proof of estimate \eqref{eq:uz-estimate}.
\end{proof}

\subsection{Controllability with linear cost and linear controllability}
\label{sec:sstlc-linear-equiv}

As stated in \cref{rk:sstlc-linear-equiv}, obtaining controllability with a linear cost for the nonlinear system implies controllability for the linear system. We provide a short proof below.

\bigskip

Let $\Gamma$ satisfying \eqref{def:Gamma} and $T > 0$. Let us assume that there exists $C, \delta > 0$ such that, for any $z_0 \in L^2(0,\pi)$ with $\|z_0\|_{L^2} \leq \delta$, there exists $u \in L^\infty(0,T)$ with $\|u\|_{L^\infty(0,T)} \leq C \|z_0\|_{L^2}$ such that the solution to the nonlinear system \eqref{eq:z} satisfies $z(T) = 0$. We want to prove that this implies that the linear system is null-controllable in time $T$.

Let $z_0 \in L^2(0,\pi)$. We consider the family of initial data $\varepsilon z_0$ for $\varepsilon > 0$. From our assumption, for $\varepsilon$ small enough, there exists $u^\varepsilon$ with $\|u^\varepsilon\|_{L^\infty(0,T)} \leq C \varepsilon \|z_0\|_{L^2}$ such that the associated solution $z^\varepsilon \in Z$ to the nonlinear system \eqref{eq:z} satisfies $z^\varepsilon(T) = 0$. The sequence of controls $u^\varepsilon / \varepsilon$ is bounded in $L^2(0,T)$ and thus weakly converges in $L^2(0,T)$ towards some given control $u \in L^2(0,T)$. We will prove that the control $u$ drives the initial state $z_0$ to $0$ for the linear system \eqref{eq:z-linear}.

On the one hand, let $y^\varepsilon \in Z$ be the solution to the linear system \eqref{eq:z-linear} with initial data~$z_0$ and control $u^\varepsilon / \varepsilon$. Let $y$ be the solution to \eqref{eq:z-linear} with initial data $z_0$ and control $u$. Then
\begin{equation} \label{eq:y-yeps-series}
 \|y(T) - y^\varepsilon(T)\|_{L^2(0,\pi)}^2 = \sum_{k=0}^{\infty} \left| \langle \gmu, \varphi_k \rangle \left( \int_0^T \left(u(t) - \frac{u^\varepsilon(t)}{\varepsilon}\right) e^{-k^2(T-t)} \dd t \right) \right|^2. 
\end{equation}
Since $u^\varepsilon/\varepsilon$ converges weakly to $u$, each term of the series \eqref{eq:y-yeps-series} converges to $0$. Moreover, using that $\gmu \in H^{-1}_N(0,\pi)$, we can get a uniform bound to apply the dominated convergence theorem. Hence $y^\varepsilon(T)$ converges (strongly) towards $y(T)$ in $L^2(0,\pi)$.

On the other hand, one has
\begin{equation}
 z^\varepsilon(T) = \varepsilon y^\varepsilon(T) + \underset{\varepsilon\to 0}{O}(\varepsilon^2).
\end{equation}
Since $z^\varepsilon(T) = 0$ and $y^\varepsilon(T)$ converges to $y(T)$, this proves that $y(T) = 0$. Hence, for any $z_0 \in L^2(0,\pi)$ we build a control $u \in L^2(0,T)$ driving the solution of the linear system \eqref{eq:z-linear} to zero.

\section{Obstructions caused by quadratic integer drifts}
\label{sec:obs-integer}

The goal of this section is to prove \cref{thm:obs-integer}. We start by explaining the heuristic of the proof. Then, we justify the successive steps of the heuristic in the following paragraphs.

\subsection{Heuristic}

Let $\Gamma$ satisfying (\ref{def:Gamma}) and (\ref{lost_direction}). Let $T > 0$. Let $\eta>0$ small enough be given by \cref{Prop:WP+source}. Let $u \in L^\infty(0,T)$ with $\|u\|_{L^\infty} < \eta$. Let $\delta \in [-1,1]$. We consider the solution $z$ to (\ref{eq:z}) with $z_0=\delta \varphi_0$.

\bigskip
\noindent
By \cref{thm:wp-linear}, one may consider the solutions $z_1, z_2 \in Z$ of
\begin{itemize} 
\item the linearized system of (\ref{eq:z}), i.e.,
\begin{equation} \label{eq:z1}
\left\{
\begin{aligned}
& \partial_t z_1(t,x) - \partial_{xx} z_1(t,x) = u(t) \gmu(x), \\
& \partial_x z_1 (t,0) =  \partial_x z_1(t,\pi) = 0, \\
& z_1(0,.) = 0,
\end{aligned}
\right.
\end{equation}
which can be explicitly computed as
\begin{equation} \label{z1_explicit}
z_1(t)= \sum_{j=0}^\infty \langle \gmu , \varphi_j \rangle \left(\int_0^t u(\tau)e^{-j^2(t-\tau)} \dd\tau\right) \varphi_j,
\end{equation}
\item the second-order system of (\ref{eq:z}), i.e.,
\begin{equation} \label{eq:z2}
\left\{
\begin{aligned}
& \partial_t z_2(t,x) - \partial_{xx} z_2(t,x) = u(t) (\glambda z_1(t))(x), \\
& \partial_x z_2 (t,0) =  \partial_x z_2(t,\pi) = 0, \\
& z_2(0,.) = 0.
\end{aligned}
\right.
\end{equation}
\end{itemize}

\noindent
Under appropriate assumptions, and in an appropriate sense, the nonlinear solution can be approximated by its second-order Taylor expansion with respect to $u$, so that one has
\begin{equation} \label{approx_formel}
z(T) \approx \delta \varphi_0 + z_1(T) + z_2(T).
\end{equation}
On the one hand, the assumption (\ref{lost_direction}) leads to
\begin{equation} \label{eq:app1}
 \langle z_1(t),\varphi_0\rangle =0.
\end{equation}
Thus, the component along $\varphi_0$ is not controlled on the linearized system. On the other hand, straightforward computations lead to
\begin{equation} \label{eq:app2}
  \langle z_2(T),\varphi_0\rangle =  \int_0^T u(t) \int_0^t u(\tau) K(t-\tau) \dd \tau \dd t,
\end{equation}
where we introduce the quadratic kernel
\begin{equation} \label{def:K}
 K(\sigma):=\sum_{j=1}^\infty c_j e^{-j^2 |\sigma| }
\end{equation}
and the coefficients $c_j$ are defined in (\ref{Def:cj}). Using integration by parts, we will prove that, in an appropriate sense, there holds
\begin{equation} \label{eq:app3}
 \int_0^T u(t) \int_0^t u(\tau) K(t-\tau) \dd \tau \dd t \approx (-1)^n K^{(2n-1)}(0) \| u_n \|_{L^2(0,T)}^2.
\end{equation}
Combining \eqref{eq:app1}, \eqref{eq:app2} and \eqref{eq:app3} will prove the asymptotic behavior \eqref{trop_vague}, under the assumption that the control is small in an appropriate Sobolev space. In the following paragraphs:
\begin{itemize}
 \item we quantify (\ref{approx_formel}) in \cref{sec3:approx}, using the regularity assumption \eqref{def:Gamma},
 \item we quantify \eqref{eq:app3} in \cref{sec3:ibp}, under the assumption that the final state satisfies \eqref{eq:return},
 \item we combine these elements to conclude the proof of \cref{thm:obs-integer} in \cref{sec3:fin}.
\end{itemize}

\subsection{Approximation of the nonlinear solution} \label{sec3:approx}

We start with a definition that lightens the notations in the sequel,
which corresponds to boundedness under smallness assumption on both $T$ and $\|u\|_{L^\infty(0,T)}$, that holds uniformly with respect to $\delta \in [-1,1]$.

\begin{definition} \label{def:OO}
 Given two observable quantities $A(T,u,\delta)$ and $B(T,u,\delta)$ of interest, we will write $A(T,u,\delta)=\mathcal{O}(B(T,u,\delta))$ when there exist $C, T^* > 0$ such that, for any $T\in(0,T^*]$, there exists $\eta>0$ such that, 
 for any $u \in L^\infty(0,T)$ with $\|u\|_{L^\infty(0,T)} \leq \eta$ and any $\delta \in [-1,1]$, then one has the estimate $| A(T,u,\delta) | \leq C | B(T,u,\delta) |$. 
 In particular, the following examples hold true and will be used in the sequel:
 \begin{gather}
  \label{O1}
  \| u \|_{L^\infty(0,T)} = \mathcal{O}(T),
  \\
  \label{O2}
  A = \mathcal{O}(TA + B)
  \quad \Rightarrow \quad
  A = \mathcal{O}(B).
 \end{gather}
\end{definition}

\begin{proposition} \label{Prop:NL/quad}
 Let $\Gamma$ satisfying (\ref{def:Gamma}). For $\delta \in [-1,1]$, let $z$ denote the solution of system (\ref{eq:z}) with $z_0=\delta \varphi_0$, and $z_1, z_2$ denote the solutions of systems (\ref{eq:z1}), (\ref{eq:z2}). There holds
\begin{equation} \label{error_1}
\| z - \left( \delta \varphi_0 + z_1  \right) \|_Z 
= \mathcal{O}\left( \|u\|_{L^\infty(0,T)}^2 + |\delta| \|u\|_{L^\infty(0,T)} \right),
\end{equation}
\begin{equation} \label{error_2}
\| z - \left( \delta \varphi_0 + z_1 + z_2 \right) \|_Z 
= \mathcal{O}\left( \|u\|_{L^\infty(0,T)}^3 + |\delta| \|u\|_{L^\infty(0,T)} \right).
\end{equation}
\end{proposition}

\begin{proof}
 We denote by $z^u$ the "pure-control" solution to
\begin{equation} 
\left\{
\begin{aligned}
& \partial_t z^u(t,x) - \partial_{xx} z^u(t,x) = u(t) \Gamma[z^u(t)](x), \\
& \partial_x z^u (t,0) =  \partial_x z^u(t,\pi) = 0, \\
& z^u(0,x) = 0,
\end{aligned}
\right.
\end{equation}
The following estimates are direct consequences of the iterated application of \cref{Prop:WP+source}, then \cref{thm:wp-linear} and the regularity assumption \eqref{def:Gamma} on $\Gamma$. There holds
\begin{align}
 \|z^u\|_{Z} 
 & = \mathcal{O} \Big( \|u\|_{L^\infty(0,T)} \Big),
 \\
 \label{eq_cube1}
 \|z^u-z_1\|_{Z} 
 & = \mathcal{O} \Big( \|u\|_{L^\infty(0,T)}^2 \Big),
 \\
 \label{eq_cube2}
 \|z^u-z_1-z_2\|_{Z} 
 & = \mathcal{O} \Big( \|u\|_{L^\infty(0,T)}^3 \Big).
\end{align}
Moreover, we can write $z= \delta \varphi_0 + z^u + \overline{z}$ 
where the function $\overline{z}$ solves
\begin{equation}
 \left\{
\begin{aligned}
& \partial_t \overline{z}(t,x) - \partial_{xx} \overline{z}(t,x) = 
u(t) \left( \Gamma[z(t)]-\Gamma[z^u(t)] \right), \\
& \partial_x \overline{z} (t,0) =  \partial_x \overline{z}(t,\pi) = 0, \\
& \overline{z}(0,x) = 0,
\end{aligned}
\right.
\end{equation}
By \cref{thm:wp-linear} and \eqref{def:Gamma}, 
\begin{equation} 
 \label{eq_cubez}
 \|\overline{z}\|_Z=\mathcal{O}\Big( |\delta| \|u\|_{L^\infty(0,T)} \Big).
\end{equation}
Combining \eqref{eq_cube1} and \eqref{eq_cubez} proves \eqref{error_1}. Combining \eqref{eq_cube2} and \eqref{eq_cubez} proves \eqref{error_2}.
\end{proof}

\subsection{Study of the quadratic form} \label{sec3:ibp}

Under assumption (\ref{K_C2n}), the function $K$ belongs to $C^{2n}([0,\infty),\R)$ and we can integrate by parts $2n$ times in the quadratic form, which yields the following result.

\begin{proposition} \label{Prop:fq_syst}
 Let $n \in \N^*$ and $K \in C^{2n}(0,+\infty) \cap C^{2n-1}([0,+\infty))$ with $K^{(2n)} \in L^1_{\mathrm{loc}}([0,+\infty))$. There exists a quadratic form $Q_n$ on $\R^{2n}$, such that, for  $T>0$ and $u \in L^\infty(0,T)$,
 \begin{equation}
  \begin{split}
  \int_0^T u(t) & \int_0^t u(\tau) K(t-\tau) \dd\tau \dd t
  = (-1)^n \int_0^T u_n(t) \int_0^t u_n(\tau) K^{(2n)}(t-\tau)\dd\tau \dd t 
  \\ 
  & + \sum_{\ell=1}^n (-1)^\ell K^{(2\ell-1)}(0) \int_0^T u_\ell(t)^2 \dd t
  + Q_n \Big( u_1(T),...,u_{n}(T),\alpha_1,...,\alpha_n \Big),
  \end{split}
 \end{equation}
 where we use the shorthand notation, for $j \in \{1,...,n\}$,
 \begin{equation}
  \alpha_j := \int_0^T u_{n}(t) K^{(n+j-1)}(T-t) \dd t.
 \end{equation}
\end{proposition}

\begin{proof} 
 Let $K$ be a function satisfying the above assumptions. We prove, by finite induction on $m \in \{0,...,n\}$, that there exists a quadratic form $Q_n^m$ on $\R^{2m}$ such that
 \begin{equation} \label{hyp_rec_m}
  \begin{split}
  \int_0^T u(t) & \int_0^t u(\tau) K(t-\tau) \dd\tau \dd t
  = (-1)^m \int_0^T u_m(t) \int_0^t u_m(\tau) K^{(2m)}(t-\tau)\dd\tau \dd t 
  \\ 
  & + \sum_{\ell=1}^m (-1)^\ell K^{(2\ell-1)}(0) \int_0^T u_\ell(t)^2 \dd t
  + Q_n^m \Big( u_1(T),...,u_{m}(T),\alpha_1,...,\alpha_m \Big),
  \end{split}
 \end{equation}
 with the convention that the sum is empty when $m=0$ and $Q_n^0:=0$, so that the equality clearly holds for $m=0$. Let $m \in \{0,...,n-1\}$ be such that \eqref{hyp_rec_m} holds. We prove it at step $m+1$. Two integrations by part prove that
\begin{equation}
 \begin{split}
  \int_0^T u_m(t) & \int_0^t u_m(\tau) K^{(2m)}(t-\tau) \dd \tau \dd t 
  = - K^{(2m+1)}(0) \int_0^T u_{m+1}^2(t) \dd t
  \\
  & - \int_0^T u_{m+1}(t) \int_0^t u_{m+1}(\tau)K^{(2m+2)}(t-\tau)\dd\tau \dd t
   - \frac{K^{(2m)}(0)}{2} u_{m+1}(T)^2 
  \\
  & + u_{m+1}(T) \int_0^T u_m(\tau) K^{(2m)}(T-\tau) \dd\tau. 
 \end{split}
\end{equation}
Moreover, the integral in the last term can be rewritten as
\begin{equation}
 \int_0^T u_m(\tau) K^{(2m)}(T-\tau) \dd\tau
 = \alpha_{m+1} + \sum_{\ell=1}^{n-m} u_{m+\ell}(T) K^{(2m+\ell-1)}(0).
\end{equation}
This concludes the proof of \eqref{hyp_rec_m} at step $m+1$. 
\end{proof}
 
The following statement proves that, for particular motions of $z$, the boundary terms arising in \cref{Prop:fq_syst} can be neglected.

\begin{proposition} \label{Prop:terme_bord}
 Let $\Gamma$ be such that (\ref{def:Gamma}) and (\ref{lost_direction}) hold. Let $m \in \N^*$ and $J$ be a finite subset of $\N^*$ with cardinal $|J|=m$ and such that $\langle \gmu , \varphi_j \rangle \neq 0$ for every $j \in J$. Let $z$ be a solution to (\ref{eq:z}) with $z_0=\delta \varphi_0$ satisfying $\langle z(T) ,   \varphi_j \rangle=0$ for every $j \in J$. Then
 \begin{equation} \label{bords}
  \sum_{\ell=1}^{m} |u_\ell(T)| 
  = \mathcal{O} \Big( \sqrt{T} \|u_{m}\|_{L^2(0,T)} + \|u\|_{L^\infty(0,T)}^2 
 + |\delta|\, \|u\|_{L^\infty(0,T)}  \Big).
\end{equation}
\end{proposition}

\begin{proof}
 Let $j \in J$. Since $\langle z(T), \varphi_j \rangle = 0$, estimate \eqref{error_1} yields
 \begin{equation}
  \langle \gmu,\varphi_j\rangle \int_0^T u(t) e^{-j^2(T-t)} \dd t
=\langle z_1(T) , \varphi_j \rangle 
= \mathcal{O}\left( \|u\|_{L^\infty(0,T)}^2 + |\delta| \|u\|_{L^\infty(0,T)} \right).
 \end{equation}
 Then, $m$ iterated integrations by part and the Cauchy-Schwarz inequality prove that
 \begin{equation}
  \begin{split}
   & u_1(T) - j^2 u_2(T) +... + (-1)^{m-1} j^{2(m-1)} u_m(T)  \\
= & (-1)^{m+1} j^{2m} \int_0^T u_m(t) e^{-j^2(T-t)} \dd t + \int_0^T u(t)e^{-j^2(T-t)} \dd t \\
= & \mathcal{O}\Big( \sqrt{T} \|u_m\|_{L^2(0,T)} + \|u\|_{L^\infty(0,T)}^2 
+  |\delta| \|u\|_{L^\infty(0,T)}\Big).
  \end{split}
 \end{equation}
 Let $U=(u_k(T))_{1 \leq k \leq m} \in \R^m$ and $V$ the Vandermonde matrix associated to 
the family $(-j^2)_{j \in J}$. The invertibility of $V$ concludes the proof of \eqref{bords}.
\end{proof}


\subsection{Proof of the integer drift theorem} \label{sec3:fin}

We now conclude the proof of \cref{thm:obs-integer}.
We will use the following interpolation inequality. Such inequalities are referred to as Gagliardo-Nirenberg inequalities (see e.g. \cite[Theorem p.125]{MR0109940}). 

\begin{lemma} \label{thm:gn}
 Let $n \in \N$. There exists $C>0$ such that, for every $T>0$ and $v \in H^{3n+2}(0,T)$,
 \begin{equation} \label{eq:gns_s=0}
  \|v^{(n)}\|_{L^\infty(0,T)}^3 \leq 
  C \|v\|_{L^2(0,T)}^2 \|v^{(3n+2)}\|_{L^1(0,T)} + C T^{-3n-\frac32} \| v\|_{L^2(0,T)}^3.
\end{equation}
\end{lemma}

We proceed as explained in the heuristic paragraph. Thanks to (\ref{error_2}) of \cref{Prop:NL/quad},
\begin{equation} 
 \langle z(T),\varphi_0\rangle 
= \delta + \langle z_2(T),\varphi_0\rangle
+ \mathcal{O}\big(\|u\|_{L^\infty(0,T)}^3 + |\delta| \|u\|_{L^\infty(0,T)} \big).
\end{equation}
Moreover, thanks to the assumption \eqref{Ad_2l-1}, $K^{(2\ell-1)}(0) = 0$ for $\ell = 1, ... n-1$. Thus, applying \cref{Prop:fq_syst} yields
\begin{equation}
 \begin{split}
  \langle z_2(T),\varphi_0\rangle
= & \int_0^T u(t) \int_0^t u(\tau) K(t-\tau) \dd\tau \dd t 
\\ = &
(-1)^n  K^{(2n-1)}(0) \int_0^T |u_n|^2  +  (-1)^n \int_0^T u_n(t) \int_0^t u_n(\tau) K^{(2n)}(t-\tau) \dd\tau \dd t
\\
&  
+ \mathcal{O}\left( \sum_{\ell=1}^n |u_\ell(T)|^2 +
\left|\int_0^T u_n(t) K^{(n+\ell-1)}(t) \dd t \right|^2  \right).
 \end{split}
\end{equation}
Thanks to assumption \eqref{K_C2n}, $K^{(2n)} \in L^\infty(\R)$. Thus, using the Cauchy-Schwarz inequality,
\begin{equation}
 \langle z_2(T),\varphi_0\rangle
 = 
(-1)^n K^{(2n-1)}(0) \int_0^T |u_n|^2
       + \mathcal{O}\Big( T \|u_n\|^2_{L^2(0,T)}  + 
\sum_{\ell=1}^n  |u_\ell(T)|^2 \Big).
\end{equation}
Applying \cref{Prop:terme_bord} with $m = n$ (and taking squares on both sides), we obtain
\begin{equation}
 \begin{split}
 \sum_{\ell=1}^n  |u_\ell(T)|^2 
 & = \mathcal{O} \Big(
  T \|u_n\|_{L^2(0,T)}^2 + \|u\|_{L^\infty(0,T)}^4 + |\delta|^2 \|u\|^2_{L^\infty}
 \Big), \\
 & = \mathcal{O} \Big(
  T \|u_n\|_{L^2(0,T)}^2 + \|u\|_{L^\infty(0,T)}^3 + |\delta| \|u\|_{L^\infty}
 \Big),
 \end{split}
\end{equation}
thanks to \cref{def:OO} of the $\mathcal{O}$ notation. Thus, we conclude that
\begin{equation} \label{interm1}
 \langle z(T),\varphi_0\rangle
= \delta + (-1)^n K^{(2n-1)}(0) \int_0^T |u_n|^2 +
\mathcal{O} \Big(
  T \|u_n\|_{L^2(0,T)}^2 + \|u\|_{L^\infty(0,T)}^3 + |\delta| \|u\|_{L^\infty}
 \Big).
\end{equation}
Applying the Gagliardo-Nirenberg inequality (\ref{eq:gns_s=0}) to $v=u_n$, we get
\begin{equation} \label{interm2}
 \begin{split}
  \|u\|_{L^\infty(0,T)}^3
  &\leq 
  C \| u_n \|_{L^2(0,T)}^2
 \left(  \| u^{(2n+2)} \|_{L^1(0,T)} +
  T^{-3n-\frac{3}{2}} \|u_n\|_{L^2(0,T)}
 \right)
 \\
 &\leq 
  C \| u_n \|_{L^2(0,T)}^2
 \left(  \sqrt{T} \| u \|_{H^{2n+2}(0,T)} +
  T^{-2n-1} \|u\|_{L^\infty(0,T)}
 \right).
 \end{split}
\end{equation}
Recalling \eqref{Ad_2n-1}, gathering \eqref{interm1} and \eqref{interm2} proves that
\begin{equation} 
 \begin{split}
 \big| \langle z(T),\varphi_0\rangle & - \delta + (-1)^n a \|u_n\|_{L^2(0,T)}^2 \big|
  = 
  \\ & \mathcal{O} \Big(
  (T + \| u \|_{H^{2n+2}(0,T)} +
  T^{-2n-1} \|u\|_{L^\infty(0,T)}) \|u_n\|_{L^2(0,T)}^2 + |\delta| \|u\|_{L^\infty}
 \Big).
 \end{split}
\end{equation}
Expanding the definition of the notation $\mathcal{O}$, this means that, there exists $C_1,T_1 > 0$ such that, for any $T \in (0,T_1]$, there exists $\eta_1 > 0$ such that, for any $u \in L^\infty(0,T)$ with $\|u\|_{L^\infty(0,T)} \leq \eta_1$, the left-hand side is dominated by $C_1$ times the right-hand side.

Thus, let $\varepsilon > 0$. Let $T^* := \min \{ 1, T_1, C_1 \varepsilon/3 \}$. Let $T \in (0,T^*]$. Let $\eta := \min \{ \eta_1, T^{2n+1} \varepsilon /3 \}$. If $\|u\|_{H^{2n+2}(0,T)} \leq \eta$, these choices imply that
\begin{equation} 
 \big| \langle z(T),\varphi_0\rangle  - \delta + (-1)^n a \|u_n\|_{L^2(0,T)}^2 \big|
  \leq 
  \varepsilon \left( |\delta| + \|u_n\|_{L^2(0,T)}^2 \right). 
\end{equation}
This concludes the proof of \cref{thm:obs-integer}.

\begin{rk}
 For the previous proof of \cref{thm:obs-integer} to work, it is not necessary to assume \eqref{eq:return}. Indeed, we applied \cref{Prop:terme_bord} for $m=n$ and thus we only assumed that 
 \begin{equation} \label{contrainte_cible}
  \langle z(T),\varphi_j \rangle =0, \quad \forall j \in J,
 \end{equation}
 where $J$ is any subset of $\N^*$ of cardinal $n$ such that $\langle \mu , \varphi_j \rangle \neq 0$ for every $j \in J$. 
 
 As a consequence we prove the impossibility of any local motion from an initial condition of the form $z_0 \in \R^*_+ \varphi_0$ (or $z_0 \in \R^*_- \varphi_0$) to a target $z_f \in \overline{\mathrm{Span}}~\{ \varphi_j ; j \in \N^* \setminus J \}$.

 The constraints (\ref{contrainte_cible}) can be interpreted as the infinite dimensional analogous of the fact that, to prove a quadratic obstruction of order $n$ in finite dimension, one needs to impose that $n$ linearly controllable components of the state have returned to zero (see \cite{BMJDE} for more precise statements).
\end{rk}

\section{Obstructions caused by quadratic fractional drifts}
\label{sec:obs-fractional}

The goal of this section is to prove \cref{thm:obs-fractional}. 

\subsection{Heuristic}

We build upon the ideas used for the integer-order drifts. On the one side, we must study the asymptotic quadratic form. On the other side, we must determine if the quadratic approximation describes correctly the nonlinear state. We go through the following steps.
\begin{itemize}

 \item First, we prove in \cref{sec4:hs} that, under assumption \eqref{cj_asympt}, there exist positive constants $\gamma(s)$ and $\beta(s)$ such that the quadratic state satisfies
 \begin{equation}
  \Big| \langle z_2(T), \varphi_0 \rangle
  - a \gamma(s) (-1)^n  \|u_n\|_{H^{-s}(\R)}^2 \Big| 
  \lesssim \|u_n\|_{H^{-s-\beta}(\R)}^2.
 \end{equation}
 
 \item Then, we prove in \cref{sec4:up} that, for small-times, the $H^{-s}$ drift is indeed the dominant phenomenon because
 \begin{equation}
   \|u_n\|_{H^{-s-\beta}(\R)}^2
   \lesssim
   T^{2\beta} \|u_n\|_{H^{-s}(\R)}^2.
 \end{equation}
 
 \item Moreover, we prove in \cref{sec4:gns} a fractional Gagliardo-Nirenberg type interpolation inequality in order to absorb the cubic residuals behind the fractional drift.
 
 \item Eventually, we gather these arguments to conclude the proof of \cref{thm:obs-fractional} in \cref{sec4:fin}.

\end{itemize}

\subsection{Computation of the asymptotic quadratic form}
\label{sec4:hs}

We start with a few technical results comparing series to integrals for asymptotically large frequencies, which we state in a rather general setting, because we intend to reuse them in \cref{sec:sstlc-quad} and \cref{sec:infini}. For the fractional Sobolev drifts case, we only intend to apply them to the constant function $\Theta \equiv 1$. For a function $\Theta : I \to \R$, with $I \subset \R$, we define
\begin{equation}
 \label{W1}
 \| \Theta \|_{W^{1,\infty}(I)} := \| \Theta(x) \|_{L^\infty(I)} + \| \Theta'(x) \|_{L^\infty(I)}.
\end{equation}
For $\sigma,\sigma' \geq 0$ and a function $\Theta : \R_+ \to \R$, we define 
\begin{equation} 
 \label{W1.weight}
 \| \Theta \|_{W^{1,\infty}_{\sigma,\sigma'}} := \| x^{-\sigma} \Theta(x) \|_{L^{\infty}(1,+\infty)} + \| x^{\sigma'} \Theta'(x) \|_{L^{\infty}(1;+\infty)}
 + \| \Theta(x) \|_{L^\infty(0,1)}.
\end{equation}

\begin{lemma} \label{lm:theta.sum.int}
 Let $s \in (0,1)$, $\sigma \in [0,1)$ and $\sigma' > 0$. There exists $C,\beta > 0$ such that, for every function $\Theta \in W^{1,\infty}_{\sigma,\sigma'}$, for every $\xi \in \R$ with $|\xi| \geq 1$,
 \begin{equation} \label{eq:theta.sum.int}
  \left| \sum_{j=1}^\infty
  \frac{j^{3-4s} \Theta(j)}{j^4+\xi^2}
  -
  \int_0^\infty \frac{t^{3-4s} \Theta(t)}{t^4+\xi^2} \dd t
  \right| \leq
  C \| \Theta \|_{W^{1,\infty}_{\sigma,\sigma'}} |\xi|^{-2s-2\beta} .
  \end{equation}
\end{lemma}

\begin{proof}
 Let $s,\sigma,\sigma'$ and $\Theta$ be as above. The Taylor formula leads to
 \begin{equation} \label{tsi1}
  \begin{split}
   \sum_{j=1}^{\infty} \frac{j^{3-4s}\Theta(j)}{j^4+\xi^2} & - \int_1^\infty \frac{t^{3-4s}\Theta(t)}{t^4+\xi^2} \dd t 
   \\
   = & \sum_{j=1}^\infty \int_j^{j+1} (t-j+1) \frac{\dd}{\dd t} \left[\frac{t^{3-4s}\Theta(t)}{t^4+\xi^2}\right] \dd t
   \\
   = & \sum_{j=1}^\infty \int_j^{j+1} (t-j+1) \left( 
   \frac{(3-4s)t^{2-4s}\Theta(t) + t^{3-4s}\Theta'(t)}{t^4+\xi^2}
   - \frac{4 t^{6-4s}\Theta(t)}{(t^4+\xi^2)^2}
   \right) \dd t.
  \end{split}
 \end{equation}
 Hence, recalling the definition \eqref{W1.weight} of the weighted norm, equality \eqref{tsi1} yields
 \begin{equation} \label{tsi2}
  \begin{split}
   \Bigg| \sum_{j=1}^{\infty} \frac{j^{3-4s}\Theta(j)}{j^4+\xi^2} & - \int_1^\infty \frac{t^{3-4s}\Theta(t)}{t^4+\xi^2} \dd t \Bigg| 
   \\
   \leq 4 & \| \Theta \|_{W^{1,\infty}_{\sigma,\sigma'}}
   \left( 
    \int_1^\infty \frac{t^{3-4s-(1-\sigma)}}{t^4+\xi^2} \dd t
    + \int_1^\infty \frac{t^{3-4s-\sigma'}}{t^4+\xi^2} \dd t
    + \int_1^\infty \frac{t^{6-4s+\sigma}}{(t^4+\xi^2)^2} \dd t
   \right).
  \end{split}
 \end{equation}
 Using the change of variable $t = y |\xi|^{\frac12}$, we get, as $|\xi| \to + \infty$,
 \begin{align} 
 \label{tsi3a}
 \int_0^1 \frac{t^{3-4s}}{t^4+\xi^2} \dd t
 & = |\xi|^{-2s} 
 \int_0^{|\xi|^{-\frac12}} \frac{y^{3-4s}}{y^4+1} \dd y
 = O \left( |\xi|^{-2s-2(1-s)} \right),
 \\
 \label{tsi3b}
 \int_1^\infty \frac{t^{6-4s+\sigma}}{(t^4+\xi^2)^2} \dd t
 & = |\xi|^{-2s-\frac{1}{2}+\frac\sigma 2} 
 \int_{|\xi|^{-\frac12}}^\infty \frac{y^{6-4s+\sigma}}{(y^4+1)^2} \dd y
 = O \left( |\xi|^{-2s-\frac{1}{2}+\frac\sigma 2} \right).
 \end{align}
 Moreover, for any $\nu > 0$ ($\nu = \sigma'$ or $\nu = 1-\sigma$),
 \begin{equation} \label{tsi4}
 \begin{split}
 \int_1^\infty \frac{t^{3-4s-\nu}}{t^4+\xi^2} \dd t
 & = |\xi|^{-2s-\frac{\nu}{2}} \int_{|\xi|^{-\frac12}}^\infty \frac{y^{3-4s-\nu}}{y^4+1} \dd y 
 \\
 & = \left\lbrace 
 \begin{aligned}
 & O(|\xi|^{-2s-\frac{\nu}{2}}) && \text{ when } 3-4s-\nu>-1, \\
 & O(|\xi|^{-2s-\frac{\nu}{2}} \ln(|\xi|) ) && \text{ when } 3-4s-\nu=-1, \\
 & O(|\xi|^{-2s-2(1-s)} ) && \text{ when } 3-4s-\nu<-1.
 \end{aligned}
 \right.
 \end{split}
 \end{equation}
 Gathering \eqref{tsi2}, \eqref{tsi3a}, \eqref{tsi3b} and \eqref{tsi4} concludes the proof of \eqref{eq:theta.sum.int}.
\end{proof}

\begin{lemma} \label{lm:sum.int}
 Let $s \in (0,1)$. There exists $C, \beta > 0$ such that, for every $\xi \in \R$ with $|\xi| \geq 1$,
\begin{equation} \label{eq:sum.int.1}
\left| \sum_{j=1}^\infty
\frac{j^{3-4s}}{j^4+\xi^2}
-
\int_0^\infty \frac{t^{3-4s}}{t^4+\xi^2} \dd t
\right| \leq
C  |\xi|^{-2s-2\beta}.
\end{equation}
 and moreover, for every $\Theta \in W^{1,\infty}(\R)$ and $\xi \in \R$ with $|\xi| \geq 1$,
 \begin{equation} \label{eq:sum.int.2}
 \left| \sum_{j=1}^\infty
 \frac{j^{3-4s}\, \Theta(\ln(j))}{j^4+\xi^2} -
 \int_0^\infty \frac{t^{3-4s}\, \Theta(\ln(t))}{t^4+\xi^2} \dd t
 \right| \leq
 C \|\Theta\|_{W^{1,\infty}} |\xi|^{-2s-2\beta}.
 \end{equation}
\end{lemma}

\begin{proof} 
 Estimates \eqref{eq:sum.int.1} and \eqref{eq:sum.int.2} are corollaries of estimate \eqref{eq:theta.sum.int} from \cref{lm:theta.sum.int}. Indeed the constant function $1$ belongs to $W^{1,\infty}_{0,1}$ and, recalling definitions \eqref{W1} and \eqref{W1.weight}, there holds $\|\Theta(\ln x)\|_{W^{1,\infty}_{0,1}} \leq 2 \| \Theta \|_{W^{1,\infty}(\R)}$ for every $\Theta \in W^{1,\infty}(\R)$. 
\end{proof}

We now turn to the main result of this paragraph.

\begin{proposition} \label{Prop:Hs_0<s<1/2}
 Let $s\in(0,1)$ and $\alpha > -1+4s$. There exist constants $\gamma = \gamma(s) > 0$ and $\beta = \beta(s,\alpha) > 0$ such that, for every sequence $(c_j)_{j \in \N^*} \in \R^{\N^*}$ satisfying
 \begin{equation} \label{Hyp_cj}
  c_j=\frac{1}{j^{-1+4s}}+\underset{j\rightarrow+\infty}O\left( \frac{1}{j^\alpha} \right),
 \end{equation}
 there exists a constant $C > 0$ such that, for every $T > 0$ and $u \in L^\infty(0,T)$, if  $K:\R \rightarrow \R$ denotes the kernel defined by (\ref{def:K}), one has
 \begin{equation}
  \left| \int_0^T u(t) \int_0^t u(\tau) K(t-\tau) \dd \tau \dd t - \gamma \|u\|_{H^{-s}(\R)}^2 \right| 
  \leq C \|u\|_{H^{-(s+\beta)}(\R)}^2.
 \end{equation}
\end{proposition}

\begin{proof}

\textbf{Step 1: We prove that $K \in L^1(\R)$ and, for every $T>0$ and $u \in L^\infty(0,T)$},
\begin{equation} \label{eq:uuk.hat}
 \int_0^T u(t) \int_0^t u(\tau) K(t-\tau) \dd\tau \dd t 
= \frac{1}{4\pi} \int_{\R} |\widehat{u}(\xi)|^2 \widehat{K}(\xi) \dd\xi.
\end{equation}
Clearly, $K \in C^0(\R^*,\R)$. Moreover, for every $\kappa>0$ there exists $c(\kappa)>0$ such that, for every $y \in (0,\infty)$, $y^\kappa e^{-y} \leq c(\kappa)$. Thanks to the assumption \eqref{Hyp_cj}, there exists $M > 0$ such that $|c_j - j^{1-4s}| \leq M j^{-\alpha}$. Hence $|c_j| \leq (M+1) j^{1-4s}$. Then, for every $\sigma \in \R^*$, we have
\begin{equation}
 |K(\sigma)| \leq c(\kappa) (M+1) \left( \sum_{j=1}^\infty \frac{1}{j^{-1+4s+2\kappa}}\right) \frac{1}{|\sigma|^\kappa}.
\end{equation}
By considering $\kappa \in (1-2s,1)$ (resp. $\kappa>1$), this inequality proves that $K$ is integrable near $\sigma=0$ (resp. at infinity). Then, recalling our choice of normalization for the Fourier transform \eqref{eq:fourier}, Fubini and Plancherel's theorems prove that
\begin{equation}
 \begin{split}
  \int_0^T u(t) \int_0^t u(\tau) K(t-\tau)\dd\tau \dd t
&  = \frac{1}{2}\int_0^T u(t) \int_0^T u(\tau) K(t-\tau)\dd\tau \dd t \\
&  = \frac{1}{2}\int_0^T u(t) (u*K)(t) \dd t \\
&  = \frac{1}{4\pi} \int_{\R} |\widehat{u}(\xi)|^2 \widehat{K}(\xi) \dd\xi.
 \end{split}
\end{equation}
We introduce the constant $\gamma(s) > 0$ which is defined for $s \in (0,1)$, as
\begin{equation} \label{gamma}
 \gamma(s) := \int_0^{+\infty} \frac{y^{3-4s}}{1+ y^4} \dd y.
\end{equation}

\bigskip

\noindent \textbf{Step 2: We prove that there exists $\beta = \beta(s,\alpha) > 0$ such that}
\begin{equation} \label{eq:k2s2sb}
 \frac{1}{4\pi} \widehat{K}(\xi) = \frac{\gamma(s)}{2\pi|\xi|^{2s}} 
+ \underset{|\xi| \rightarrow \infty}{O} \left( \frac{1}{|\xi|^{2(s+\beta)}} \right).
\end{equation}
The function $x \mapsto e^{-|x|}$ has Fourier transform $\xi \mapsto \frac{2}{1+\xi^2}$ for our normalization \eqref{eq:fourier}. Thus, 
\begin{equation} \label{khat}
 \frac{1}{2} \widehat{K}(\xi)=\sum_{j=1}^\infty  \frac{j^2\, c_j }{j^4+\xi^2}.
\end{equation}
The change of variable $t=y|\xi|^{\frac12}$ proves that 
\begin{equation}
 \gamma(s) |\xi|^{-2s} = \int_0^\infty \frac{t^{3-4s}}{\xi^2 + t^4} \dd t.
\end{equation}
Thus
\begin{equation}
  \frac{1}{2} \widehat{K}(\xi)-\gamma(s)|\xi|^{-2s} = f_1(\xi) + f_2(\xi),
\end{equation}
where we define
\begin{align}
 f_1(\xi) & :=
\sum_{j=1}^\infty  \frac{j^{3-4s}}{j^4+\xi^2} 
- 
\int_0^\infty \frac{t^{3-4s}}{\xi^2 + t^4} \dd t,
 \\
 f_2(\xi) & := \sum_{j=1}^\infty \left( c_j - \frac{1}{j^{-1+4s}} \right) \frac{j^2}{j^4+\xi^2}.
\end{align}
By applying \eqref{eq:sum.int.1} from \cref{lm:sum.int}, we obtain 
$\beta_1=\beta_1(s)>0$ such that $f_1(\xi) = O ( |\xi|^{-2(s+\beta_1)} )$.
Moreover, thanks to the assumption \eqref{Hyp_cj} and $\alpha > -1+4s$, there exists $s' \in (s,1)$ and $M > 0$ such that
\begin{equation}
 |f_2(\xi)| \leq \sum_{j=1}^\infty \frac{j^{3-4s'}}{j^4+\xi^2}
 =  \gamma(s') |\xi|^{-2s'} 
 + \sum_{j=1}^\infty  \frac{j^{3-4s'}}{j^4+\xi^2} - 
 \int_0^\infty \frac{t^{3-4s'}}{\xi^2 + t^4} \dd t.
\end{equation}
By applying \eqref{eq:sum.int.1} from \cref{lm:sum.int}, we obtain $\beta_2 = \beta_2(s,\alpha) > 0$ such that $f_2(\xi) = O ( |\xi|^{-2s'}) + O(|\xi|^{-2s'-2\beta_2} )$, which concludes the proof of \eqref{eq:k2s2sb}.

\bigskip

\noindent
\textbf{Step 3: We recognize fractional Sobolev norms.} We deduce from \eqref{eq:k2s2sb} the existence of a constant $C>0$ such that, for every $\xi \in \R$,
\begin{equation} 
 \left| \frac{1}{4\pi} \widehat{K}(\xi) - \frac{\gamma(s)}{2\pi(1+\xi^2)^s} \right| \leq \frac{C}{(1+\xi^2)^{s+\beta}}
\end{equation}
Using the definition \eqref{def:normeH(-s)} of the fractional Sobolev norms, we obtain, for every $T>0$ and every $u \in L^\infty(0,T)$,
\begin{equation}
 \left| \frac{1}{4\pi} \int_{\R} \widehat{K}(\xi) |\widehat{u}(\xi)|^2 \dd\xi - \gamma \|u\|_{H^{-s}}^2 \right| \leq C \|u\|_{H^{-(s+\beta)}}^2,
\end{equation}
which, together with \eqref{eq:uuk.hat}, gives the conclusion of \cref{Prop:Hs_0<s<1/2}.
\end{proof}

\subsection{Uncertainty principle and comparison of fractional norms}
\label{sec4:up}

A non-null $L^2$ function with compact support in the time domain cannot have a compact support in the frequency domain. This idea is known as the \emph{uncertainty principle} for Fourier transform. We will use the following quantitative version of it. This inequality can be deduced from the seminal works \cite{MR0461025,MR780328}. See also \cite{MR2371612,MR1246419} for estimates of the best constant, or \cite{MR2264204,MR1448337,MR1303780} for thorough reviews.

\begin{proposition}[Uncertainty principle] \label{thm:uncertainty}
 There exists $C_{\mathrm{up}} > 0$ such that, for any $T > 0$ and for any $f \in L^2(\R)$ satisfying $|\mathrm{supp}~f| \leq T$, one has
 \begin{equation} \label{eq:uncertainty}
  \int_{\R} |\widehat{f}(\xi)|^2 \dd \xi 
  \leq C_{\mathrm{up}} \int_{|\xi| \geq 1/T} |\widehat{f}(\xi)|^2 \dd \xi.
 \end{equation}
\end{proposition}

From definition \eqref{def:normeH(-s)} of the negative fractional Sobolev norms, it was already clear that these norms were ordered, in the sense that $\|f\|_{H^{-s-\beta}(\R)} \leq \|f\|_{H^{-s}(\R)}$ for $s,\beta > 0$. Using the uncertainty principle, we prove in the following lemma that, for asymptotically small times, the weaker norms are negligible with respect to the stronger norms, up to some low-order term.

\begin{proposition} \label{Prop_Z}
 Let $\xi^\star \in i \R$, $s \in (0,1)$ and $\beta \in [0,1-s]$.
 There exists $C>0$ such that, for every $T\in(0,1]$ and every $f \in L^\infty(0,T)$,
 \begin{equation} \label{eq:embedding_KB}
  \|f\|^2_{H^{-(s+\beta)}(\R)} \leq C T^{2\beta} \|f\|^2_{H^{-s}(\R)} 
  + \frac{C}{T} |\widehat{f}(\xi^\star)|^2.
 \end{equation}
\end{proposition}

\begin{proof}

\textbf{Step 1: We start with the particular case when $s+\beta=1$ and $\widehat{f}(\xi^\star)=0$.}
First, there exists $c=c(\xi^\star)>0$ such that
\begin{equation} \label{hyp_Z}
 \forall \xi \in \R, \quad 
 \frac{1}{c} (1+\xi^2) \leq |\xi-\xi^\star|^2 \leq c (1+\xi^2).
\end{equation}
Then, using the definition \eqref{def:normeH(-s)} of the negative Sobolev norm,
\begin{equation} \label{borne_f_g}
\|f\|_{H^{-1}(\R)}^2 = \frac1{2\pi} \int_{\R} \frac{|\widehat{f}(\xi)|^2}{1+\xi^2} \dd\xi 
\leq c \int_{\R} \left| \frac{\widehat{f}(\xi)}{\xi-\xi^\star} \right|^2 \dd\xi 
= c \int_{\R} \left| \widehat{g}(\xi) \right|^2 \dd\xi,
\end{equation}
where we define
 \begin{equation} \label{eq:hatg-hatf}
  \widehat{g}(\xi) := \frac{\widehat{f}(\xi)}{\xi-\xi^\star}.
 \end{equation}
Since $f$ is supported in $[0,T]$, $\widehat{f}$ is an entire function of exponential type. There exists $C_f > 0$ such that
 \begin{equation} \label{eq:hatf-growth}
  \forall z \in \C, \quad
  |\widehat{f}(z)| \leq C_f e^{T|z|/2}.
 \end{equation}
 Moreover, since we assumed that $\widehat{f}(\xi^\star) = 0$, \eqref{eq:hatg-hatf} defines an entire function $\widehat{g}$ on $\C$. Thanks to \eqref{eq:hatf-growth}, there exists $C_g > 0$ such that
 \begin{equation} \label{eq:hatg-growth}
  \forall z \in \C, \quad
  |\widehat{g}(z)| \leq C_g e^{T|z|/2}.
 \end{equation}
 Thanks to the Paley-Wiener theorem (see e.g. \cite[Theorem 19.3, p.375]{MR924157}),
$\widehat{g}$ is the Fourier transform of a function $g \in L^2(\R)$ with a support of size at most $T$ thanks to \eqref{eq:hatg-growth}. Thus, we can apply \cref{thm:uncertainty}. From the uncertainty estimate \eqref{eq:uncertainty}, we obtain
 \begin{equation} \label{eq:gcupg1}
  \int_{\R} |\widehat{g}(\xi)|^2 \dd \xi 
  \leq C_{\mathrm{up}} \int_{|\xi| \geq 1/T} |\widehat{g}(\xi)|^2 \dd \xi\,.
 \end{equation}
 Then, (\ref{borne_f_g}) and (\ref{hyp_Z}) lead to
 \begin{equation}
  \|f\|_{H^{-1}(\R)}^2 \leq c^2 C_{\mathrm{up}} \int_{|\xi| \geq 1/T} 
 \frac{|\widehat{f}(\xi)|^2}{1+\xi^2} \dd\xi.
 \end{equation}
 Moreover, for $|\xi| \geq 1/T$, one has 
 \begin{equation} \label{eq:gcupg2}
  (1+\xi^2) \geq T^{-2(1-s)} (1+\xi^2)^s
 \end{equation}
thus
\begin{equation} \label{eq:uphomo}
 \|f\|_{H^{-1}(\R)}^2 \leq 2\pi c^2 C_{\mathrm{up}} T^{2(1-s)} \|f\|_{H^{-s}(\R)}^2.
\end{equation}

\bigskip

\noindent
\textbf{Step 2: We consider the case when $s+\beta \in (s,1)$ and $\widehat{f}(\xi^\star)=0$.}
We introduce $\theta := (1-s-\beta)/(1-s) \in (0,1)$. Using Hölder's inequality with exponents $1/\theta$ and $1/(1-\theta)$, we obtain
\begin{equation}
 \begin{split}
 \|f\|^2_{H^{-(s+\beta)}(\R)}
   & = \frac1{2\pi} \int_\R |\widehat{f}(\xi)|^{2\theta} (1+\xi^2)^{-s\theta}
   \cdot |\widehat{f}(\xi)|^{2(1-\theta)} (1+\xi^2)^{-\beta - s(1-\theta)} \dd \xi
   \\
   & \leq \|f\|_{H^{-s}(\R)}^{2\theta} \|f\|_{H^{-1}(\R)}^{2(1-\theta)}.
 \end{split}
\end{equation}
Thanks to the previous estimate \eqref{eq:uphomo}, we obtain with $C_{s,\beta} :=(2\pi c^2 C_{\mathrm{up}})^{\frac{1-\theta}{2}}$, 
\begin{equation} \label{eq:uphomo2}
 \|f\|_{H^{-(s+\beta)}(\R)}^2 \leq C_{s,\beta}^2 T^{2\beta} \|f\|_{H^{-s}(\R)}^2.
\end{equation}

\bigskip

\noindent
\textbf{Step 3: We move to the case when $\widehat{f}(\xi^\star)$ is arbitrary.} 
Let $\xi^*=ib \in i \R$ and $\chi := c_b \mathbf{1}_{[0,T]}$, where $c_b > 0$ is a
normalization constant defined as $c_b = (e^{bT}-1)^{-1}$ for $b \neq 0$ and $1/T$ for $b = 0$. Then 
\begin{equation}
 \widehat{\chi}(\xi^\star) = c_b \int_0^T e^{bt} \dd t  = 1
\end{equation}
and, for every $\sigma \in [0,1]$,
\begin{equation}
 \| \chi \|_{H^{-\sigma}(\R)} \leq \| \chi \|_{L^2(0,T)} = c_b \sqrt{T} \leq \frac{C}{\sqrt{T}}.
\end{equation}
Applying \eqref{eq:uphomo2} to the function $\widetilde{f}(t):=f(t)-\widehat{f}(\xi^\star)\chi(t)$, 
and the triangular inequality, we get
\begin{equation} \label{eq:uphomo3}
\begin{split}
\|f\|_{H^{-(s+\beta)}(\R)}
& \leq 
\| \widetilde{f} \|_{H^{-(s+\beta)}(\R)} + |\widehat{f}(\xi^\star)| \| \chi \|_{H^{-(s+\beta)}(\R)}
\\ & \leq 
C_{s,\beta} T^\beta \| \widetilde{f} \|_{H^{-s}(\R)} 
+ |\widehat{f}(\xi^\star)| \| \chi \|_{H^{-(s+\beta)}(\R)}
\\ & \leq
 C_{s,\beta} T^\beta \| f \|_{H^{-s}(\R)} 
+ |\widehat{f}(\xi^\star)| \Big( \| \chi \|_{H^{-(s+\beta)}(\R)} 
+ C_{s,\beta}T^\beta \| \chi \|_{H^{-s}(\R)} \Big)
\\
& \leq
 C_{s,\beta} T^\beta \| f \|_{H^{-s}(\R)} 
 + \frac{C}{\sqrt{T}} |\widehat{f}(\xi^\star)|\,.
\end{split}
\end{equation}
This concludes the proof of \eqref{eq:embedding_KB} in all cases.
\end{proof}

\begin{rk}
 At first sight, the lower order term involving $| \widehat{f}(\xi^\star) |$ might seem strange in \eqref{eq:embedding_KB}. When  $\xi^\star = 0$, this term is linked to the average of $f$, hence providing information on the low frequencies. Isolating low frequencies is necessary to obtain the power of $T$ we obtain. For example, the inequality $\|f\|_{L^2(0,T)} \leq C T \| f \|_{H^1(0,T)}$ is violated by constant functions. But, using the Poincaré-Wirtinger inequality, one always has
 \begin{equation} \label{eq:tuto}
  \| f \|_{L^2(0,T)} \leq T \| f \|_{H^1(0,T)} + \frac{1}{\sqrt{T}} \left| \int_0^T  f \right|.
 \end{equation}
 \cref{Prop_Z} is a transcription of \eqref{eq:tuto} to our fractional negative Sobolev spaces setting. Moreover, we need to consider the case $\xi^\star \neq 0$ because we will apply \cref{Prop_Z} to functions for which we only have information on $\widehat{f}(i)$ and not on $\widehat{f}(0)$. The proof is a little technical due to the fact that we use $H^{-a}(\R)$ norms which are not ``localized'' in $[0,T]$.
\end{rk}

\begin{lemma} \label{Prop:u(n+1)_u(n)_H(-s)}
 There exists $C>0$ such that, for $s \in (0,1)$, $T \in (0,1]$, $u \in L^\infty(0,T)$ and $n \in \N$,
 \begin{equation} \label{eq:Prop:u(n+1)_u(n)_H(-s)}
  \int_0^T |u_{n+1}(t)|^2 \dd t \leq C \left(  
  T^{2(1-s)} \|u_n\|_{H^{-s}(\R)}^2    +   T  |u_{n+1}(T)|^2   \right).
 \end{equation}
\end{lemma}

\begin{proof}
 It is sufficient to work with $n=0$. The function $u_1$ is extended by zero outside $(0,T)$. By Plancherel's equality and \cref{thm:uncertainty}, we have
 \begin{equation}
  \int_0^T |u_1(t)|^2 \dd t
  = \frac{1}{2\pi} \int_{\R} |\widehat{u_1}(\xi)|^2 \dd\xi
  \leq \frac{C_{\mathrm{up}}}{2\pi} \int_{|\xi|>1/T} |\widehat{u_1}(\xi)|^2 \dd\xi.
 \end{equation}
Moreover, using integration by parts, we obtain
\begin{equation}
 \widehat{u_1}(\xi)
 = \int_0^T u_1(t) e^{-it\xi} \dd t
 = \frac{\widehat{u}(\xi)}{i\xi} - \frac{u_1(T)e^{-iT\xi}}{i\xi}.
\end{equation}
Thus,
\begin{equation}
 \begin{split}
  \int_{|\xi|>1/T} |\widehat{u_1}(\xi)|^2 \dd\xi
& \leq 
2 \int_{|\xi|>1/T} \left|\frac{\widehat{u}(\xi)}{\xi}\right|^2 \dd\xi +
2 |u_1(T)|^2 \int_{|\xi|>1/T} \frac{\dd\xi}{\xi^2} 
\\ & \leq 
4 T^{2(1-s)} \int_{|\xi|>1/T} \frac{|\widehat{u}(\xi)|}{(1+\xi^2)^s} \dd\xi
+ 4T |u_1(T)|^2\,.
\end{split}
\end{equation}
Indeed, taking into account that $T\leq 1$, we have, for every $\xi \geq 1/T$,
\begin{equation}
 \xi^2 \geq \frac{1}{2}(1+\xi)^2 \geq \frac{T^{-2(1-s)}}{2} (1+\xi)^{2s}.
\end{equation}
Gathering these estimates concludes the proof of \eqref{eq:Prop:u(n+1)_u(n)_H(-s)}.
\end{proof}

\subsection{A fractional Gagliardo-Nirenberg inequality}
\label{sec4:gns}

In order to bound the cubic terms by the quadratic drift, we will use an interpolation inequality, similar to the one stated in \cref{thm:gn} for the integer-order case, but adapted to our fractional setting. We start with a weighted Young convolution inequality. To lighten the computations, we introduce the Japanese bracket notation by defining for $x \in \R$,
\begin{equation}
 \langle x \rangle := \left(1+x^2\right)^{\frac12}.
\end{equation}

\newcommand{\rxr}{\langle x \rangle}

\begin{lemma} \label{Lem:Young}
 Let $s \geq 0$. For every $f \in L^2(\R)$ and $g \in L^1(\rxr^s \dd x)$, there holds
 \begin{equation}
  \int_{\R} \frac{|(f*g)(x)|^2}{\rxr^{2s}} \dd x
  \leq 2^s
  \left(\int_{\R} \frac{|f(x)|^2}{\rxr^{2s}} \dd x\right)
  \left( \int_{\R} |g(x)| \rxr^{s} \dd x \right)^2.
 \end{equation}
\end{lemma}

\begin{proof}
 We use the duality caracterization of the $L^2\left( \rxr^{-2s} \dd x \right)$-norm. Let $\phi \in L^2\left( \rxr^{-2s} \dd x \right)$. Using the Fubini theorem and the relation
 \begin{equation}
  \langle x-y \rangle
  \leq \left( 1+2x^2+2y^2 \right)^{\frac12}
  \leq 2^{\frac12} \rxr \langle y \rangle, \quad \forall x,y \in \R,
 \end{equation}
 we obtain
 \begin{equation}
  \begin{split}
   \int_\R | (f*g)&(x) \phi(x) | \frac{\dd x}{\rxr^{2s}}
   =
  \int_\R \Big| \int_\R f(x-y) g(y) \dd y \Big| |\phi(x) | \frac{\dd x}{\rxr^{2s}}
\\  \leq &
\int_\R |g(y)| \int_\R
\frac{|f(x-y)|}{\langle x-y \rangle^{s}} \,  \frac{\phi(x)}{\rxr^s} \,
\frac{\langle x-y\rangle^{s}}{\rxr^s} \dd x \dd y
\\  \leq &
2^{\frac{s}2} \left( \int_\R \langle y \rangle^s |g(y)| \dd y \right)
\left( \int_\R \frac{|f(z)|^2}{\langle z \rangle^{2s}} \dd z \right)^{\frac12}
\left( \int_\R \frac{|\phi(x)|^2}{\rxr^{2s}} \dd x \right)^{\frac12},
  \end{split}
 \end{equation}
which gives the conclusion.
\end{proof}

Now we can prove the following fractional Gagliardo-Nirenberg interpolation inequality. For a recent reference tackling the fractional case with optimal norms and constants, we refer to \cite{MOROSI2017}. Although it is not optimal, we will use the following statement for which we provide a detailed proof, because it mixes Sobolev norms on $(0,T)$ and Sobolev norms on $\R$.

\begin{proposition} \label{thm:gns}
Let $n \in \N$ and $s \in (0,1)$.
There exists a constant $C>0$ such that, for $T\in(0,1]$ and 
$v \in H^{3n+2s+\frac32}(0,T)$, there holds
\begin{equation} \label{eq:gns}
\|v^{(n)}\|_{L^\infty(0,T)}^3 \leq 
\frac{C}{T^{3(n+1)}} \|v\|_{H^{-s}(\R)}^2 \|v\|_{H^{3n+2s+\frac32}(0,T)}.
\end{equation}
\end{proposition}

\begin{proof}

\textbf{Step 1:} We prove the existence of $C_1>0$ such that,
for every $f \in H^{3n+2s+\frac{3}{2}}(\R)$,
\begin{equation} \label{gn.step1}
\|f^{(n)}\|_{L^\infty(\R)}^3 \leq C_1 
\|f\|_{H^{-s}(\R)}^2 \|f\|_{H^{3n+2s+\frac{3}{2}}(\R)}.
\end{equation}
First, for every $f \in H^{3n+2s+\frac{3}{2}}(\R)$, using Fourier inversion and the Cauchy-Schwarz inequality, 
\begin{equation}
 \begin{split}
  \| f^{(n)} \|_{L^\infty}^2 
  & = \sup_{t \in \R} |f^{(n)}(t)|^2 
  \\
  & = \sup_{t \in \R} \left| \frac{1}{2\pi} \int_{\R} \widehat{f^{(n)}}(\xi) e^{it\xi} \dd\xi \right|^2 
  \\
  & \leq \frac{1}{(2\pi)^2} \int_{\R} |\widehat{f}(\xi)|^2 (\langle\xi\rangle^{-2s} + |\xi|^{6n+4s+3}) \dd\xi \int_{\R} \frac{|\xi|^{2n}}{\langle\xi\rangle^{-2s} + |\xi|^{6n+4s+3}} \dd\xi
  \\
  & \leq C \int_{\R} \langle\xi\rangle^{-2s} |\widehat{f}(\xi)|^2 \dd \xi 
  + C \int_{\R} |\xi|^{6n+4s+3} |\widehat{f}(\xi)|^2  \dd\xi.
\end{split}
\end{equation}
For $\lambda \geq 1$, applying this inequality to $f_\lambda(t) := f(t/\lambda)$, we obtain
\begin{equation}
 \| f_\lambda^{(n)} \|_{L^\infty}^2 \leq C \int_{\R} \langle\xi\rangle^{-2s} |\widehat{f}_\lambda(\xi)|^2 \dd \xi 
 + C \int_{\R} |\xi|^{6n+4s+3} |\widehat{f}_\lambda(\xi)|^2  \dd\xi.
\end{equation}
Moreover, for every $\xi \in \R$, $\widehat{f}_\lambda(\xi) = \lambda \widehat{f}(\lambda \xi)$ and $\langle \xi \rangle^{-2s} \leq \lambda^{2s} \langle \lambda \xi \rangle^{-2s}$ because $\lambda \geq 1$. Hence,
\begin{equation}
 \lambda^{-2n} \| f^{(n)} \|_{L^\infty}^2 \leq C \lambda^{1+2s} \int_{\R} \langle\xi\rangle^{-2s} |\widehat{f}(\xi)|^2 \dd \xi 
+ C \lambda^{-6n-4s-2} \int_{\R} |\xi|^{6n+4s+3} |\widehat{f}(\xi)|^2  \dd\xi.
\end{equation}
Thus, for every $\lambda \geq 1$,
\begin{equation} \label{lll}
 \| f^{(n)} \|_{L^\infty}^2 \leq C \lambda^{1+2n+2s} \| f \|_{H^{-s}(\R)}^2
 + C \lambda^{-4n-4s-2} \| f \|_{H^{3n+2s+\frac32}(\R)}^2
 .
\end{equation}
As a function of $\lambda \in (0,+\infty)$, the right-hand side of \eqref{lll} has exactly one critical point which corresponds to the argument of its minimum, achieved at 
\begin{equation}
 \lambda_\star := \left(\frac{4n+4s+2}{2n+2s+1}\frac{\| f \|^2_{H^{3n+2s+\frac32}(\R)}}{\| f \|^2_{H^{-s}(\R)}}\right)^{\frac{1}{6n+6s+3}}
 .
\end{equation}
Since $\lambda_\star \geq 1$, we can use inequality \eqref{lll} for $\lambda = \lambda_\star$, which concludes the proof of \eqref{gn.step1}.

\bigskip

\noindent
\textbf{Step 2:} We construct a continuous extension operator $P:L^2(0,1)\rightarrow L^2(\R)$ that maps continuously $H^{3n+4}(0,1)$ into $H^{3n+4}(\R)$:
\begin{equation}
\| P(v) \|_{H^{3n+4}(\R)} \leq C_2 \|v\|_{H^{3n+4}(0,1)}\,, 
\quad \forall v \in H^{3n+4}(0,1).
\end{equation}
The construction uses a generalized reflection argument, in the spirit of Babi\v c (see \cite{MR0056675}).
To simplify the notations, we denote by $m$ the integer $3n+4$ in this step.
We also identify a function $v \in L^2(0,1)$ with 
its extension by zero to the whole real line.
Let $(\alpha_1,...,\alpha_m) \in \R^m$ be the solution of the Vandermonde linear system
\begin{equation} \label{Vdm}
\sum_{\ell=1}^{m} \alpha_\ell \left(-\frac{\ell}{m} \right)^k = 1\,, 
\quad \forall k \in \{0,...,m-1\}\,.
\end{equation}
We define an extension operator $\widetilde{P}:L^2(0,1)\rightarrow L^2(-1,2)$ by
\begin{equation}
 \widetilde{P}(v)(t):=\left\lbrace \begin{aligned}
& v(t) && \text{ if } t \in [0,1]\,, \\
& \sum_{\ell=1}^{m} \alpha_\ell v\left( -\frac{\ell\, t}{m} \right)
&& \text{ if } t \in (-1,0)\,,\\
& \sum_{\ell=1}^{m} \alpha_\ell v\left(1-\frac{\ell(t-1)}{m} \right) 
&& \text{ if } t\in(1,2)\,.\\
\end{aligned}\right.
\end{equation}
Clearly, $\widetilde{P}$ maps continuously $L^2(0,1)$ into $L^2(-1,2)$.
The relation (\ref{Vdm}) ensures, for every $k \in \{0,...,m-1\}$, 
the continuity of $\widetilde{P}(v)^{(k)}$ at $t=0$ and $t=1$ when $v \in C^k([0,1],\R)$.
Thus $\widetilde{P}$ maps continuously $H^{m}(0,1)$ into $H^{m}(-1,2)$.
Let $\chi \in C^\infty_c(\R,\R)$ be such that $\chi=1$ on $[0,1]$ and $\text{supp}(\chi) \subset (-1,2)$. Then the operator $P$ defined by $P(v):=\chi\, \widetilde{P}(v)$ gives the conclusion of Step 2.

\bigskip

For every $a \in (0,3n+4)$, the Sobolev space $H^a(\R)$ defined by its norm with a Fourier multiplier is equal to an appropriate interpolate of $L^2(\R)$ and $H^{3n+4}(\R)$ (see e.g.\ \cite[Theorem 6.2.4]{MR0482275}). Moreover, $H^a(0,1)$ is also defined as an appropriate interpolate of $L^2(0,1)$ and $H^{3n+4}(0,1)$. Hence, by \cite[Theorem 4.4.1]{MR0482275}, the extension operator $P$ maps continuously $H^a(0,1)$ into $H^a(\mathbb{R})$ for every $a \in (0,3n+4)$:
\begin{equation}
\| P(v) \|_{H^{a}(\R)} \leq C_2 \|v\|_{H^{a}(0,1)}\,, 
\quad \forall v \in H^{a}(0,1)\,, \forall a \in [0,3n+4].
\end{equation}

\bigskip

\noindent \textbf{Step 3:} We prove the existence of $C_3>0$ such that, for every $v \in L^2(0,1)$,
\begin{equation}
\| P(v) \|_{H^{-s}(\R)} \leq C_3 \| v \|_{H^{-s}(\R)}.
\end{equation}
Identifying $v$ with its extension to $\R$ by zero, we obtain
\begin{equation}
 P(v) - v = \chi \sum_{\ell=1}^{m} \alpha_\ell \, v_{0,\ell }
+ \chi \sum_{\ell=1}^{m} \alpha_\ell \, v_{1,\ell},
\end{equation}
where, we set, for every $t \in \R$,
\begin{equation}
 v_{0,\ell}(t) :=v\left( -\frac{\ell\, t}{m} \right) 
\quad \text{ and } \quad 
v_{1,\ell}(t) := v\left(1-\frac{\ell(t-1)}{m} \right).
\end{equation}
Thus
\begin{equation}
\widehat{P(v)}(\xi) - \widehat{v}(\xi) =  
 \sum_{\ell=1}^{m} \alpha_\ell \, \widehat{\chi} *\widehat{v_{0,\ell}}
+  \sum_{\ell=1}^{m} \alpha_\ell \, \widehat{\chi}*\widehat{v_{1,\ell}},
\end{equation}
where
\begin{equation}
\widehat{v_{0,\ell}}(\xi)=\widehat{v}\left( -\frac{m \xi}{\ell} \right) 
\quad \text{ and } \quad 
\widehat{v_{1,\ell}}(\xi) = \widehat{v}\left( -\frac{m \xi}{\ell} \right) 
e^{i \left( 1+\frac{m}{\ell} \right) \xi}.
\end{equation}
Taking into account that
\begin{equation}\int_{\R} \langle\xi\rangle^{s} |\widehat{\chi}(\xi)|\dd\xi
 < \infty,\end{equation}
 because $\chi \in C^{\infty}_c(\R)$,
\cref{Lem:Young} proves the existence of a constant $C_3=C_3(\chi)>0$ such that
\begin{equation}
\int_{\R} \frac{|\widehat{P(v)}(\xi) - \widehat{v}(\xi)|^2}{\langle\xi\rangle^{2s}}\dd\xi \leq C_3 \int_{\R} \frac{| \widehat{v}(\xi)|^2}{\langle\xi\rangle^{2s}}\dd\xi,
\end{equation}
which gives the conclusion.

\bigskip

\noindent \textbf{Step 4:} We prove inequality (\ref{eq:gns}) when $T=1$. For $v \in H^{3n+2s+\frac32}(0,1)$, we have
\begin{equation}
\begin{split} 
 \|v^{(n)}\|_{L^\infty(0,1)}^3
& \leq \| P(v)^{(n)} \|_{L^\infty(\R)}^3 \\
& \leq C_1 \| P(v) \|_{H^{-s}(\R)} \|P(v)\|_{H^{3n+2s+\frac32}(\R)}
\text{ by Step 1} \\
& \leq C_1 C_2 C_3^2 \|v\|_{H^{-s}(\R)}^2 \|v\|_{H^{3n+2s+\frac32}(0,1)}
\text{ by Steps 2 and 3}.
\end{split}
\end{equation}

\noindent \textbf{Step 5:} We prove inequality (\ref{eq:gns}) with an arbitrary $T \in (0,1)$.
Let $v \in H^{3n+2s+\frac32}(0,T)$.
We consider the function $w:(0,1)\rightarrow \R$ defined by 
$w(\theta)=v(T\theta)$.
Then, for every $T \in (0,1]$,
\begin{align}
 \|w^{(n)}\|_{L^\infty(0,1)}^3 & = T^{3n} \|v^{(n)}\|_{L^\infty(0,T)}^3,
 \\
 \|w\|_{H^{3n+2s+\frac32}(0,1)} & \leq \|v\|_{H^{3n+2s+\frac32}(0,T)}
\end{align}
and, taking into account that $\widehat{w}(\xi)=\frac{1}{T} \widehat{v}\left(\frac{\xi}{T}\right)$,
\begin{equation}
 \|w\|_{H^{-s}(\R)}^2
 = \frac{1}{2\pi T} \int_{\R} \frac{|\widehat{v}(\eta)|^2}{(1+T^2 \eta^2)^s} \dd\eta 
\leq \frac{1}{2\pi T^{1+2s}} 
\int_{\R} \frac{|\widehat{v}(\eta)|}{(1+ \eta^2)^s} \dd\eta
= \frac{\|v\|_{H^{-s}(\R)}^2}{T^{1+2s}}.
\end{equation}
Finally, applying Step 4 to $w$ gives
\begin{equation}
 \|v^{(n)}\|_{L^\infty(0,T)}^3 \leq \frac{C_1 C_2 C_3^2}{T^{3n+2s+1}}
\|v\|_{H^{-s}(\R)}^2 \|v\|_{H^{3n+2s+\frac32}(0,T)}.
\end{equation}
This concludes the proof of \eqref{eq:gns} in the general case.
\end{proof}

\subsection{Proof of the fractional drift theorem}
\label{sec4:fin}

Let $n \in \N$, $s \in (0,1)$, $\Gamma$ be such that 
\eqref{def:Gamma} and \eqref{lost_direction} hold and the coefficients $c_j$ defined by \eqref{Def:cj} satisfy \eqref{cj_asympt} for some $a \in \R^*$ and $\alpha > -1+4s$. We also assume that \eqref{Ad_2l-1_frac} holds.

\bigskip

Working as in Section \ref{sec3:fin}, we obtain
\begin{equation} \label{inter0}
 \langle z(T),\varphi_0\rangle 
 =
 \delta 
 + \langle z_2(T),\varphi_0 \rangle 
 + \mathcal{O} \left( \|u\|_{L^\infty(0,T)}^3 + \delta \|u\|_{L^\infty(0,T)} \right).
\end{equation}
By \cref{Prop:fq_syst} and assumption (\ref{Ad_2l-1_frac}), we obtain
\begin{equation} \label{inter1}
\begin{split}
\langle z_2(T),\varphi_0\rangle =
 &  (-1)^n  \int_0^T u_n(t) \int_0^t u_n(\tau)K^{(2n)}(t-\tau)\dd\tau \dd t +
\mathcal{O} \left( 
\sum_{\ell=1}^{n} |\alpha_\ell|^2
+  |u_\ell(T)|^2 
\right).
\end{split}
\end{equation}
Moreover, one can check (see \cref{thm:alpha.petit} below) that there exists $C,\beta > 0$
such that
\begin{equation}
 \sum_{\ell=1}^n |\alpha_\ell|^2 \leq C \|u_n\|_{H^{-s-\beta}(\R)}^2.
\end{equation}
Up to choosing a smaller $\beta$, we deduce from \cref{Prop:Hs_0<s<1/2} that
\begin{equation} \label{inter2}
\langle z_2(T),\varphi_0\rangle =
 (-1)^n a \gamma(s) \|u_n\|_{H^{-s}(\R)}^2
+ \mathcal{O} \left(  \|u_n\|_{H^{-s-\beta}(\R)}^2 
+ \sum_{\ell=1}^{n} |u_\ell(T)|^2  \right).
\end{equation}
Applying \cref{Prop_Z} to $f=u_n$ and $\xi^*=i$, we obtain
\begin{equation} \label{inter3}
\langle z_2(T),\varphi_0\rangle =
 (-1)^n a \gamma(s) \|u_n\|_{H^{-s}(\R)}^2
+ \mathcal{O} \left(   
T^{2\beta}\|u_n\|_{H^{-s}(\R)}^2  
+ \frac{\left|\widehat{u_n}(i)\right|^2}{T}
+ \sum_{\ell=1}^{n}  |u_\ell(T)|^2 
\right).
\end{equation}
Moreover, an integration by parts and the Cauchy-Schwarz formula prove that
\begin{equation} \label{uni2}
 \frac{1}{T}\left|\widehat{u_n}(i)\right|^2 
 =
 \frac{1}{T}\left|u_{n+1}(T) e^T - \int_0^T u_{n+1}(t) e^{t} \dd t \right|^2 
 =
 \mathcal{O}\left( \frac{|u_{n+1}(T)|^2}{T} + \|u_{n+1}\|_{L^2(0,T)}^2 \right).
\end{equation}
%
%
By \cref{Prop:terme_bord} applied with $m=n+1$ (assumption \eqref{cj_asympt} implies that an infinite number of coefficients $\langle \mu, \varphi_j \rangle$ are non-zero),
\begin{equation} \label{inter5}
\frac{1}{T} \sum_{\ell=1}^{n+1} |u_\ell(T)|^2
= \mathcal{O} \left(  \|u_{n+1}\|_{L^2(0,T)}^2 + \frac{\|u\|_{L^\infty(0,T)}^4}{T} 
+ \frac{\delta^2 \|u\|_{L^\infty(0,T)}^2}{T} \right)
\end{equation}
and by \cref{Prop:u(n+1)_u(n)_H(-s)}, we have
\begin{equation}  \label{inter6}
 \|u_{n+1}\|_{L^2(0,T)}^2  = \mathcal{O} \left(  
  T^{2(1-s)} \|u_n\|_{H^{-s}(\R)}^2    +   
  T  |u_{n+1}(T)|^2   \right).
\end{equation}
Combining (\ref{inter5}) and (\ref{inter6}) with the properties \eqref{O1} and \eqref{O2} of the definition of the notation~$\mathcal{O}$, which includes smallness assumptions on $T$ and on $u$, we deduce
\begin{equation} \label{uni3}
 \|u_{n+1}\|_{L^2(0,T)}^2 
 + \frac{1}{T} \sum_{\ell=1}^{n+1} |u_\ell(T)|^2
 = 
 \mathcal{O} \left(  
 T^{2(1-s)} \|u_n\|_{H^{-s}(\R)}^2 
 + \|u\|_{L^\infty(0,T)}^3 
 + \delta \|u\|_{L^\infty(0,T)}
 \right).
\end{equation}
Finally, incorporating (\ref{uni2})  and (\ref{uni3}) into (\ref{inter3}), we get
\begin{equation} \label{inter9a}
\begin{split}
\langle z(T),\varphi_0\rangle =
\delta & + (-1)^n a \gamma(s) \|u_n\|_{H^{-s}(\R)}^2
\\
& 
+ \mathcal{O} \left(  
T^{\rho} \|u_n\|_{H^{-s}(\R)}^2  +
\|u\|_{L^\infty(0,T)}^3 + \delta \|u\|_{L^\infty(0,T)}
\right).
\end{split}
\end{equation}
where $\rho:=\min\{ 2\beta, 2(1-s) \}>0$. 
Applying the Gagliardo-Nirenberg inequality of \cref{thm:gns} to $v=u_n$ we get
\begin{equation}
 \|u\|_{L^\infty(0,T)}^3
= \| u_n^{(n)} \|_{L^\infty(0,T)}^3
= \mathcal{O} \left( \frac{1}{T^{3(n+1)}} \|u_n\|_{H^{-s}(\R)}^2 \|u_n\|_{H^{3n+2s+\frac32}(0,T)} \right).
\end{equation}
Moreover, for every $k \in \{0,...,n-1\}$,
\begin{equation}
\|u_n^{(k)}\|_{L^2(0,T)} = \|u_{n-k}\|_{L^2(0,T)} \leq T^{n-k} \|u\|_{L^2(0,T)}
\end{equation}
thus 
\begin{equation}
\|u_n\|_{H^{3n+2s+\frac32}(0,T)} = \mathcal{O} \left( \|u\|_{H^{2n+2s+\frac32}(0,T)} \right)
\end{equation}
and
\begin{equation}
\|u\|_{L^\infty(0,T)}^3 =  \mathcal{O} \left( \frac{1}{T^{3(n+1)}} \|u_n\|_{H^{-s}(\R)}^2 \|u\|_{H^{2n+2s+\frac32}(0,T)} \right).
\end{equation}
Incorporating the previous relation in (\ref{inter9a}), we get
\begin{equation} \label{inter10}
 \begin{split}
 \langle z(T),\varphi_0\rangle =
 \delta & + (-1)^n a \gamma(s) \|u_n\|_{H^{-s}(\R)}^2
 \\ &
+ \mathcal{O} \left(  
\Big( T^{\rho} +  \frac{\|u\|_{H^{2n+2s+\frac32}(0,T)}}{T^{3(n+1)}} \Big) \|u_n\|_{H^{-s}(\R)}^2   
 + \delta \|u\|_{L^\infty(0,T)}
\right).
 \end{split}
\end{equation}
We conclude as in the integer-order drift case, thanks to \cref{def:OO} of the $\mathcal{O}$ notation. Indeed, equality \eqref{inter10} means that there exists $C_1, T^\star_1 > 0$ such that, for every $T \in (0,T^\star_1]$, there exists $\eta_1 > 0$ such that, for every $\delta \in [-1,1]$ and $u \in L^\infty(0,T)$ with $\|u\|_{L^\infty} \leq \eta_1$, 
\begin{equation} \label{inter11}
 \begin{split}
  \Big| \langle z(T),\varphi_0\rangle -
 \delta & - (-1)^n a \gamma(s) \|u_n\|_{H^{-s}(\R)}^2 \Big| \\
 & \leq C_1 \left( \Big( T^{\rho} +  \frac{\|u\|_{H^{2n+2s+\frac32}(0,T)}}{T^{3(n+1)}} \Big) \|u_n\|_{H^{-s}(\R)}^2   
 + \delta \|u\|_{L^\infty(0,T)} \right).
 \end{split}
\end{equation}
Let $\varepsilon > 0$. We start by choosing $T^\star > 0$ with $T^\star \leq T^\star_1$ and $C_1 T^\star \leq \varepsilon / 2$. Then, let $T \in (0,T^\star]$. We choose $\eta > 0$ such that $\eta \leq \eta_1$ and $C_1 \eta T^{-3(n+1)} \leq \varepsilon / 2$ and $C_1 \eta \leq \varepsilon$. With these choices, \eqref{inter11} proves the main conclusion \eqref{eq:DRIFT_HS} of \cref{thm:obs-fractional}.

\begin{lemma} \label{thm:alpha.petit}
 Under the decay assumption \eqref{cj_asympt} on the coefficients $c_j$, there exists $C,\beta > 0$ such that, for any $\ell \in \{1,...n\}$, one has
 \begin{equation} \label{alpha.petit}
  \left| \int_0^T u_n(t) K^{(n+\ell-1)}(T-t) \dd t \right|
  \leq
   C \|u_n\|_{H^{-s-\beta}(\R)}.
 \end{equation}
\end{lemma}

\begin{proof}
 We extend $u_n$ by zero to $\R$. Seeing the integral as an $L^2(\R)$ scalar product, we obtain, thanks to Plancherel's theorem
 \begin{equation}
  \left| \int_0^T u_n(t) K^{(n+\ell-1)}(T-t) \dd t \right|
  \leq
  \frac{1}{2\pi}
  \sum_{j\in\N^*} |c_j| j^{2(n+\ell-1)} 
  \int_\R |\widehat{u_n}(\xi)| \frac{2j^2}{j^4 + \xi^2} \dd\xi.
 \end{equation}
 Let $\beta \in (0,1-s)$. Estimate \eqref{alpha.petit} holds with
 \begin{equation}
  C := 2 \sum_{j\in\N^*} |c_j| j^{4n} \left(
  \int_\R \frac{(1+\xi^2)^{s+\beta}}{(j^4+\xi^2)^2} \dd \xi \right)^{\frac12}.
 \end{equation}
 Thanks to assumption \eqref{cj_asympt}, there exists $M_1>0$ such that $|c_j| \leq M_1 j^{1-4n-4s}$. Moreover, there exists $M_2>0$ such that
 \begin{equation}
  \left( \int_\R \frac{(1+\xi^2)^{s+\beta}}{(j^4+\xi^2)^2} \dd \xi \right)^{\frac12}
  \leq
  M_2 j^{2s+2\beta-3}.
 \end{equation}
 Hence, the constant $C$ is finite when $\beta < s + \frac12$.
\end{proof}

\subsection{Drifts in essentially arbitrary norms}
\label{sec:weights}

As claimed in the abstract and in \cref{rk:weights}, the quadratic drifts can sometimes be quantified by other more general weighted fractional negative Sobolev norms. In this paragraph, we explain what modifications are needed in the proof and we give some examples. 

\bigskip

Let $s \in (0,1)$, $\sigma \in [0,1)$ and $\sigma' > 0$. Let $\Theta \in W^{1,\infty}_{\sigma,\sigma'}(\R_+)$ be a positive function. We assume that there exists $\alpha > -1+4s$ such that the sequence $c_j$ defined by \eqref{Def:cj} satisfies 
\begin{equation} 
 c_j = j^{1-4s} \Theta(j) + O(j^{-\alpha}).
\end{equation}
Proceeding as in the previous paragraphs proves that there exists $\beta > 0$ such that
\begin{equation}
 \widehat{K}_\Theta(\xi) = 2 |\xi|^{-2s} \int_0^\infty \frac{y^{3-4s}\Theta(y|\xi|^{\frac12})}{y^4+1}\dd y
 + O\left(|\xi|^{-2s-2\beta}\right).
\end{equation}
We assume moreover that the function $\Theta$ satisfies the assumption
\begin{equation} \label{eq:ax.x}
 \forall a > 0, \quad \lim_{x\to +\infty} \frac{\Theta(ax)}{\Theta(x)} = 1.  
\end{equation}
Then, since $y^{3-4s}/(1+y^4) \in L^1(\R_+)$, the dominated converge theorem implies that
\begin{equation}
 \widehat{K}_\Theta(\xi) \sim 2 \gamma(s) |\xi|^{-2s} \Theta(|\xi|^{\frac12}).
\end{equation}
Hence, following the same method as in the previous paragraphs, we can prove that the drift is quantified by the following norm of the control
\begin{equation} \label{norm.theta}
 \| u \|_{\Theta}^2 := \frac{\gamma(s)}{2\pi} \int_{\R} |\widehat{u}(\xi)|^2 (1+\xi^2)^{-s} \Theta(|\xi|^{\frac12}) \dd \xi.
\end{equation}
This quantity therefore corresponds to a kind of weighted fractional negative Sobolev norm. The assumptions of this paragraph (including \eqref{eq:ax.x}) are for example satisfied by functions of the following type for $b,c \geq 0$,
\begin{equation}
 \Theta(x) := (\ln (1+x^2))^b (\ln \ln (1+x^2))^c. 
\end{equation}
Since the norms \eqref{norm.theta} under assumption \eqref{eq:ax.x} behave essentially like the $H^{-s}$ norm, all the arguments used in the previous paragraph can also be applied in this setting.

\subsection{A methodological remark on more complex systems}
\label{sec:wsio}

Thanks to our particular choice of nonlinear system \eqref{eq:z}, the quadratic integral operators that we manipulated are associated with kernels of the form $K(t-\tau)$. Hence, we were able to interpret the quadratic operators in the Fourier time frequency domain and perform explicit computations. We chose this setting for the ease of presentation. However, the same kind of results could probably be obtained if the kernel is of the form $K(t,\tau)$. For such kernels, it is harder to perform the computations in the frequency domain. Instead, one can study the degeneracy of $K(t,\tau)$ near the diagonal $t=\tau$ to compute the associated coercivity. The residues can then be estimated using the theory of \emph{weakly singular integral operators} (see e.g.\ \cite{MR1048075,MR1397491}). Such a study was performed by the second author in \cite{2015arXiv151104995M} in the case of a nonlinear Burgers equation, establishing a drift in $H^{-5/4}$ norm.

\section{Smooth controllability stemming from the second order} 
\label{sec:sstlc-quad}

We prove \cref{thm:sstlc-quad}, which illustrates that, for scalar-input parabolic systems, smooth small-time null controllability can sometimes be recovered from the quadratic approximation.

\subsection{Construction of a magic system} \label{sec:magic}

We construct a nonlinearity $\Gamma$ satisfying \eqref{def:Gamma} with good properties. We perform this construction as a very first step, to stress that the choice of $\Gamma$ does not depend on the allotted control time $T$. From the previous sections concerning the fractional obstructions to controllability, we guess that we must build a quadratic kernel whose Fourier transform takes both positive and negative values up to infinity. We start with the following elementary lemma.
 
\begin{lemma} \label{jeconverge}
 Let $s \in (0,1)$ and $\epsilon > 0$. There exists a constant $L = L(\epsilon,s) \geq 1$ such that 
 \begin{equation}
  \label{L.epsilon}
  \int_{\R_+ \setminus [e^{-L},e^{L}]} \frac{y^{3-4s}}{1+y^4} \dd y
  \leq 
  \epsilon \gamma(s)
  .
 \end{equation}
\end{lemma}

\begin{proof}
 For $s \in (0,1)$ the integral over $\R_+$ is finite and equal to $\gamma(s)$, defined in \eqref{gamma}.
\end{proof}

Let $s \in (0,\frac12)$ and $L \geq 1$ be given by \cref{jeconverge} for $\epsilon = \frac16$. These constants are fixed for the remainder of this section. We will neither recall these assumptions, nore keep track of the dependency of the constants in the estimates with respect to these quantities. We can now turn to the construction of the system.

\paragraph{Elementary block.} We consider an elementary building block $\chi \in W^{1,\infty}(\R)$, compactly supported in $[0,3]$, affine by parts, and with a plateau of value $1$. We set
\begin{equation} \label{chi.x}
\chi(x) :=
\left\{
\begin{aligned}
& x && \text{for } x \in [0,1],
\\
& 1 && \text{for } x \in [1,2],
\\
& 3-x && \text{for } x \in [2,3].
\end{aligned}
\right.
\end{equation}
Recalling the definition of the $W^{1,\infty}$ norm in \eqref{W1}, one has $\| \chi \|_{W^{1,\infty}} = 2$. We also consider the dilated elementary block $\chi(\cdot / 2L)$, which has a plateau of length $2L$.

\paragraph{Oscillating function.} We build a function $\Theta \in W^{1,\infty}(\R)$ using alternatively positive and negative elementary blocks. We use infinitely many positive and negative blocks in alternance. It is constructed by Alg.\ \ref{algo_nul} (this algorithmic approach will be most useful in \cref{sec:infini}). More simply, the function $\Theta$ can also be seen as a periodic function of period $12L$ (see \cref{fig:thetas0}), which is caracterized by its values for $x \in [0,12L]$ by the formula
\begin{equation} \label{theta.chi.x}
 \Theta(x) := \chi\left(\frac{x}{2L}\right) - \chi\left(\frac{x}{2L}-3\right).
\end{equation}

\begin{figure}[ht!]
 \centering
 \includegraphics{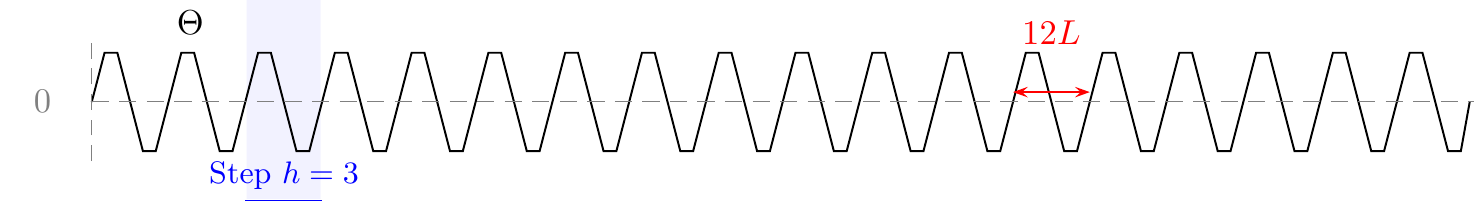}
 \caption{Single periodic oscillating function.}
 \label{fig:thetas0}
\end{figure}

\begin{algorithm}
 \KwData{$L \geq 1$}
 \KwResult{$\Theta \in W^{1,\infty}(\R)$}
 \BlankLine
 Initialize $\Theta \leftarrow 0$\;
 Initialize $X \leftarrow 0$\;
 \For{$h \leftarrow 1$ \KwTo $+\infty$}{
   $\Theta \leftarrow \Theta + \chi\left(\frac{\cdot}{2L}-X\right) -  \chi\left(\frac{\cdot}{2L}-X-3\right)$\;
   $X \leftarrow X + 6$\;
 }
 \caption{Construction of a single periodic function}
 \label{algo_nul}
\end{algorithm}

Thanks to this oscillating function, we define our nonlinearity $\Gamma_\Theta$. We set, for $z\in H^1_N(0,\pi)$,
\begin{equation} \label{magic}
 \Gamma_\Theta[z] := \sum_{k\in\N^*} k^{\frac12-3s} \varphi_k
 + \left(\sum_{j\in\N^*} \Theta(\ln j) j^{\frac12-s} \langle z, \varphi_j \rangle\right)
 \varphi_0
\end{equation}
Elementary estimates prove that definition \eqref{magic} ensures that $\Gamma_\Theta$ satisfies the regularity assumptions~\eqref{def:Gamma}. The first direction $\varphi_0$ is lost since $\langle \Gamma_\Theta[0],\varphi_0\rangle = 0$. Moreover, for $j \in \N^*$, the coefficients $c_j$ defined by \eqref{Def:cj} are given by
\begin{equation} \label{cjtheta}
 c_j = \frac{\Theta(\ln j)}{j^{-1+4s}}.
\end{equation}
In the following paragraphs, we prove that the nonlinear system \eqref{eq:z} is smoothly small-time locally null controllable with quadratic cost for the nonlinearity $\Gamma_\Theta$. 

\begin{rk}
 The shape $\Theta(\ln j)$ guarantees that the sequence $c_j$ oscillates between positive and negative values, up to infinity, at an increasingly slower pace. This property is, at least on an heuristic level, mandatory to obtain small-time controllability. 
\end{rk}

\begin{rk}
For the ease of presentation, we define by \eqref{magic} a system which is homogeneous in the sense that, if $z$ denotes the trajectory to \eqref{eq:z} with $z_0 = 0$, one has, for $j \in \N^*$,
\begin{equation} \label{z.phij}
 \langle z(T), \varphi_j \rangle = j^{\frac{1}{2}-3s} \int_0^T u(t) e^{-j^2(T-t)} \dd t 
\end{equation}
and
\begin{equation} \label{z.phi0}
 \langle z(T), \varphi_0 \rangle = \frac 12 \int_0^T\int_0^T u(t) u(t') K_\Theta (t-t') \dd t \dd t',
\end{equation}
where $K_\Theta$ is defined as in \eqref{def:K} for the coefficients \eqref{cjtheta}. The first component is purely quadratic whereas the others are purely linear. This choice simplifies the proof of \cref{thm:sstlc-quad}, but it would be possible to prove similar results for more intricate systems (e.g.\ involving cubic remainders).
\end{rk}

\subsection{Control strategy and reduction of the proof}
\label{sec:control.strategy.N=1}

We intend to prove \cref{thm:sstlc-quad} with the nonlinearity $\Gamma_\Theta$ defined in \eqref{magic}. We reduce its proof to the construction of two elementary controls.
Let $z_0 \in L^2(0,\pi)$ and $m \in \N^*$. The control strategy goes as follows:
\begin{itemize}
\item First, during the time interval $[0,T/2]$, we use a control strategy inspired by the linear dynamics. Thanks to \cref{thm:linear-cost}, there exists a control $u \in H^m_0(0,T/2)$ such that the linear system \eqref{eq:z-linear} with initial data $z_0$ and control $u$ satisfies $\langle z(T/2), \varphi_{k} \rangle = 0$ for every $k \in \N^*$. Moreover, there exists $C(T,m) > 0$ such that this control satisfies the cost estimate 
\begin{equation} \label{uhm.zp0}
 \|u\|_{H^m} \leq C \| z_0-\langle z_0,\varphi_0 \rangle \varphi_0\|_{L^2}.
\end{equation}
Applying the same control to the nonlinear system \eqref{eq:z} also ensures that $\langle z(T/2),\varphi_k \rangle = 0$ for every $k\in\N^*$, because the nonlinearity~$\Gamma_\Theta$ does not imply retroaction of the quadratic term inside the linear part of the state, as given by \eqref{z.phij}. Hence, there exists $\alpha_0 \in \R$ such that $z(T/2)=\alpha_0 \varphi_0$. Moreover, thanks to \eqref{z.phi0}, there exists $C(T) > 0$ such that
\begin{equation} \label{alpha0}
 |\alpha_0| \leq |\langle z_0,\varphi_0 \rangle| + C \|u\|_{H^m}^2.
\end{equation}
\item Second, during the time interval $[T/2, T]$, we use a control strategy dedicated to the quadratic term. We use a control
\begin{equation} \label{u.upm}
u(t) := |\alpha_0|^{\frac12} u^{\mathrm{sgn}(\alpha_0)}\left(t-\frac T2\right),
\end{equation}
where the base controls $u^\pm \in H^{m}_0(0,T/2)$ are given by the following proposition and don't depend on the initial data $z_0$. Hence, the inhomogeneous cost estimate \eqref{eq:quad.cost.N=1} from \cref{thm:sstlc-quad} follows from the estimates \eqref{uhm.zp0}, \eqref{alpha0} and the explicit formula \eqref{u.upm}.
\end{itemize}

\begin{proposition} \label{prop:elementary.control.N=1}
For every $T \in (0,1]$ and $m \in \N$, there exists $u^{\pm} \in H^{m}_0(0,T)$ such that the associated solution to system \eqref{eq:z} with $z_0=\pm \varphi_0$ satisfies $z(T)=0$.
\end{proposition}

The proof of \cref{thm:sstlc-quad} with the nonlinearity $\Gamma_\Theta$ defined in \eqref{magic} is now reduced to the proof of \cref{prop:elementary.control.N=1}. The fact that the proposition is only stated for $T \leq 1$ is not restrictive since we can always achieve null controllability in time $1$ then use null controls on the remaining time interval if $T > 1$.

We divide the proof of \cref{prop:elementary.control.N=1} in the following steps. First, we prove in \cref{sec:v.omega} that, for any $T>0$, we can find controls $v^\pm \in L^2(0,T)$ such that the associated solution to system~\eqref{eq:z} (for the nonlinearity $\Gamma_\Theta$) with $z_0 = \pm \varphi_0$ satisfies $\langle z(T),\varphi_0 \rangle = 0$, without requiring that $\langle z(T),\varphi_j\rangle=0$ for $j > 0$. Then, we prove in \cref{sec:retranche} that we can adapt the controls $v^\pm$ to also ensure these conditions. Last, we prove in \cref{sec:N=1.regular} that we can regularize these controls while preserving the result.

\subsection{Construction of rough elementary controls}
\label{sec:v.omega}

We prove that, for any $T > 0$, we can find controls $v^\pm \in L^2(0,T)$ such that the associated solution to system~\eqref{eq:z} with $z_0 = \pm \varphi_0$ satisfies $\langle z(T),\varphi_0 \rangle = 0$, without requiring that $\langle z(T),\varphi_j\rangle=0$ for $j > 0$. The heuristic behind this paragraph is that the Fourier transform of the kernel behaves (in some appropriate weak sense) as
\begin{equation} \label{khat.flou}
 \widehat{K}_\Theta(\xi) \propto \frac{\Theta(\ln\xi^{\frac12})}{|\xi|^{2s}}
\end{equation}
and thus, choosing controls whose Fourier transform are roughly supported near some $\omega_\pm$ such that $\Theta(\frac12 \ln \omega_\pm) = \pm 1$ will allow to achieve both signs in the quadratic form \eqref{z.phi0}.

\bigskip

For $\omega, \lambda \in \R$ and $\tau \geq 0$, we therefore define elementary oscillating controls $v_{\omega,\tau,\lambda} \in L^\infty(\R)$ as
\begin{equation} \label{def:v.omega.lambda}
 v_{\omega,\tau,\lambda}(t) := \sin (\omega t) \mathbf{1}_{[0,\tau]}(t) e^{-\lambda t}
 .
\end{equation}
The last tuning parameter $\lambda$ will only be necessary in \cref{sec:infini}. Each elementary oscillating control $v_{\omega,\tau,\lambda}$ oscillates at the frequency $\omega$. Thus, their Fourier transforms are mostly supported near $\xi = \pm \omega$. More precisely, they satisfy the following properties.

\begin{lemma}
 \label{lemma:v.TF}
 For every $\omega, \lambda \in \R$, $\tau \geq 0$ and $\xi \in \R$,
 \begin{align}
  \label{v.hat}
  \widehat{v}_{\omega,\tau,\lambda}(\xi) 
  & =
  e^{-\frac\tau2(\lambda+i\xi)} \frac{\tau}{2i} \left( 
  e^{i\frac\tau2\omega} \sinc \left(\frac\tau2(
  \xi-\omega-i\lambda)\right)
  -
  e^{-i\frac\tau2\omega} \sinc \left(\frac\tau2(
\xi+\omega-i\lambda)\right)
  \right),
  \\
  \label{v.loc}
  | \widehat{v}_{\omega,\tau,\lambda}(\xi) | 
  & \leq 4 \tau(1+e^{-\lambda \tau}) \left( \langle \tau(\xi - \omega) \rangle^{-1} + \langle \tau(\xi + \omega) \rangle^{-1} \right)
  .
 \end{align}
\end{lemma}

\begin{proof}
 Formula \eqref{v.hat} is a direct consequence of an explicit computation of the Fourier transform using the normalization \eqref{eq:fourier} and the definition \eqref{def:v.omega.lambda} of $v_{\tau,\omega,\lambda}$. By symmetry, we only prove that the first term of \eqref{v.hat} can be dominated by the first term of \eqref{v.loc}. Expanding the cardinal sine,
 \begin{equation}
  \begin{split}
   \Big| e^{-\frac\tau2(\lambda+i\xi)} \frac\tau2 e^{i\frac\tau2\omega} & \sinc \left(\frac\tau2(\lambda+i(\xi-\omega))\right) \Big| 
   \\ & = e^{-\frac{\lambda\tau}2} \frac{\left| \sin (-i\frac{\lambda\tau}2) \cos ((\xi-\omega)\frac\tau2) + \sin((\xi-\omega)\frac\tau2)\cos(-i\frac{\lambda\tau}2) \right|}{(\lambda^2+|\xi-\omega|^2)^{\frac12}} 
   \\ & \leq e^{-\frac{\lambda\tau}2} \frac{\left|\sin (-i\frac{\lambda\tau}2)\right|}{(\lambda^2+|\xi-\omega|^2)^{\frac12}} 
   + e^{-\frac{\lambda\tau}2} \frac{\left| \sin((\xi-\omega)\frac\tau2)\cos(-i\frac{\lambda\tau}2) \right|}{|\xi-\omega|} 
   \\ & = \frac{\left|1-e^{-\lambda\tau}\right|}{2(\lambda^2+|\xi-\omega|^2)^{\frac12}} 
   + (1+e^{-\lambda\tau}) \frac{|\sin((\xi-\omega)\frac\tau 2)|}{2|\xi-\omega|} 
   .
  \end{split}
 \end{equation}
 The first term can be bounded both by $\frac\tau2(1+e^{-\lambda\tau}) |\tau(\xi-\omega)|^{-1}$ and by $\frac\tau2(1+e^{-\lambda\tau})$. Moreover, for $x \in \R$, $| \min (1, |x|^{-1}) | \leq 2 \langle x \rangle^{-1}$ and $|\sinc (x/2)| \leq 3 \langle x \rangle^{-1}$. Thus,
 \begin{equation}
  \Big| e^{-\frac\tau2(\lambda+i\xi)} \frac\tau2 e^{i\frac\tau2\omega} \sinc \left(\frac\tau2(\lambda+i(\xi-\omega))\right) \Big| 
  \leq 
  \frac\tau2(1+e^{-\lambda\tau}) \left( 2 
  + \frac32 \right) \langle \tau(\xi-\omega) \rangle^{-1},
 \end{equation}
 which concludes the proof of \eqref{v.loc}.
\end{proof}

\begin{lemma} \label{v.frac}
 Let $\theta \in \left[0,\frac12\right)$. There exists $C_\theta > 0$ such that, for $\omega, \lambda\in \R$ and $\tau \in[0,1]$,
 \begin{equation} \label{v.frac.lambda}
  \left| 
  \| v_{\omega,\tau,\lambda} \|^2_{H^{-\theta}(\R)} - 
  \frac{1-e^{-2\lambda\tau}}{4\lambda} \frac{\omega^2}{\omega^2+\lambda^2}
  \langle\omega\rangle^{-2\theta} \right|
  \leq C_\theta (1+e^{-\lambda\tau})^2 \langle \omega \rangle^{-1},
 \end{equation}
 which, in the particular case of $\lambda=0$, should be interpreted as
 \begin{equation} \label{v.frac.0}
  \left| 
  \| v_{\omega,\tau,0} \|^2_{H^{-\theta}(\R)} - \frac\tau2
  \langle\omega\rangle^{-2\theta} \right|
  \leq 4 C_\theta \langle \omega \rangle^{-1}.
 \end{equation}
\end{lemma}

\begin{proof}
 The proof relies on estimating the difference between $\|v_{\omega,\tau,\lambda}\|^2_{H^{-\theta}(\R)}$ and $\langle \omega \rangle^{-2\theta} \|v_{\omega,\tau,\lambda}\|^2_{L^2}$. \textbf{First}, we compute explicitly the $L^2$ norm
 \begin{equation}
  \begin{split}
  \| v_{\omega,\tau,\lambda} \|_{L^2}^2 
  & = \int_0^{\tau} e^{-2\lambda t} \sin^2 (\omega t) \dd t
  \\ & = \frac{1-e^{-2\lambda\tau}}{4\lambda} \frac{\omega^2}{\omega^2+\lambda^2}
  + e^{-2\lambda\tau} \frac{\lambda(\cos(2\omega\tau)-1) - \omega\sin(2\omega\tau)}{4(\lambda^2+\omega^2)}
  .
  \end{split}
 \end{equation}
 Moreover, for $\tau \in [0,1]$ and $\omega,\lambda\in\R$,
 \begin{equation}
  \frac{|\lambda(\cos(2\omega\tau)-1)|}{4(\lambda^2+\omega^2)} +
  \frac{|\omega\sin(2\omega\tau)|}{4(\lambda^2+\omega^2)}
  \leq 
  \frac{|\cos(2\omega\tau)-1|}{4\omega} + 
  \frac{|\sin(2\omega\tau)|}{4\omega}
  \leq \langle\omega\rangle^{-1}
  .
 \end{equation}
 Hence,
 \begin{equation}
  \left| \| v_{\omega,\tau,\lambda} \|_{L^2}^2 - \frac{1-e^{-2\lambda\tau}}{4\lambda} \frac{\omega^2}{\omega^2+\lambda^2} \right|
  \leq (1+e^{-\lambda\tau})^2 \langle\omega\rangle^{-1}.
 \end{equation}
 \textbf{Second}, for $\theta>0$, we prove the existence of a constant $C_\theta > 0$ such that
 \begin{equation}
  \left| \| v_{\omega,\tau,\lambda} \|^2_{H^{-\theta}(\R)} - \langle\omega\rangle^{-2\theta} \| v_{\omega,\tau,\lambda} \|^2_{L^2} \right|
  \leq C_\theta (1+e^{-\lambda\tau})^2 \langle \omega \rangle^{-1}.
 \end{equation}
 On the one hand, there exists $C_1(\theta)>0$ such that, for every $\xi \in \left[ \frac{\omega}{2} , \frac{3\omega}{2} \right]$,
 \begin{equation}
  \left| \langle\xi\rangle^{-2\theta} - \langle\omega\rangle^{-2\theta} \right|
  \leq |\xi-\omega| \max\left\{ {2 \theta \eta }\langle\eta\rangle^{-2\theta-2} ; \eta \in [\xi,\omega] \right\}
  \leq {C_1(\theta)|\omega-\xi|} \langle\omega\rangle^{-2\theta-1}.
 \end{equation}
 On the other hand, this difference is also bounded by $1$. Hence, for any $\sigma \in (0,1)$,
 \begin{equation}
  \left| \langle\xi\rangle^{-2\theta} - \langle\omega\rangle^{-2\theta} \right|
  \leq 
  {C_1^\sigma(\theta)|\omega-\xi|^\sigma}\langle\omega\rangle^{-2\theta\sigma-\sigma}.
 \end{equation}
 Thanks to estimate \eqref{v.loc} from \cref{lemma:v.TF} and by symmetry, we only need to estimate the integral
 \begin{equation}
  \begin{split}
  \int_\R \tau^2 \langle \tau(\xi-\omega) \rangle^{-2} 
  & \left| \langle\xi\rangle^{-2\theta} - \langle\omega\rangle^{-2\theta} \right| \dd \xi 
  \\ & \leq  
 \int_{\omega/2}^{3\omega/2} 
\frac{\tau^2  |\omega-\xi|^\sigma }{\langle (\omega-\xi)\tau \rangle^{2}} 
\frac{C_1^\sigma(\theta)}{\langle\omega\rangle^{2\theta\sigma+\sigma}}\dd\xi 
+  \int_{|\xi-\omega|>\frac{\omega}{2}} 
\frac{\tau^2}{\langle (\omega-\xi)\tau \rangle^{2}} \dd\xi
\\
& \leq {2^\sigma \tau^{1-\sigma} C_1^\sigma(\theta)}{\langle\omega\rangle^{-2\theta\sigma-\sigma}}
\int_{\R_+} |\eta|^\sigma \langle \eta \rangle^{-2} \dd\eta 
+ 2  \int_{|x|>\frac{\omega}{2}} \frac{\dd x}{x^2}
.
  \end{split}
 \end{equation}
 This proves the claimed estimates by using $\sigma = 1/(2\theta+1) < 1$, which gives both 
 the correct power for $\langle \omega \rangle$ and a positive power $1-\sigma$ of $\tau \in [0,1]$, which is thus bounded.
\end{proof}

We can now justify the announced behavior \eqref{khat.flou} rigorously. 
In the following proposition, we use as a shorthand the notation, for $\omega, \tau \geq 0$,
\begin{equation} \label{v.omega.tau}
 v_{\omega,\tau} := v_{\omega,\tau,0}.
\end{equation}

\begin{proposition} \label{danslemille_new}
 There exist $C,\beta> 0$ such that, for any $\Theta \in W^{1,\infty}(\R)$, $0\leq\tau\leq T\leq 1$ and $\omega \geq 1$, if $K_\Theta$ denotes the kernel associated to $\Theta$ by \eqref{def:K} and \eqref{cjtheta}, then,
 \begin{equation}
  \left| \int_0^T \int_0^T v_{\omega,\tau}(t) v_{\omega,\tau}(t') K_\Theta(t-t') \dd t \dd t'
- 
\frac{\tau}{\omega^{2s}} G_{s,\Theta,\omega} \right|
\leq \frac{C}{\omega^{2s+2\beta}}\|\Theta\|_{W^{1,\infty}},
 \end{equation}
 where $v_{\omega,\tau}$ is defined in \eqref{v.omega.tau} and we set
 \begin{equation} \label{gsto}
  G_{s,\Theta,\omega} := \int_0^{+\infty} \frac{y^{3-4s}}{1+y^4} \Theta \left(\ln \left(y \omega^{\frac 12}\right) \right) \dd y.
 \end{equation}
\end{proposition}

\begin{proof} 
 Working as in the proof of \cref{Prop:Hs_0<s<1/2} (Step 1), we see that $K_{\Theta} \in L^1(\R)$ and, as in~\eqref{eq:uuk.hat} there holds,
 \begin{equation}
  \int_0^T \int_0^T v_{\omega,\tau}(t) v_{\omega,\tau}(t') K_\Theta(t-t') \dd t \dd t'
  = \frac{1}{2\pi} \int_\R |\widehat{v}_{\omega,\tau}(\xi)|^2 \widehat{K}_{\Theta}(\xi)\dd\xi.
 \end{equation}
 The Fourier transform $\widehat{K}_\Theta$ was already computed in \eqref{khat}.
 On the one hand, by \eqref{eq:sum.int.2} from \cref{lm:sum.int}, there exists $C_1,\beta>0$ independent of $\Theta$ such that, for $|\xi| \geq 1$,
 \begin{equation}
  \left| \sum_{j=1}^\infty \frac{2 j^{3-4s}\, \Theta(\ln(j))}{j^4+\xi^2} -
  \frac{2}{|\xi|^{2s}} \int_0^\infty \frac{y^{3-4s}}{1+y^4}\, 
  \Theta\left(\ln(y|\xi|^{\frac12})\right) \dd y \right|
  \leq \frac{C_1 \|\Theta\|_{W^{1,\infty}}}{|\xi|^{2s+2\beta}}.
 \end{equation}
 On the other hand, for every $\xi \in \R$,
 \begin{equation}
  | \widehat{K}_\Theta(\xi) |
  =
  \left| \sum_{j\in\N^*} \frac{2j^{3-4s}\Theta(\ln(j))}{j^4+\xi^2} \right|
  \leq 
  2  \|\Theta\|_{L^\infty} \sum_{j\in\N^*} j^{-1-4s}.
 \end{equation}
 Thus, there exists $C_2>0$ independent of $\Theta$, such that, for every $\xi \in (0,\infty)$
 \begin{equation}
  \left| \widehat{K}_{\Theta}(\xi) - 2 \langle\xi\rangle^{-2s} \int_0^{+\infty} \frac{y^{3-4s}}{1+y^4}\, 
\Theta\left(\ln(y|\xi|^{\frac12})\right) \dd y \right| \leq C_2\|\Theta\|_{W^{1,\infty}} \langle\xi\rangle^{-2s-2\beta}.
 \end{equation}
 Then, by \eqref{v.frac.0} from \cref{v.frac} with $\theta=s+\beta$, there exists $C_3>0$ independent of $\Theta$, such that
 \begin{equation}
   \Bigg| \int_\R |\widehat{v}_{\omega,\tau}(\xi)|^2 \widehat{K}_{\Theta}(\xi)\dd\xi
    - \int_\R
   \frac{2 |\widehat{v}_{\omega,\tau}(\xi)|^2 }{\langle\xi\rangle^{2s}} \int_0^{+\infty} \frac{y^{3-4s}}{1+y^4}
   \Theta\left(\ln(y|\xi|^{\frac12})\right) \dd y\dd\xi \Bigg| 
   \leq \frac{C_3 \|\Theta\|_{W^{1,\infty}}}{\langle\omega\rangle^{2s+2\beta}}.
 \end{equation}
 Let $\sigma \in (0,1)$.
 Let us focus on the main term, which we decompose, the integrand being even, as
 \begin{equation}
   \int_\R 
   \frac{|\widehat{v}_{\omega,\tau}(\xi)|^2 }{\langle\xi\rangle^{2s}} \int_0^{+\infty} \frac{y^{3-4s}}{1+y^4}
   \Theta\left(\ln(y|\xi|^{\frac12})\right) \dd y\dd\xi 
   = 2(I_1 + I_2 + I_3),
 \end{equation}
 where 
 \begin{align}
   \label{i1}
   I_1 & := \int_{\R_+} 
   \frac{|\widehat{v}_{\omega,\tau}(\xi)|^2 }{\langle\xi\rangle^{2s}} 
   \int_0^{+\infty} \frac{y^{3-4s}}{1+y^4}\, 
\Theta\left(\ln(y|\omega|^{\frac12})\right) \dd y\dd\xi,
  \\
  \label{i2}
  I_2 & := \int_{|\xi-\omega|\leq\omega^\sigma} 
  \frac{|\widehat{v}_{\omega,\tau}(\xi)|^2 }{\langle\xi\rangle^{2s}} 
  \int_{0}^{+\infty} \frac{y^{3-4s}}{1+y^4} \left( 
\Theta\left(\ln(y|\xi|^{\frac12})\right) 
- \Theta\left(\ln(y|\omega|^{\frac12})\right)\right) \dd y\dd\xi,
  \\
  \label{i3}
  I_3 & := \int_{\R_+ \cap |\xi-\omega|\geq \omega^{\sigma}} \frac{|\widehat{v}_{\omega,\tau}(\xi)|^2 }{\langle\xi\rangle^{2s}} 
  \int_0^{+\infty} \frac{y^{3-4s}}{1+y^4}\, 
\left( 
\Theta\left(\ln(y|\xi|^{\frac12})\right) 
- \Theta\left(\ln(y|\omega|^{\frac12})\right)\right) \dd y \dd \xi.
 \end{align}
 The first term $I_1$ gives the claimed asymptotic, while the others are bounded with high enough decaying powers of $\omega$.
 
 \bigskip \noindent \textbf{Estimate of $I_1$.} From \eqref{i1}, \eqref{def:normeH(-s)} and \eqref{gsto}, one infers that
 \begin{equation}
  \frac{4 I_1}{2\pi} = 2 \| v_{\omega,\tau} \|_{H^{-s}}^2 G_{s,\Theta,\omega}.
 \end{equation}
 Then, thanks to estimate \eqref{v.frac.0} from \cref{v.frac}, we conclude
 \begin{equation}
  \left| \frac{4 I_1}{2\pi} - \frac{\tau}{\omega^{2s}} G_{s,\Theta,\omega} \right| \leq C \omega^{-1}. 
 \end{equation}

 \bigskip
 \noindent \textbf{Estimate of $I_2$.} 
 From \eqref{i2} and the mean value inequality, 
 \begin{equation}
  \left| I_2 \right| 
  \leq C \gamma(s) \omega^{-2s} \| \widehat{v}_{\omega,\tau} \|_{L^2}^2 
  \| \Theta' \|_{L^\infty} \sup_{|\xi-\omega|\leq\omega^\sigma} |\ln (\xi/\omega)|
  \leq C \omega^{-2s-1+\sigma} \| \Theta \|_{W^{1,\infty}}.
 \end{equation}

 \bigskip 
 \noindent \textbf{Estimate of $I_3$.} First, from \eqref{i3}, one has
 \begin{equation}
     I_3 \leq 2 \gamma(s) \|\Theta\|_{L^\infty} \int_{\R_+ \cap |\xi-\omega|\geq \omega^{\sigma}} \frac{|\widehat{v}_{\omega,\tau}(\xi)|^2 }{\langle\xi\rangle^{2s}} \dd \xi.
 \end{equation}
 Thanks to estimate \eqref{v.loc} of \cref{lemma:v.TF} (for $\lambda = 0$) and using the inequality $\tau \langle \tau x \rangle^{-1} \leq C \langle x \rangle^{-1}$ uniformly for $\tau \in [0,T^\star]$,
 \begin{equation}
    \int_{\R_+\cap|\xi-\omega|\geq \omega^{\sigma}} \frac{|\widehat{v}_{\omega,\tau}(\xi)|^2 }{\langle\xi\rangle^{2s}} \dd \xi
    \leq 
    C \int_{\R_+\cap|\xi-\omega|\geq \omega^{\sigma}}
    \langle\xi\rangle^{-2s} \left( 
    \langle\xi-\omega\rangle^{-2} + 
    \langle\xi+\omega\rangle^{-2} \right) \dd \xi.
 \end{equation}
 For the second term, we use the crude estimate
 \begin{equation}
    \int_{\xi \geq 0}
    \langle\xi\rangle^{-2s} 
    \langle\xi+\omega\rangle^{-2} \dd \xi
    \leq C \omega^{-1} 
    \int_{\xi \geq 0}
    \langle\xi\rangle^{-1-2s} \dd \xi
    \leq C \omega^{-1}.
 \end{equation}
 For the first term, we split the into three parts
 \begin{align}
    \int_{0 \leq \xi \leq \frac \omega 2}
    \langle\xi\rangle^{-2s}
    \langle\xi-\omega\rangle^{-2} \dd \xi 
    & \leq C \int_{0}^{\frac \omega 2} \omega^{-2} \dd \xi
    \leq C \omega^{-1}, \\
    \int_{\frac \omega 2 \leq \xi \leq \omega - \omega^\sigma}
    \langle\xi\rangle^{-2s}
    \langle\xi-\omega\rangle^{-2} \dd \xi 
    & \leq C \omega^{-2s} \int_{x \geq \omega^{\sigma}} \frac{\dd x}{x^2}
    \leq C \omega^{-2s-\sigma}, \\
    \int_{\xi \geq \omega + \omega^\sigma}
    \langle\xi\rangle^{-2s}
    \langle\xi-\omega\rangle^{-2} \dd \xi 
    & \leq C \omega^{-2s} \int_{x \geq \omega^{\sigma}} \frac{\dd x}{x^2}
    \leq C \omega^{-2s-\sigma}.
 \end{align}
 These estimates conclude the proof of the proposition.
\end{proof}

\begin{rk}
 In this section, we only apply \cref{danslemille_new} for the function $\Theta$ constructed by Alg.~\ref{algo_nul}. Nevertheless, we state it as a result valid for any $\Theta \in W^{1,\infty}(\R)$ because we will use it in \cref{sec:danslemille} for other base functions.
\end{rk}

\begin{corollary} \label{danslemille5}
 There exist $C,\beta> 0$ such that, for the function $\Theta \in W^{1,\infty}(\R)$, defined by Algo.~\ref{algo_nul}, for every $0\leq\tau\leq T\leq 1$ and $\omega \geq 1$ such that $\ln(\sqrt{\omega})$ is in the middle of a plateau of length $2L$ on which $\Theta$ is constant, i.e.,
 \begin{equation} \label{xLomega}
  \forall x \in [\ln\sqrt{\omega}-L,\ln\sqrt{\omega}+L], \quad \Theta(x)=\Theta(\ln\sqrt{\omega}),
 \end{equation}
 if $K_\Theta$ denotes the kernel associated to $\Theta$ by \eqref{def:K} and \eqref{cjtheta}, then,
 \begin{equation}
  \left| \int_0^T \int_0^T v_{\omega,\tau}(t) v_{\omega,\tau}(t') K_\Theta(t-t') \dd t \dd t'
-
\frac{\gamma(s) \tau}{\omega^{2s}}  \Theta(\ln\sqrt{\omega}) \right|
\leq \left( \frac{2\epsilon\gamma(s) \tau}{\omega^{2s}} + \frac{C}{\omega^{2s+2\beta}} \right),
 \end{equation}
 where the constant $\gamma(s)$ is defined in \eqref{gamma} and $v_{\omega,\tau}$ is defined in \eqref{v.omega.tau}.
\end{corollary}

\begin{proof}
 Thanks to the assumption \eqref{xLomega}, for $y \in (e^{-L}, e^{+L})$, $\ln \left(y\omega^{\frac 12}\right)$ is within the plateau of $\Theta$. Hence, recalling \eqref{gsto}, 
 \begin{equation}
  G_{s,\Theta,\omega} = \Theta\left(\ln \sqrt{\omega}\right) \int_{e^{-L}}^{e^{L}} \frac{y^{3-4s}}{1+y^4} \dd y
  + \int_{\R_+\setminus [e^{-L},e^{L}]} \frac{y^{3-4s}}{1+y^4} \Theta\left(\ln \left(y\omega^{\frac 12}\right)\right) \dd y.
 \end{equation}
 Using estimate \eqref{L.epsilon} from \cref{jeconverge} yields
 \begin{equation}
  \left| G_{s,\Theta,\omega} - \gamma(s) \Theta\left(\ln \sqrt{\omega}\right) \right| \leq 2 \epsilon \gamma(s) \| \Theta \|_{L^\infty}.
 \end{equation}
 Hence \cref{danslemille5} follows from \eqref{danslemille_new}.
\end{proof}

We set $\omega_+ := e^{2(12k_+ + 3)L}$ and $\omega_- := e^{2(12k_-+9)L}$, for $k_\pm \in \N$ large enough. Thanks to \eqref{chi.x} and \eqref{theta.chi.x}, $\ln \sqrt{\omega_\pm}$ are in the middle of negative and positive plateaus of length $2L$ of $\Theta$ (with an opposite sign convention). Thus, \cref{danslemille5} and \eqref{z.phi0} imply that, for controls of the form $c_\pm v_{\omega_\pm, \frac T4}$, where $c_\pm \in \R$ are to be chosen later, one obtains
\begin{equation} \label{z0zTphi0c}
 \langle z(T), \varphi_0 \rangle - \langle z_0, \varphi_0 \rangle
 = 
 \mp c_\pm^2 \frac{\gamma(s)T}{4 \omega_\pm^{2s}} \left(\frac 12 + \tilde{\epsilon}\right),
\end{equation}
where there exists $C > 0$ independent of $T > 0$ such that
\begin{equation}
|\tilde{\epsilon}| \leq \frac 13 + \frac{C}{T \omega_\pm^{2\beta}},
\end{equation}
by the choice of $\epsilon$ in \cref{sec:magic}.
Assuming that $\langle z_0, \varphi_0 \rangle = +1$ for example and fixing $T > 0$, we first choose $k_+$ large enough ($\omega_+$ large enough) in order to obtain $|\tilde{\epsilon}| < \frac 1 2$. Then one can choose $c_+$ using \eqref{z0zTphi0c} in order to guarantee $\langle z(T), \varphi_0 \rangle = 0$, because the term multiplying $c_+$ in the right-hand side of \eqref{z0zTphi0c} is non null.

\subsection{Construction of $L^2$ controls leaving the linear order invariant}
\label{sec:retranche}

To obtain small-time local null controllability for the full system, it is necessary to build controls realizing the elementary motions in the directions $\pm \varphi_0$ without moving the other components. The goal of this section is to construct the controls $u^{\pm}$ of \cref{prop:elementary.control.N=1} for $m=0$.

\bigskip

Let $T>0$ and $\omega \geq 1$. The sequence $(d_k(v_{\omega,\frac{T}4}))_{k \in \N}$ defined by
\begin{equation}
 d_k(v_{\omega,\frac{T}4}) := \int_0^T v_{\omega,\frac{T}4}(t) e^{-k^2(T-t)} \dd t
\end{equation}
belongs to $D_T$ (defined in \eqref{eq:DT-def}) because, for every $k \in \N$,
\begin{equation}
 \begin{split}
  |d_k(v_{\omega,\frac{T}4}) | 
  & = \left| e^{-k^2(T-\frac{T}4)} \int_0^{\frac{T}4} v_{\omega,\frac{T}4}(t) e^{-k^2(\frac{T}4-t)} \dd t \right|
  \\
  & = 
 e^{-\frac34 k^2 T} \left| \int_0^{\frac{T}4} \sin (\omega t) 
 e^{-k^2(\frac{T}4-t)} \dd t \right|
 \\
 & \leq
 e^{-\frac34 k^2 T} \frac{2}{\omega}.
 \end{split}
\end{equation}
Thus, by \cref{thm:moments}, we can consider a corrected control that leaves the linear order invariant
\begin{equation} \label{def:vtilde}
\widetilde{v}_{\omega,\frac T4} := v_{\omega,\frac T4} - \mathfrak{L}^T_1(d(v_{\omega,\frac T4})) 
\end{equation}
and there exists $C_T> 0$ such that
\begin{equation}
 \| \mathfrak{L}^T_1(d(v_{\omega,\frac T4})) \|_{L^2(0,T)} \leq 
 \frac{C_T}{\omega}.
\end{equation}
Moreover, since $K_\Theta \in L^1$, it defines a continuous bilinear form on $L^2(0,T)$. Hence, thanks to \cref{danslemille5} and defining $\omega_{\pm}$ as in the previous paragraph, there exists $k_\pm$ large enough, such that
\begin{equation}
 \pm \int_0^T \int_0^T \widetilde{v}_{\omega_\pm,\frac T4}(t)
 \widetilde{v}_{\omega_\pm,\frac T4}(t') K_\Theta(t-t') \dd t \dd t' > 0,
\end{equation}
where we used $\omega^{-2s} \gg \omega^{-1}$ for large values of $\omega$. Up to a rescaling by some constant $c_\pm(T) \in \R$, this allows to construct $u^\pm \in L^2(0,T)$ (with $u_\pm = c_\pm \widetilde{v}_{\omega_{\pm}\pm, \frac T4}$ as in the previous paragraph) such that the associated trajectories to~\eqref{eq:z} starting from $z_0 = \pm \varphi_0$ satisfy $z(T)=0$, which concludes the proof of \cref{prop:elementary.control.N=1} for $m=0$ (with $L^2$ controls).

\subsection{Construction of regular controls}
\label{sec:N=1.regular}

We construct, for $m \in \N^*$, the controls 
$u^{\pm} \in H^{m}_0(0,T)$ of \cref{prop:elementary.control.N=1},
by smoothing the controls built in the previous paragraph. This construction relies on:
\begin{itemize}
\item the density of $H^m_0(0,\frac T4)$ in $L^2(0,\frac T4)$,
to approximate $v_{\omega_\pm,\frac T4}$ in $L^2(0,\frac T4)$,
\item the continuity of the map
$u\in L^2(0,\frac T4) \mapsto \mathfrak{L}^T_1(d(u)) \in H^m_0(0,T)$,
\item which implies the continuity of the following map on $L^2(0,\frac T4)$
\begin{equation}
 u \mapsto \int_0^T \int_0^T \Big(u-\mathfrak{L}^T_1(d(u))\Big)(t)
 \Big(u-\mathfrak{L}^T_1(d(u))\Big)(t') K_\Theta(t-t') \dd t \dd t'
 .
\end{equation}
\end{itemize}
For an appropriate approximation $\overline{v}_{\omega_\pm,\frac T4}\in H^m_0(0,\frac T4)$ of 
$v_{\omega_\pm,\frac T4}$, we obtain
\begin{equation}
 \pm \int_0^T \int_0^T \Big(\overline{v}_{\omega_\pm,\frac T4}-\mathfrak{L}^T_1(d(\overline{v}_{\omega_\pm,\frac T4}))\Big)(t)
 \Big(\overline{v}_{\omega_\pm,\frac T4}-\mathfrak{L}^T_1(d(\overline{v}_{\omega_\pm,\frac T4}))\Big)(t') K_\Theta(t-t') \dd t \dd t' > 0
 .
\end{equation}
Once more, up to the choice of rescaling parameters $c_\pm(T) > 0$, this allows to construct 
\begin{equation}
 u_m^\pm := c_\pm \Big( \overline{v}_{\omega_\pm,\frac T4}-\mathfrak{L}^T_1(d(\overline{v}_{\omega_\pm,\frac T4})) \Big) \quad \in H^m_0(0,T)
\end{equation}
such that the associated trajectories to~\eqref{eq:z} starting from $z_0 = \pm \varphi_0$ satisfy $z(T)=0$.

\subsection{Controllability with quadratic cost}
\label{sec:quad.cost}

As stated in \cref{rk:quad.cost}, controllability with a quadratic cost estimate like \eqref{eq:quad.cost.N=1} implies that the quadratic approximation of the system is controllable. More precisely, if $\Gamma$ is a nonlinearity satisfying the regularity assumptions \eqref{def:Gamma}, with $\langle \Gamma[0],\varphi_0 \rangle = 0$, for which the nonlinear system~\eqref{eq:z} is $H^1$-STLNC with the cost estimate \eqref{eq:quad.cost.N=1}, then the modes $\varphi_k$ for $k \geq 1$ are controllable on the linear approximation (this can be seen as a consequence of the method explained in \cref{sec:sstlc-linear-equiv}) and the first mode $\varphi_0$ is controllable on the quadratic approximation, with controls leaving the linear order invariant.

More precisely, let us prove that, for every $T > 0$, there exist $u^\pm \in L^\infty(0,T)$ such that the associated solutions to the linear system \eqref{eq:z1} satisfy $z_1^\pm(T) = 0$ and the solutions to the quadratic system~\eqref{eq:z2} with initial data $\pm \varphi_0$ satisfy $\langle z_2^\pm(T), \varphi_0 \rangle = 0$.

Let $T > 0$. Since the nonlinear system is $H^1$-STLNC with cost estimate \eqref{eq:quad.cost.N=1}, there exists $C > 0$ such that, for every $\varepsilon > 0$, there exists a pair of controls $u^\pm_\varepsilon \in L^\infty(0,T)$ with $\|u^\pm_\varepsilon \|_{L^\infty} \leq C \varepsilon^{\frac12}$, such that the associated solutions to the nonlinear system \eqref{eq:z} with initial data $\pm \varepsilon \varphi_0$ satisfy $z^\pm_\varepsilon(T) = 0$. Since the sequences $u^\pm_\varepsilon / \varepsilon^{\frac12}$ are bounded in $H^1(0,T)$, they strongly converge (up to an extraction) towards some controls $u^\pm \in L^\infty(0,T)$. First, denoting by $z^\pm_1$ the solutions to the linear system \eqref{eq:z1} with controls $u^\pm$, since
\begin{equation} 
 0 = \frac{z_\varepsilon^\pm(T)}{\varepsilon^{\frac12}} = z_1^\pm(T)
 + \mathcal{O}(\|u^\pm - u^\pm_\varepsilon/\varepsilon^{\frac12}\|_{L^\infty}) 
 + \mathcal{O}(\varepsilon^{\frac12}),
\end{equation}
we obtain that $z^\pm_1(T) = 0$. Second, denoting by $z^\pm_2$ the solutions to the quadratic system \eqref{eq:z2} with controls $u^\pm$ and initial data $\pm \varphi_0$, since $\langle \Gamma[0], \varphi_0 \rangle = 0$, we obtain
\begin{equation} 
 0 = \langle \frac{z_\varepsilon^\pm(T)}{\varepsilon}, \varphi_0 \rangle = \langle z^\pm_2(T),\varphi_0 \rangle 
 + \mathcal{O}(\|u^\pm - u^\pm_\varepsilon/\varepsilon^{\frac12}\|_{L^\infty}) 
 + \mathcal{O}(\varepsilon^{\frac12}),
\end{equation}
which implies that $\langle z^\pm_2(T) , \varphi_0 \rangle = 0$, which concludes the proof.

\section{Recovering an infinite number of lost directions} 
\label{sec:infini}

This section is dedicated to the proof of \cref{thm:infini}. From now on, we fix a constant $s \in (0,\frac12)$ and a constant $L \geq 1$ given by \cref{jeconverge} for $\epsilon := \frac 16$. We will not recall these assumptions as they are valid through this whole section. We will not track the dependency of the constants in the estimates with respect to these (fixed) parameters.

\subsection{Construction of a magic system}

We use the same elementary building block $\chi$ defined in \eqref{chi.x}.

\paragraph{Family of oscillating functions.} We build a sequence of functions $(\Theta_k)_{k \in \N}$ in $W^{1,\infty}(\R)$ by combining elementary blocks in a diagonal-like construction. The first steps of the construction are represented in \cref{fig:thetas}. The blocks are inserted such that their supports are disjoint. For each $x \in \R_+$, there is at most one $k \in \N$ such that $\Theta_k(x) \neq 0$. For each $k \in \N$, there are infinitely many positive and negative blocks on the line of $\Theta_k$.  The sequence of functions can be seen as being constructed by Alg.\ \ref{algo}.

\begin{figure}[ht!]
 \centering
 \includegraphics{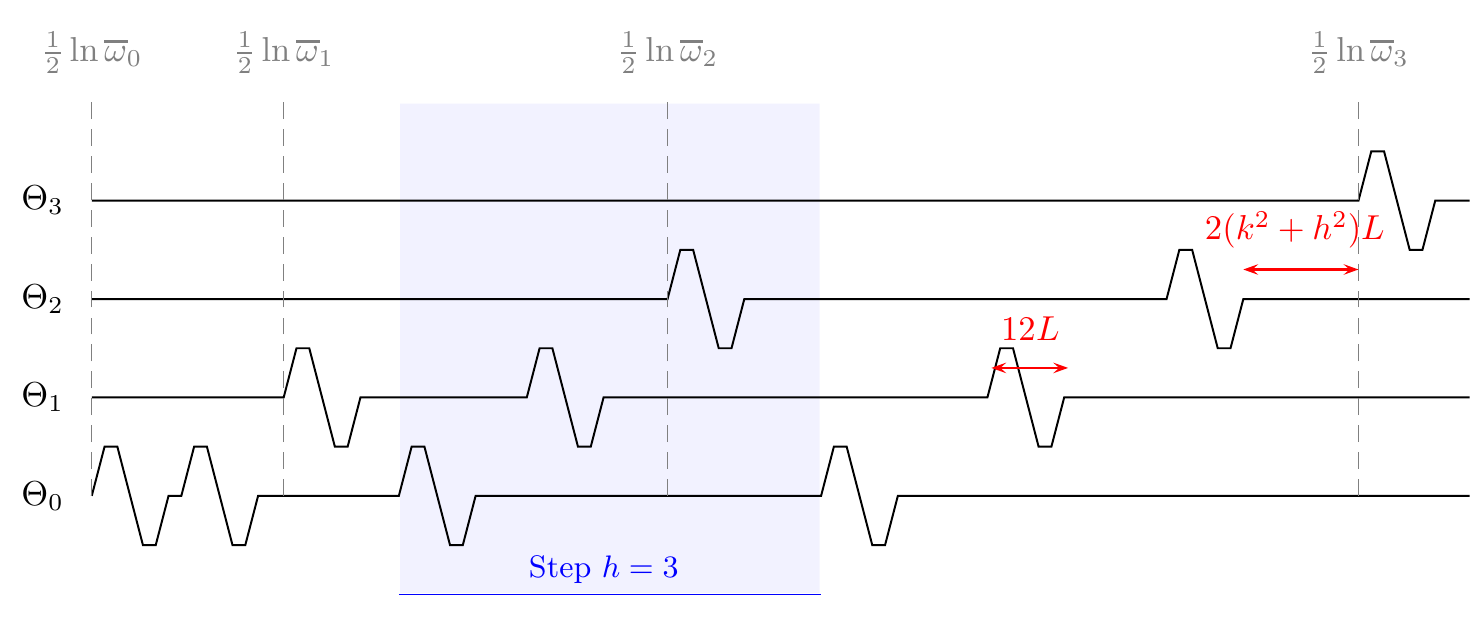}
 \caption{First steps of the construction of sparse oscillating functions from elementary blocks.}
 \label{fig:thetas}
\end{figure}

\begin{algorithm}
 \KwData{$L \geq 1$}
 \KwResult{$(\Theta_k)_{k \in \N} \in \ell^\infty(W^{1,\infty}(\R))$}
 \BlankLine
 For $k \in \N$, initialize $\Theta_k \leftarrow 0$\;
 Initialize $X \leftarrow 0$\;
 \For{$h \leftarrow 1$ \KwTo $+\infty$}{
  \For{$k \leftarrow 0$ \KwTo $h-1$}{
   $\Theta_{k,h} \leftarrow \chi\left(\frac{\cdot}{2L}-X\right) -  \chi\left(\frac{\cdot}{2L}-X-3\right)$\;
   $\Theta_k \leftarrow \Theta_k + \Theta_{k,h}$\;
   $X \leftarrow X + 6 + k^2+h^2$\;
  }
 }
 \caption{Construction of a sequence of oscillating functions}
 \label{algo}
\end{algorithm}

\bigskip

\begin{lemma}
 \label{lemma:algo}
 Let $(\Theta_k)_{k \in \N}$ be constructed as in Alg.\ \ref{algo}. There exists positive constants $\rho_1,\rho_2 > 0$ such that, for every $k\in\N$,
 \begin{equation}
  \label{estimate.algo1}
   \rho_1 L k^4
  \leq 
  \inf \supp \Theta_k 
  \leq 
  \rho_2 L k^4
 \end{equation}
 and moreover, for every $k \in \N$, $h\in\N^*$ and $l \in \N$ with $l \neq k$,
 \begin{equation}
  \label{estimate.algo2}
  \mathrm{dist} \left( \supp \Theta_{l,h} , \supp \Theta_k \right) 
  \geq 
  \rho_1 L (k^2+l^2+h^2).
 \end{equation}
\end{lemma}

\begin{proof}
 We start with the first assertion. During step $h \in \N^*$, the increment of $X$ is given by
 \begin{equation}
  X_{end}(h) - X_{begin}(h) = \sum_{k=0}^{h-1} \left(6 + k^2 + h^2\right) = 6 h + h^3 + \sum_{k=0}^{h-1} k^2.
 \end{equation}
 Thus, one has the bounds
 \begin{equation}
  h^3
  \leq 
  X_{end}(h) - X_{begin}(h)
  \leq 
  8 h^3
  .
 \end{equation}
 Hence, summing these bounds, using $X_{begin}(h)=X_{end}(h-1)$ and dropping lower order terms,
 \begin{equation}
  \frac{1}{4} h^4
  \leq 
  X_{end}(h) 
  \leq 
  8 h^4.
 \end{equation}
 Moreover, the function $\Theta_k$ gains its first positive-negative blocks pair $\Theta_{k,k}$ during the step indexed by $h_k = k+1$. Thus, for $k \geq 1$,
 \begin{align}
  \inf \supp \Theta_k 
  & 
  \leq X_{end}(h_k)
  \leq 8 h_k^4
  \leq \rho_2 k^4,
  \\
  \inf \supp \Theta_k
  &
  \geq X_{end}(h_k-1)
  \geq \frac{1}{4} k^4
 \end{align}
 and these estimates also hold for $k=0$ because $\inf \supp \Theta_0 = 0$. This proves~\eqref{estimate.algo1}. 
 
 \medskip
 
 Moreover, \eqref{estimate.algo2} is a direct consequence of the construction. Indeed, when $x_l \in \supp \Theta_{l,h}$, then $x_l$ is part of an elementary block which is prefixed by a margin of size at least $L(l-1)^2 + L(h-1)^2$ (which was inserted after a block on line $l-1$) and postfixed by a margin of size at least $Ll^2+Lh^2$. Moreover, if $x_k \in \supp \Theta_k$, it is also prefixed by a margin of size at least $L(k-1)^2$ and postfixed by a margin of size at least $Lk^2$. This yields the existence of $\rho_1 > 0$ such that \eqref{estimate.algo2} holds.
\end{proof}

In the sequel, the following frequency will play an important role
\begin{equation} \label{omk}
 \omk := e^{2 \inf \supp \Theta_k}.
\end{equation}

\begin{rk} \label{rk:omk.large}
 Up to choosing $L$ larger than our initial choice given by \cref{jeconverge}, one can assume that $\rho_1 L \geq 1$. Hence, using \eqref{omk} and \eqref{estimate.algo1}, one has $\omk \geq \exp(k^4)$ for every $k \in \N$.
\end{rk}

\paragraph{Nonlinearity.}

We consider a nonlinear heat equation for which the linearized system misses all the odd indexed directions, i.e.\ we assume that $\langle \mu, \varphi_{2j} \rangle \neq 0$ for every $j \in \N$ but $\langle \mu, \varphi_{2k+1} \rangle = 0$ for every $k \in \N$. Let $C_\Theta := 32 \gamma(s)^{-1}$. We define the nonlinearity $\Gamma_s$ as
\begin{equation} \label{Gamma_s}
 \Gamma_s[z] := \varphi_0 + \sum_{j=1}^{+\infty} j^{\frac{1}{2}-3s} \varphi_{2j}
 + C_\Theta \sum_{k=0}^{+\infty} \left( \sum_{j=1}^{+\infty} \Theta_k(\ln j) j^{\frac{1}{2}-s} \langle z, \varphi_{2j} \rangle \right) \varphi_{2k+1}.
\end{equation}
Since the family $(\Theta_k)_{k\in\N}$ is bounded in $L^\infty(\R)$ with disjoint supports, $\Gamma_s$ is well defined as a map from $H^{1}_N(0,\pi)$ to $H^{-1}_N(0,\pi)$ and satisfies the regularity assumptions from the introduction \eqref{def:Gamma}. Indeed, one has for example for $z\in H^1_N(0,\pi)$, the estimate
\begin{equation}
 \begin{split}
 \| \Gamma_s[z] - \Gamma_s[0] \|_{H^{-1}_N(0,\pi)}
 & \leq 
 C_\Theta 
 \sum_{k=0}^{+\infty}
 \left(
 \sum_{j=1}^{+\infty}
 | \Theta_k(\ln j) | j^{\frac12-s} | \langle z, \varphi_{2j} \rangle| \right) 
 \| \varphi_{2k+1} \|_{H^{-1}_N(0,\pi)}
 \\
 & \leq C_\Theta 
 \sum_{j=1}^{+\infty} 
  j^{\frac12-s} | \langle z, \varphi_{2j} \rangle|
  \sum_{k=0}^{+\infty} \frac{| \Theta_k(\ln j) |}{2k+2}
 \\
 & \leq C_\Theta 
 \left(\sum_{j=1}^{+\infty} j^2 |\langle z, \varphi_{2j} \rangle |^2 \right)^{\frac12}
 \left(\sum_{j=1}^{+\infty} j^{-1-2s} \right)^{\frac12}
 ,
 \end{split}
\end{equation}
because, for each $j\geq 1$, there is at most one $k\in \N$ such that $\Theta_k(\ln j) \neq 0$. Similar computations prove that $\Gamma_s$ satisfies all the regularity assumptions of \eqref{def:Gamma}.

\subsection{Control strategy and reduction of the proof} 

As explained in \cref{sec:control.strategy.N=1} for the case of a single lost direction, the control strategy proceeds in two steps: first, control to zero the linearly controllable modes, then, use a purely quadratic strategy, while making sure to bring back the linear state to zero. The proof of \cref{thm:infini} therefore reduces to the proof of the two following propositions.

\begin{proposition}
 For every $T > 0$, there exists $C_T > 0$ such that, for every $z_0 \in H^1_N(0,\pi)$, there exists $u \in H^1_0(0,T)$ with 
 \begin{equation} \label{cost.uj}
  \|u\|_{H^1} \leq C_T \| (\langle z_0, \varphi_{2j}\rangle)_{j\in\N} \|_{\ell^2}
 \end{equation}
 such that the associated solution $z \in Z$ to \eqref{eq:z} satisfies $\langle z(T), \varphi_{2j} \rangle = 0$ for every $j \in \N$ and
 \begin{equation} \label{zT.l1}
  \| (\langle z(T), \varphi_{2k+1} \rangle)_{k\in\N} \|_{\ell^1}
  \leq 
  \| (\langle z_0, \varphi_{2k+1} \rangle)_{k\in\N} \|_{\ell^1}
  + C_T \| (\langle z_0, \varphi_{2j}\rangle)_{j\in\N} \|_{\ell^2}^2.
 \end{equation}
\end{proposition} 

\begin{proof}
 The existence and the cost estimate \eqref{cost.uj} for $u$ are given by \cref{thm:linear-cost}. We only prove the size estimate for the odd modes.
 Due to the structure of \eqref{Gamma_s}, one has $z = z_1 + z_2$ where $z_1$ is the solution to the linear system \eqref{eq:z-linear} on the even modes and $z_2$ is the solution to the second order expansion on the odd modes, given for $k\in \N$ by
 \begin{equation}
  \langle z_2(T), \varphi_{2k+1} \rangle 
  = \langle z_0, \varphi_{2k+1} \rangle 
  + C_\Theta \int_0^T e^{-(2k+1)^2(T-t)} u(t) \left(\sum_{j=1}^{+\infty} \Theta_k(\ln j) j^{\frac12-s} \langle z_1(t),\varphi_{2j}\rangle\right) \dd t.
 \end{equation}
 Hence, one has
 \begin{equation}
  \begin{split}
   \| (\langle z(T), \varphi_{2k+1} \rangle)_{k\in\N} \|_{\ell^1}
  & - \| (\langle z_0, \varphi_{2k+1} \rangle)_{k\in\N} \|_{\ell^1}
   \\ & \leq 
   C_\Theta \|u\|_{L^\infty} \sum_{j=1}^{+\infty} j^{\frac12-s} \int_0^T |\langle z_1(t),\varphi_{2j} \rangle| \dd t
    \sum_{k=0}^{+\infty} | \Theta_k(\ln j)|
   \\
    & \leq 
    C_\Theta \|u\|_{L^\infty} \left(\sum_{j=1}^{+\infty} j^2 T \int_0^T |\langle z_1(t),\varphi_{2j} \rangle|^2 \dd t\right)^{\frac12} 
    \left(\sum_{j=1}^{+\infty} j^{-1-2s}\right)^{\frac12}
   \\
   & \leq 
   C T^{\frac12} \| u \|_{L^\infty} \| z_1 \|_{L^2(H^1)}.
  \end{split}
 \end{equation}
 Moreover, thanks to the linear well-posedness stated in \cref{thm:wp-linear}, one has from \eqref{eq:wp-linear} that $\|z_1\|_{L^2(H^1)} \leq C_T \| (\langle z_0, \varphi_{2j}\rangle)_{j\in\N} \|_{\ell^2}$. This concludes the proof of \eqref{zT.l1}.
\end{proof}

\begin{proposition} \label{prop:quad.hs}
 Let $T > 0$ and $z_0 \in L^2(0,\pi)$ such that $(\langle z_0, \varphi_{2k+1} \rangle)_{k\in\N} \in {\ell^1}$ and $\langle z_0,\varphi_{2j}\rangle = 0$ for every $j\in \N$. 
 There exists $u \in H^{-s}(0,T)$ and $z \in Z$ a solution to \eqref{eq:z} such that $z(T) = 0$ and 
 \begin{equation} \label{u.hs.z0}
  \|u\|_{H^{-s}} \leq \| (\langle z_0, \varphi_{2k+1} \rangle)_{k\in\N} \|_{\ell^1}^{\frac12}.
 \end{equation}
\end{proposition}

\begin{proof}
 It is sufficient to prove this result for $T > 0$ small enough, as one can always use a null control on the remaining time interval after reaching the null state. Thus, we can assume that $T \leq s\rho_1L$. We wish to reach the null state at time $T/4$ (this trick will lighten the computations). We look for $u \in H^{-s}(0,T/4)$ under the form $t \mapsto v(4t)$ where $v\in H^{-s}(0,T)$. By homogeneity, we decompose the state as $z = z_1 + z_2$ (the linear and the quadratic part). Since the even modes vanish at the initial time, for every $\tau \in [0,T]$, the linear state satisfies
 \begin{equation} \label{z1.T}
  \left\langle z_1\left(\frac{\tau}{4}\right), \varphi_{2j} \right\rangle = \frac{1}{4} j^{\frac{1}{2}-3s} \int_0^\tau e^{-j^2(\tau-t)} v(t)  \dd t .
 \end{equation}
 Moreover, for every $\tau \in [0,T]$, each component of the quadratic state can be rewritten as
 \begin{equation} \label{z2.T.Qk}
 \left\langle z_2\left(\frac \tau4\right), \varphi_{2k+1} \right\rangle 
 = e^{-2\lambda_k \tau} \left( \langle z_0, \varphi_{2k+1} \rangle + \frac{C_\Theta}{32} Q_{k,\tau}(v,v)\right)
 ,
 \end{equation}
 where, for $a,b \in L^2(0,T)$, $\tau \geq 0$, $\sigma \in \R$ and $k\in\N$, we define,
 \begin{align} 
 \label{def:Qk}
 Q_{k,\tau}(a,b) 
 & = \int_0^\tau \int_0^\tau a(t) e^{\lambda_k t} \cdot b(t') e^{\lambda_k t'} \cdot 
 e^{\lambda_k |t-t'|} K_{\Theta_k}(t-t') \dd t \dd t',
 \\
 K_{\Theta_k}(\sigma) 
 & := \sum_{j=0}^{+\infty} e^{-j^2|\sigma|} \langle \mu, \varphi_{2j} \rangle
 \langle \Gamma_s'[0] \varphi_{2j}, \varphi_{2k+1} \rangle
 = \sum_{j=1}^{+\infty} e^{-j^2|\sigma|} \frac{\Theta_k(\ln j)}{j^{-1+4s}},
 \\
 \label{lambdak}
 \lambda_k 
 & := \frac{(2k+1)^2}{8}.
 \end{align}
 Thus, the conclusion of \cref{prop:quad.hs} follows from the following \cref{prop:surjectif} applied to the $\ell^1$ sequence $y_k := - 32 \langle z_0, \varphi_{2k+1} \rangle / C_\Theta$. We conclude that $z(T/4) = 0$ and we extend the control and the solution by zero. Therefore, using \eqref{v.hs.y}
 \begin{equation}
   \| u \|_{H^{-s}}
   =
   \| v(4\cdot) \|_{H^{-s}}
   \leq \frac 12 \| v \|_{H^{-s}}
   \leq \frac 12 2 \gamma(s)^{-\frac 12} \left(\frac{32}{C_\Theta} \| (\langle z_0, \varphi_{2k+1}\rangle)_{k\in\N} \|_{\ell^1} \right)^{\frac12}
   .
 \end{equation}
 The choice $C_\Theta = 32 \gamma(s)^{-1}$ gives $\gamma(s)^{-\frac12} (32/C_\Theta)^{\frac12} = 1$, which enables to obtain \eqref{u.hs.z0}.
 
 \bigskip
 
 Although the nonlinear system \eqref{eq:z} may be ill-posed for controls $u \in H^{-s}(0,T)$, we prove in \cref{sec:solution} that we can indeed build a regular solution $z \in Z$ for the controls we construct.
\end{proof}

In order to ensure the convergence of some infinite sums, we work in the sequel with control times $T$ small enough, i.e.\ $T \in (0,T^\star]$, where we define
\begin{equation} \label{Tstar}
    T^\star := s \rho_1 L.
\end{equation}
In particular, recalling \eqref{estimate.algo1}, \eqref{omk} and \eqref{lambdak}, this guarantees that there exists $C > 0$ such that
\begin{align} 
    \label{tstar.omk.sum}
    \sum_{k\in\N} \langle k \rangle \omk^{-2s} e^{2\lambda_k T} & \leq C, \\
    \label{tstar.k2}
    \sum_{k\in \N} \langle k \rangle e^{-2s\rho_1 L k^2} e^{2\lambda_k T^\star} & \leq C,
\end{align}
where the second sum converges because $2\lambda_k \sim k^2$.

\begin{proposition} \label{prop:surjectif}
 For every $T \in (0,T^\star]$ and $(y_k)_{k\in\N} \in \ell^1$, there exists $v \in H^{-s}(0,T)$ satisfying $y_k = Q_{k,T}(v,v)$ for every $k\in\N$, with a size
 \begin{equation} \label{v.hs.y}
  \|v\|_{H^{-s}} \leq 2 \gamma(s)^{-\frac12} \|y\|_{\ell^1}^{\frac12}
 \end{equation}
 and which moreover drives the linear state from zero back to zero, i.e.,
 \begin{equation} \label{vj.moments}
  \forall j \in \N, \quad \int_0^T e^{-j^2 (T-t)} v(t) \dd t = 0.
 \end{equation}
\end{proposition}

The fact that we need to ensure \eqref{vj.moments} for all $j \in \N$ (and not only even values of $j$) comes from the fact that the final control $u(t)$ is computed as $v(4t)$, which thus ensures that only the linearly controllable even modes go back to zero at the final time (see e.g.\ \eqref{z1.T} in the proof of \cref{prop:quad.hs}).

The remainder of this section is devoted to the proof of \cref{prop:surjectif}. The existence of $v$ is proved in the next paragraph thanks to a Schauder fixed point argument. Heuristically, the core of the proof relies on an almost explicit construction of $v$ as an infinite sum of elementary controls oscillating at frequencies corresponding to appropriate plateaus of $\Theta_k$ (either positive or negative plateaus depending on the sign of $y_k$). The fact that we only obtain controllability for small times is not restrictive, as one can always choose a null control in the remaining time interval.

\subsection{Construction of approximate controls}
\label{sec:construction}

Let $y \in \ell^1(\R)$. We construct a sequence of controls $\widetilde{V}_n(y) \in H^{-s}(0,T)$, for which we are going to prove that $(Q_{k,T}(\widetilde{V}_n(y),\widetilde{V}_n(y)))_{k\in\N} \approx y$ in $\ell^1$ for $n$  large enough. Let $T > 0$ and $n \in \N$. We define
\begin{itemize}
 \item a sequence of scalar multipliers $(\beta_l)_{l\in\N}$ as 
 \begin{equation} \label{beta_T}
  \beta_l := \left(\frac{4}{T \gamma_{\mathrm{sign}(y_l)}(s)}\right)^{\frac12},
 \end{equation}
 where $\gamma_\pm(s) \approx \gamma(s)$ are defined below in \eqref{gamma+} and \eqref{gamma-},
 \item an increasing sequence of frequencies $(\omega_l)_{l\in\N}$ where $\omega_l \geq 1$ is chosen such that
 \begin{itemize}
 \item $\ln \sqrt{\omega_l}$ belongs to the support of $\Theta_{l,\max(l,n+1)}$,
 \item $\ln \sqrt{\omega_l}$ is in the middle of a plateau of length $2L$ of $\Theta_l$,
 \item $\Theta_l(\ln \sqrt{\omega_l}) = \mathrm{sgn}~y_l$, 
 \end{itemize}
 \item a sequence of oscillating controls for $l \in \N$ (see \eqref{def:v.omega.lambda}),
 with $\lambda_l$ defined in \eqref{lambdak}, 
 \begin{equation} \label{vl}
  v_l := v_{\omega_l,\frac T4,\lambda_l},
 \end{equation}
 \item a purely oscillating quadratic candidate control
 \begin{equation} \label{V_n}
  V_n(y) := \sum_{l \in \N} {\beta_l} |y_l|^{\frac12} \omega_l^s v_l
  ,
 \end{equation}
 \item a slightly corrected control, by some linear map $L$ defined in \cref{lemma:LV},
 \begin{equation} \label{Vn.LVn}
  \widetilde{V}_n(y) := V_n(y) - L(V_n(y))
 \end{equation}
 which drives the linear state from zero back to zero and where $L(V_n(y)) \in H^1_0(0,T)$.
\end{itemize}

\begin{rk} \label{rk.omega.n}
 The sequence of controls $\widetilde{V}_n(y)$ depends on $n$ in a subtle way. Indeed, the only place where $n$ comes into play in the definition of this sequence is that we require that $\ln \sqrt{\omega_l}$ belongs to the support of $\Theta_{l,\max(l,n+1)}$. It corresponds to the fact that, if $T$ is small, we will choose $n$ large enough such that all the frequencies $\omega_l$ (including the smallest one $\omega_1$) are large enough. Heuristically, the frequencies should at least satisfy $\omega_l \geq 1/T$ (so that the elementary controls oscillate sufficiently). Rigorously we should therefore write $\omega_l(n)$ to stress this dependency, but we will omit this indexing because the computations are already quite heavy. Instead, we will use the following properties:
 \begin{gather}
  \label{n.omcmp}
  \forall n \in \N, \forall l \in \N^*, \quad \omega_l(n) \geq \omega_1(n), \\
  \label{n.om1}
  \lim_{n \to +\infty} \omega_1(n) = + \infty, \\
  \label{n.omsum}
  \forall a > 0, \quad \lim_{n \to +\infty} \sum_{l \geq 1} \omega_l^{-a}(n) = 0.
 \end{gather}
 The last property is obtained from Alg.~\ref{algo} because $\inf \supp \Theta_{l,n+1} \geq L(l-1)^2 + Ln^2$ from which we deduce $\omega_l(n) \leq C e^{-Ll} e^{-Ln^2}$ (which implies that the sums are finite and that they tend to zero).
\end{rk}

The following proposition is the cornerstone of \cref{sec:infini} and is a consequence of the careful definition of the sequence $\widetilde{V}_n(y)$.

\begin{proposition} \label{prop:epsilon_n}
 Let $T \in (0,T^\star]$. There exists a sequence $\varepsilon_n(T) \to 0$ such that, for every $y \in \ell^1$,
 \begin{equation} 
  \label{epsilon_n}
  \sum_{k\in\N} \sup_{\tau\in[0,T]} {\langle k \rangle} \left| Q_{k,\tau}\left(\widetilde{V}_n(y),\widetilde{V}_n(y)\right) - \min\left(1,\frac{4\tau}{T}\right) y_k \right|
  \leq \varepsilon_n \| y \|_{\ell^1}
 \end{equation}
 and moreover
 \begin{align}
  \label{Vn.Hs.main}
  \| \widetilde{V}_n(y) \|_{H^{-s}} \leq (2(1-\epsilon)\gamma(s))^{-\frac12} (1 + \varepsilon_n) \| y \|_{\ell^1}^{\frac12}.
 \end{align}
\end{proposition}

\begin{proof}[Proof of \cref{prop:surjectif} from \cref{prop:epsilon_n}]
 Let $T \in (0,T^\star]$ and $n_0 \in \N$ large enough such that $\varepsilon_{n_0}(T) \leq \sqrt{2}(1-\epsilon) - 1$. Let $\bar y \in \ell^1$. We define the map
 \begin{equation}
 \mathcal{F} : 
 \left\{
 \begin{aligned}
 \ell^1(\R) & \to \ell^1(\R),
 \\
 y & \mapsto (\bar y_k + y_k - Q_{k,T}(\widetilde{V}_{n_0}(y),\widetilde{V}_{n_0}(y)))_{k\in\N}.
 \end{aligned}
 \right.
 \end{equation}
 Then, by \eqref{epsilon_n}, $\mathcal{F}$ maps the ball of radius $2 \| \bar y \|_{\ell^1}$ of $\ell^1$ into itself as long as $\varepsilon_n \leq \frac 12$. Moreover, thanks to the weight $\langle k \rangle$ in the left-hand side of \eqref{epsilon_n}, the map $\mathcal{F}$ is compact. By the Schauder fixed point theorem, there therefore exists $y \in \ell^1$ with $\|y\|_{\ell^1} \leq 2 \|\bar y\|_{\ell^1}$ such that $\mathcal{F}(y) = y$. Letting $v := \widetilde{V}_{n_0}(y)$ one has $Q_{k,T}(v,v) = \bar{y}_k$ for every $k\in\N$ and, by \eqref{Vn.Hs.main}, 
 \begin{equation}
  \|v\|_{H^{-s}} 
  \leq (2(1-\epsilon)\gamma(s))^{-\frac12} (1+\sqrt{2}(1-\epsilon)-1) \| y\|_{\ell^1}^{\frac12} 
  \leq \gamma(s)^{-\frac12} (2 \| \bar y \|_{\ell^1})^{\frac12}
 \end{equation}
 which yields the estimate \eqref{v.hs.y}. This concludes the proof of \cref{prop:surjectif}.
\end{proof}

In the sequel, we prove \cref{prop:epsilon_n}. First, the sum~\eqref{V_n} defining $V_n$ converges in $H^{-s}(0,T)$ for $y \in \ell^1$ (see \cref{lemma:Vn.Hs_new} with $\tau = T$ and $\lambda = 0$, in \cref{sec:Vn.Hs}). The correction $L(V_n(y))$ is asymptotically small in $H^1_0(0,T)$ (see \cref{lemma:LV} in \cref{sec:LVn}). Combining these facts yields the estimate \eqref{Vn.Hs.main}. To prove the main estimate~\eqref{epsilon_n}, we decompose the sum thanks to definitions \eqref{V_n} and \eqref{Vn.LVn} as
\begin{equation} \label{couleurs}
 \begin{split}
  \sum_{k\in\N} \sup_{\tau \in [0,T]} {\langle k \rangle} \Big| Q_{k,\tau} & (\widetilde{V}_n(y),\widetilde{V}_n(y)) - \min\left(1,\frac{4\tau}{T}\right) y_k \Big|
  \\
  & \leq \sum_{k\in\N} \sup_{\tau \in [0,T]} {\langle k \rangle} \Big| {\beta_k^2} \omega_k^{2s} |y_k| Q_{k,\tau}(v_k,v_k) - \min\left(1,\frac{4\tau}{T}\right) y_k \Big|
  \\
  & \quad + \sum_{k\in\N} {\langle k \rangle} \sup_{\tau \in [0,T]} \sum_{l \neq k} {\beta_l^2} \omega_l^{2s} |y_l| \left| Q_{k,\tau}(v_l,v_l) \right|
  \\
  & \quad + \sum_{k\in\N} {\langle k \rangle} \sup_{\tau \in [0,T]} \sum_{l \neq l'} {\beta_l \beta_{l'}} \omega_l^s \omega_{l'}^s |y_l|^{\frac12} |y_{l'}|^{\frac12} \left| Q_{k,\tau}(v_l,v_{l'}) \right|
  \\
  & \quad + \sum_{k\in\N} {\langle k \rangle} \sup_{\tau \in [0,T]} \big( \left| Q_{k,\tau}(L(V_n(y)),L(V_n(y))) \right| 
  + 2 \left| Q_{k,\tau}(L(V_n(y)),V_n(y)) \right| \big).
 \end{split}
\end{equation}
The first term is estimated in \cref{sec:danslemille}. The second term is estimated in \cref{sec:almost.diagonal}, the third term in \cref{sec:ortho}, while the last two terms due to the correction to bring the linear part of the state back to zero are estimated in \cref{sec:ortho}.

\subsection{Regularity and estimate of the constructed control}
\label{sec:Vn.Hs}

The convergence of the sum~\eqref{V_n} defining $V_n$ is obtained thanks to the quasi-orthogonality of the elementary oscillating controls $v_l$. We start with technical lemmas yielding a precise statement.

\begin{lemma}
 \label{lemma:orthogonal}
 There exists $C > 0$ such that, for every $\tau > 0$ and $\omega,\omega' \in \R$,
 \begin{equation} 
  \label{eq:orthogonal}
  \int_\R \langle \tau(\omega - \xi) \rangle^{-1} \langle \tau(\omega'-\xi) \rangle^{-1} \dd \xi
  \leq
  C \frac{1 + \ln \langle \tau(\omega-\omega') \rangle }{\tau \langle \tau(\omega-\omega')\rangle}
  .
 \end{equation}
\end{lemma}

\begin{proof}
 By translation and scaling, proving inequality \eqref{eq:orthogonal} reduces to proving that, for any $a > 0$,
 \begin{equation} \label{eq:axixi}
  \int_\R \langle \xi \rangle^{-1} \langle a - \xi \rangle^{-1} \dd \xi 
  \leq C \frac{1+\ln \langle a \rangle}{\langle a \rangle}.
 \end{equation}
 This inequality is non-trivial only for large values of $a$. We split the integral in four parts
 \begin{align}
  & \int_{\xi < -\frac{a}{2}} \langle \xi \rangle^{-1} \langle a - \xi \rangle^{-1} \dd \xi 
  \leq 
  \int_{\xi < -\frac{a}{2}} \langle \xi \rangle^{-2} \dd \xi
  \leq C 
  \langle a \rangle^{-1},
  \\ &
  \int_{|\xi|\leq\frac{a}{2}} \langle \xi \rangle^{-1} \langle a - \xi \rangle^{-1} \dd \xi
  \leq C 
  \langle a \rangle^{-1} \int_{|\xi|\leq \frac{a}{2}} \langle \xi \rangle^{-1} \dd \xi 
  \leq C
  \langle a \rangle^{-1} \ln \langle a \rangle,
  \\ &
  \int_{|\xi-a|<\frac{a}{2}} \langle \xi \rangle^{-1} \langle a - \xi \rangle^{-1} \dd \xi
  \leq C
  \langle a \rangle^{-1} \int_{|\xi-a|<\frac{a}{2}} \langle a - \xi \rangle^{-1} \dd \xi
  \leq C
  \langle a \rangle^{-1} \ln \langle a \rangle,
  \\ &
  \int_{\xi \geq \frac32 a} \langle \xi \rangle^{-1} \langle a - \xi \rangle^{-1} \dd \xi
  \leq 
  \int_{\xi' \geq \frac{a}{2}} \langle \xi' \rangle^{-2} \dd \xi
  \leq C
  \langle a \rangle^{-1}.
 \end{align}
 Summing these four estimates proves \eqref{eq:axixi} and concludes the proof of the result.
\end{proof}

\begin{lemma} \label{lemma:ortho}
 There exists $C_s > 0$ such that, for $\lambda,\lambda' \in \R$, $\omega,\omega' \geq 1$ with $\omega \geq 2 \omega'$,  $\tau \in [0,1]$,
 \begin{equation} \label{ortho}
  \int_\R | \widehat{v}_{\omega,\tau,\lambda}(\xi) \widehat{v}_{\omega',\tau,\lambda'}(\xi) | \dd \xi 
  \leq C_s (1+e^{-\lambda \tau})(1+ e^{-\lambda' \tau}) \omega^{-s-\frac12(\frac12-s)} \omega'^{-s-\frac12(\frac12-s)}
  .
 \end{equation}
\end{lemma}

\begin{proof}
 Thanks to estimate \eqref{v.loc} from \cref{lemma:v.TF} and \eqref{eq:orthogonal} from \cref{lemma:orthogonal},
 \begin{equation}
  \begin{split}
   &  \int_\R | \widehat{v}_{\omega,\lambda}(\xi) \widehat{v}_{\omega',\lambda'}(\xi) | \dd \xi
  \\
  & \leq 16 \tau^2 (1+e^{-\lambda \tau})(1+ e^{-\lambda'\tau})
  \int_\R \langle \tau(\omega\pm\xi)\rangle^{-1} \langle \tau(\omega'\pm\xi)\rangle^{-1} \dd \xi
  \\
  & \leq C \tau (1+e^{-\lambda \tau})(1+ e^{-\lambda'\tau})
  (\langle \tau(\omega+\omega') \rangle^{-\frac12-s}
  + \langle \tau(\omega-\omega') \rangle^{-\frac12-s}
  )
  .
  \end{split}
 \end{equation}
 We used the assumption $s < \frac12$ so that $s+\frac12 < 1$ and we conclude using $\omega \geq 2 \omega'$.
\end{proof}

\begin{proposition} \label{lemma:Vn.Hs_new}
 Let $T > 0$. There exists a sequence $\varepsilon_n(T) \to 0$ such that, for every $y \in \ell^1$, every $\lambda \geq 0$ and every $\tau \in [0,T]$,
 \begin{equation} \label{Vn.Hs_new}
  \| t \mapsto \mathbf{1}_{[0,\tau]}(t) e^{\lambda t} V_n(y)(t) \|_{H^{-s}} \leq (2 {(1-\epsilon)}\gamma(s))^{-\frac12} (1+\varepsilon_n) e^{\lambda T} \|y\|_{\ell^1}^{\frac12}
  .
 \end{equation}
\end{proposition}

\begin{proof}
 Let $0 \leq \tau \leq T$. We define a shorthand notation
 \begin{equation}
  \label{tau4}
  \tau_4 := \min \left( \tau, \frac T 4 \right).
 \end{equation}
 In particular, for $l\in\N$, $\mathbf{1}_{[0,\tau]}(t) e^{\lambda t} v_l(t) = v_{\omega_l,\tau_4,\lambda_l-\lambda}(t)$. From definition \eqref{def:normeH(-s)} of the $H^{-s}$ norm,
 \begin{equation} \label{Vn.Hs.1}
  \begin{split}
   \| \mathbf{1}_{[0,\tau]}(t) e^{\lambda t} V_n(y)(t) \|_{H^{-s}}^2 
   & = \frac{1}{2\pi} \int_\R \langle\xi\rangle^{-2s} \left| \sum_{l\in\N} {\beta_l} |y_l|^{\frac12}\omega_l^s \widehat{v}_{\omega_l,\tau_4,\lambda_l-\lambda}(\xi) \right|^2 \dd \xi
   \\
   & \leq \sum_{l\in\N} |y_l| {\beta_l^2} \omega_l^{2s} \| v_{\omega_l,\tau_4,\lambda_l-\lambda} \|_{H^{-s}(\R)}^2
   \\
   & \quad
   +  \sum_{l\neq l'} |y_l|^{\frac12} |y_{l'}|^{\frac12} {\beta_l \beta_{l'}} \omega_l^s \omega_{l'}^s 
   \int_\R \left| \widehat{v}_{\omega_l,\tau_4,\lambda_l-\lambda}(\xi) \widehat{v}_{\omega_{l'},\tau_4,\lambda_{l'}-\lambda}(\xi) \right| \dd \xi 
   .
  \end{split}
 \end{equation}
Thanks to estimate \eqref{v.frac.lambda} from \cref{v.frac}, there exists $C_1 > 0$ such that the first term in the right-hand side of \eqref{Vn.Hs.1} is bounded as
 \begin{equation} \label{vnhs1}
  \begin{split}
   \sum_{l\in\N} {\beta_l^2} |y_l| \omega_l^{2s} & \| v_{\omega_l,\tau_4,\lambda_l-\lambda} \|_{H^{-s}(\R)}^2 
   \\ & \leq 
  \sum_{l\in\N} |y_l| \omega_l^{2s} \left(  \frac{|1-e^{-2(\lambda_l-\lambda) \tau_4}|}{T|\lambda_l-\lambda|{\gamma_{\mathrm{sign}(y_l)}}(s)} \frac{\omega_l^2}{\omega_l^2+(\lambda_l-\lambda)^2} \langle\omega_l\rangle^{-2s} + \frac{C_1}{T} e^{2\lambda T} \omega_l^{-1} \right)
  \\
  & \leq 
  \left( \frac{1}{2{(1-\epsilon)}\gamma(s)} + \frac{C_1 }{\omega_1^{1-2s} T} \right) e^{2\lambda T} 
  \| y \|_{\ell^1}
  ,
  \end{split}
 \end{equation}
 where we used \eqref{n.omcmp}. Thanks to estimate \eqref{ortho} from \cref{lemma:ortho}, there exists $C_s > 0$ such that
 \begin{equation}
  \int_\R \left| \widehat{v}_{\omega_l,\tau_4,\lambda_l-\lambda}(\xi) \widehat{v}_{\omega_{l'},\tau_4,\lambda_{l'}-\lambda}(\xi) \right| \dd \xi 
  \leq 4 C_s e^{2\lambda T} \omega_l^{-s-\frac12(\frac12-s)} \omega_{l'}^{-s-\frac12(\frac12-s)}
  .
 \end{equation}
 Hence, the second term in the right-hand side of \eqref{Vn.Hs.1} is bounded as
 \begin{equation} \label{vnhs2}
  \begin{split}
	\sum_{l\neq l'} {\beta_l\beta_{l'}} |y_l|^{\frac12} |y_{l'}|^{\frac12} \omega_l^s \omega_{l'}^s  &
\int_\R \left| \widehat{v}_{\omega_l,\tau_4,\lambda_l-\lambda}(\xi) \widehat{v}_{\omega_{l'},\tau_4,\lambda_{l'}-\lambda}(\xi) \right| \dd \xi 
\\ & \leq \frac {C_s}{T}
 e^{2\lambda T} \| y \|_{\ell^1} \sum_{l\in\N} \omega_l^{-(\frac12-s)}.
  \end{split}
 \end{equation}
  Combining \eqref{vnhs1} and \eqref{vnhs2} concludes the proof of \eqref{Vn.Hs_new}
  thanks to \eqref{n.om1} and \eqref{n.omsum} respectively.
\end{proof}

\subsection{Diagonal action of elementary controls}
\label{sec:danslemille}

We estimate the first sum in \eqref{couleurs}. For $\Theta \in W^{1,\infty}(\R_+,\R)$ and $\lambda \in \R$, we define a modified kernel
\begin{equation} \label{K.Theta.lambda}
 K_{\Theta,\lambda}(\sigma) :=
 e^{\lambda|\sigma|} K_{\Theta}(\sigma)
 .
\end{equation}
We start with a technical lemma, which notably helps to compare $K_{\Theta,\lambda}$ with $K_\Theta$.

\begin{lemma} \label{lemma:khat}
 There exists $C > 0$ such that, for every $k \in \N$ and $\xi \in \R$,
 \begin{align}
  \label{khat1}
  |\widehat{K}_{\Theta_k}(\xi)| & \leq C \omk^{-2s},
  \\
  \label{khat2}
  |\widehat{K}_{\Theta_k,\lambda_k}(\xi)| & \leq C \omk^{-2s},
  \\
  \label{khat3}
  |\widehat{K}_{\Theta_k,\lambda_k}(\xi) - \widehat{K}_{\Theta_k}(\xi)| & \leq C \lambda_k \omk^{-2s}|\xi|^{-1}
  .
 \end{align}
\end{lemma}

\begin{proof}
 Thanks to the support properties of $\Theta_k$ (see the definition \eqref{omk} of $\omk$), $\Theta_k(\ln j) \neq 0$ implies that $j^2 \geq \omk$. Moreover, thanks to the definition \eqref{lambdak} of $\lambda_k$ and \cref{rk:omk.large} concerning $\omk$, $\omk \geq 2 \lambda_k$ for every $k\in \N$. Thus, $j^2 \geq 2 \lambda_k$ for the indexes $j$ that we need to consider. For $\lambda \in \{0,\lambda_k\}$,
 \begin{equation} \label{khat.sum}
  \widehat{K}_{\Theta_k,\lambda}(\xi) = 2 \sum_{j^2\geq\omk}
  j^{1-4s} \Theta_k(\ln(j)) \frac{j^2-\lambda}{(j^2-\lambda)^2+\xi^2}.
 \end{equation}
 Since $j^2 \geq 2\lambda_k \geq 2 \lambda$, $(j^2-\lambda)/((j^2-\lambda)^2+\xi^2) \leq 2 j^{-2}$ and this yields, using a straightforward monotonic series-integral comparison,
 \begin{equation}
  \left| 
   \widehat{K}_{\Theta_k,\lambda}(\xi)
  \right|
  \leq 
  2 \| \Theta_k \|_{L^\infty} \sum_{j^2 \geq \omk} j^{1-4s} \frac{2}{j^2}
  \leq 
  C_s \omk^{-2s}
  .
 \end{equation}
 This proves the inequalities \eqref{khat1} and \eqref{khat2}. We move on to the last inequality. Since we use indexes for which $j^2 \geq 2\lambda_k$, the mean value theorem applied to $x \mapsto x / (x^2+\xi^2)$, yields, for $\xi \in \R$,
 \begin{equation}
  \left| \frac{j^2-\lambda_k}{(j^2-\lambda_k)^2+\xi^2} - \frac{j^2}{j^4+\xi^2} \right|
  \leq 
  \frac{4\lambda_k}{j^4+\xi^2}
  .
 \end{equation}
 Proceeding as for the previous inequalities concludes the proof of \eqref{khat3}.
\end{proof}

\begin{lemma} \label{danslemille_BIS_new}
 There exists $C,\beta > 0$ such that, for every $0 \leq \tau \leq T \leq T^\star$ (defined in \eqref{Tstar}), every $k\in\N$ and $\omega_k$ such that $\ln \sqrt{\omega_k}$ is in the middle of a plateau of $\Theta_k$ of length $2L$ of value $\pm 1$, there holds
 \begin{equation} \label{eq:mille_BIS_new}
  \left| Q_{k,\tau}(v_k,v_k) \mp \min\left(\tau,\frac{T}{4}\right) \gamma_{\pm}(s) \omega_k^{-2s} \right|
  \leq 
  C e^{-\beta n^2 - \beta k^2} \omega_k^{-2s} + C \omega_k^{-2s-2\beta}
  ,
 \end{equation}
 where the constants $\gamma_\pm(s) > 0$ are defined below.
\end{lemma}

\begin{proof}
\emph{Step 1: We transform the expression of $Q_{k,\tau}(v_k,v_k)$.}
Using the definitions \eqref{def:Qk} of~$Q_k$, \eqref{vl} of $v_{k}$, \eqref{K.Theta.lambda} of $K_{\Theta,\lambda}$ and \eqref{tau4} of $\tau_4$, one has
\begin{equation} \label{QkT(v)_decomp}
 \begin{split}
   Q_{k,\tau}(v_k,v_k)
   & = \int_0^{\tau_4} \int_0^{\tau_4} v_{\omega_k,\tau_4}(t) v_{\omega_k,\tau_4}(t') K_{\Theta_k,\lambda_k}(t-t') \dd t \dd t'
   \\
   & = \int_0^{\tau_4} \int_0^{\tau_4} v_{\omega_k,\tau_4}(t) v_{\omega_k,\tau_4}(t') K_{\Theta_k}(t-t') \dd t \dd t'
   \\
   & \quad + \int_0^{\tau_4} \int_0^{\tau_4} v_{\omega_k,{\tau_4}}(t) v_{\omega_k,{\tau_4}}(t') (K_{\Theta_k,\lambda_k}-K_{\Theta_k})(t-t') \dd t \dd t'.
 \end{split}
\end{equation}
\emph{Step 2: We study the first term of the right hand side of (\ref{QkT(v)_decomp}).}
By \cref{danslemille_new}, there exists $C_1,\beta_1>0$ such that
\begin{equation}
 \Big| 
 \int_0^{\tau_4}\int_0^{\tau_4} v_{\omega_k,\tau_4}(t) v_{\omega_k,\tau_4}(t') K_{\Theta_k}(t-t') \dd t \dd t' 
 -  \tau_4 \omega_k^{-2s} G_{s,\Theta_k,\omega_k}
 \Big|
 \leq C_1 \omega_k^{-2s-2\beta_1}.
\end{equation}
From Algo.~\ref{algo} and the choice of $\omega_k$, which is in the middle of a plateau of length $2L$ of $\Theta_{k,\max(k,n+1)}$, the closest non-zero pattern of $\Theta_k$ is at distance at least $\rho_1 L (k^2+n^2)$ of $\omega_k$ (see \eqref{estimate.algo2}). Hence, recalling \eqref{gsto},
\begin{equation}
 G_{s,\Theta_k,\omega_k} = \int_{e^{-9L}}^{e^{9L}} \frac{y^{3-4s}}{1+y^4} \Theta_k\left(\ln y\sqrt{\omega_k}\right) \dd y
 + \int_{\R_+\setminus[\Omega^{-1},\Omega]} \frac{y^{3-4s}}{1+y^4} \Theta_k\left(\ln y \sqrt{\omega_k}\right) \dd y,
\end{equation}
where $\Omega := \exp (\rho_1 L (k^2+n^2))$. From the convergence of the integral in $y$ and using $\|\Theta_k\|_{L^\infty} \leq 1$, the second term is bounded by $C \Omega^{-\sigma}$ for some $C,\sigma > 0$. The first term is $\pm \gamma_\pm(s)$ (depending on wether $\omega_k$ is in the middle of a plateau of value $\pm 1$), where we set
\begin{align} 
 \label{gamma+}
 \gamma_+(s) := & \int_{e^{-3L}}^{e^{9L}} \frac{y^{3-4s}}{1+y^4}
 \left( \chi\left(\frac{\ln y + 3L}{2L}\right) - \chi\left(\frac{\ln y - 3L}{2L}\right) \right) \dd y, \\
 \label{gamma-}
 \gamma_-(s) := - & \int_{e^{-9L}}^{e^{3L}} \frac{y^{3-4s}}{1+y^4}
 \left( \chi\left(\frac{\ln y + 9L}{2L}\right) - \chi\left(\frac{\ln y + 3L}{2L}\right) \right) \dd y.
\end{align}
Moreover, thanks to \eqref{L.epsilon}, $(1-\epsilon) \gamma(s) \leq \gamma_\pm(s) \leq \gamma(s)$. Letting $\beta_2 := \sigma \rho_1 L$, this steps proves that
\begin{equation} \label{ccl.step2}
 \begin{split}
 \Big| 
 \int_0^{\tau_4}\int_0^{\tau_4} v_{\omega_k,\tau_4}(t) v_{\omega_k,\tau_4}(t') K_{\Theta_k}(t-t') \dd t \dd t' 
 & \mp \tau_4 \omega_k^{-2s} \gamma_{\pm}(s)
 \Big|
 \\
 & \leq C e^{-\beta_2 n^2 -\beta_2 k^2} \omega_k^{-2s} + C_1 \omega_k^{-2s-2\beta_1}.
 \end{split}
\end{equation}
\emph{Step 3: We study the second term of the right hand side of (\ref{QkT(v)_decomp}).} Working as in the proof of \cref{danslemille_new} (to use Plancherel's formula), we obtain
\begin{equation} \label{step3.1}
 \int_0^{\tau_4}\int_0^{\tau_4} v_{\omega_k,{\tau_4}}(t) v_{\omega_k,{\tau_4}}(t')  
\Big( K_{\Theta_k,\lambda_k} - K_{\Theta_k} \Big)(t-t')  \dd t \dd t'
= \frac{1}{\pi} \int_0^{\infty} |\widehat{v}_{\omega_k,{\tau_4}}(\xi)|^2 
\Big( \widehat{K}_{\Theta_k,\lambda_k} - \widehat{K}_{\Theta_k} \Big)(\xi) \dd \xi.
\end{equation}
We split the integral between low and high frequencies. On the one hand, by \eqref{v.hat}, there exists $C>0$ such that, for every $k \in \N$ and $0 < \xi < \frac12\omega_k$,
\begin{equation}
 |\widehat{v}_{\omega_k,{\tau_4}}(\xi)|^2 \leq C \omega_k^{-2}.
\end{equation}
Thus, thanks to estimates \eqref{khat1} and \eqref{khat2} from \cref{lemma:khat}, one has
\begin{equation} \label{vk.low}
 \int_0^{\frac12\omega_k} |\widehat{v}_{\omega_k,{\tau_4}}(\xi)|^2 
\left| \widehat{K}_{\Theta_k,\lambda_k} - \widehat{K}_{\Theta_k} \right|(\xi) \dd \xi
 \leq 
 \frac12\omega_k \cdot C \omega_k^{-2} \cdot 2 C \omk^{-2s}
 \leq 
 C^2
 \omega_k^{-1}
 .
\end{equation}
On the other hand, for $\xi \geq \frac12 \omega_k$, thanks to estimate \eqref{khat3} from \cref{lemma:khat}, one has
\begin{equation} 
 \int_{\frac12\omega_k}^{+\infty} |\widehat{v}_{\omega_k,{\tau_4}}(\xi)|^2 
\left| \widehat{K}_{\Theta_k,\lambda_k} - \widehat{K}_{\Theta_k} \right|(\xi) \dd \xi
 \leq 
 C_1 \lambda_k \omk^{-2s} \frac{2}{\omega_k} \int_\R |\widehat{v}_{\omega_k,{\tau_4}}(\xi)|^2  \dd \xi
 .
\end{equation}
By definition \eqref{omk} of $\omk$ and the lower bound \eqref{estimate.algo1} for $\ln \omk$, there exists $C_2$ such that $\lambda_k \omk^{-2s} \leq C_2$. Moreover, by \eqref{v.frac.0} from \cref{v.frac}, $\| v_{\omega_k,\tau_4} \|_{L^2}^2 \leq \tau_4 / 2 + 4 C_0 \omega_k^{-1}$. Hence
\begin{equation} \label{vk.high}
 \int_{\frac12\omega_k}^{+\infty} |\widehat{v}_{\omega_k,{\tau_4}}(\xi)|^2 
\left| \widehat{K}_{\Theta_k,\lambda_k} - \widehat{K}_{\Theta_k} \right|(\xi) \dd \xi
\leq C_1 C_2 \omega_k^{-1} (\tau_4 + 8 C_0 \omega_k^{-1}).
\end{equation}
Thanks to \eqref{step3.1}, \eqref{vk.low} and \eqref{vk.high}, we have $C_3,\beta_3 >0$ such that
\begin{equation} \label{ccl.step3}
 \left| \int_0^{\tau_4}\int_0^{\tau_4} v_{\omega_k,{\tau_4}}(t) v_{\omega_k,{\tau_4}}(t')  
\Big( K_{\Theta_k,\lambda_k} - K_{\Theta_k} \Big)(t-t')  \dd t \dd t' \right|
\leq C_3 \omega_k^{-2s-2\beta_3}.
\end{equation}
\emph{Conclusions}. Combining equality \eqref{QkT(v)_decomp} of Step 1, estimate \eqref{ccl.step2} of Step 2 and \eqref{ccl.step3} of Step 3 gives the estimate \eqref{eq:mille_BIS_new} which concludes the proof.
\end{proof}

\begin{proposition}
 Let $T \in (0,T^\star]$. There exists a sequence $\varepsilon_n(T) \to 0$ such that
 \begin{equation}
  \sum_{k\in\N} \sup_{\tau \in [0,T]} \langle k \rangle  \left| \beta_k^2 \omega_k^{2s} |y_k| Q_{k,\tau}(v_k,v_k) - \min\left(1,\frac{4\tau}{T}\right) y_k \right|
  \leq \varepsilon_n \|y\|_{\ell^1}
  .
 \end{equation}
\end{proposition}

\begin{proof}
 First, for $k\in\N$, $\Theta_k(\ln\sqrt{\omega_k}) = \mathrm{sign}(y_k)$ by the choice of $\omega_k$ in \cref{sec:construction}. Hence, recalling definition \eqref{beta_T} of $\beta_k$ and applying estimate \eqref{eq:mille_BIS_new} from \cref{danslemille_BIS_new} yields, for every $\tau \in [0,T]$,
 \begin{equation}
  \left| \beta_k^2 \omega_k^{2s} |y_k| Q_{k,\tau}(v_k,v_k) - \min\left(1,\frac{4\tau}{T}\right) y_k \right|
  \leq \left( \frac{Ck}{T} e^{-\beta n^2 - \beta k^2} + \frac{Ck}{T}\omega_k^{-2\beta} \right) |y_k|.
 \end{equation}
 Since $\omega_k \geq \omk$, recalling \eqref{tstar.omk.sum} and \eqref{n.omsum}, summing these inequalities concludes the proof of the estimate.
\end{proof}

\subsection{Estimate of almost diagonal terms}
\label{sec:almost.diagonal}

We estimate the second term in \eqref{couleurs}. We start with the following lemma.

\begin{lemma} \label{Khat_L}
 There exists $C>0$ such that, for every $k,l \in \mathbb{N}$ such that $l \neq k$ and $\xi \in \left(\frac{\omega_l}{2},\frac{3\omega_l}{2} \right)$,
 \begin{equation} \label{eq:Khat_L}
  | \widehat{K}_{\Theta_k,\lambda_k}(\xi) | \leq C \omega_l^{-2s} e^{-4s\rho_1L(k^2+n^2)}
  .
 \end{equation}
\end{lemma}

\begin{proof}
 Let $k,l \in \mathbb{N}$ be such that $l \neq k$ and $j \in \N$ such that $\Theta_k(\ln(j)) \neq 0$. Then, thanks to \eqref{estimate.algo2},
 \begin{equation}
   \big| \ln j - \ln \sqrt{\omega_l} \big| \geq \rho_1 L (k^2+n^2) =: \kappa.
 \end{equation}
 Thus, either $j \geq \sqrt{\omega_l} e^{+\kappa}$ or $j \leq \sqrt{\omega_l} e^{-\kappa}$. Therefore, for every $\xi \in \left(\frac{\omega_l}{2},\frac{3\omega_l}{2} \right)$, recalling the formula~\eqref{khat.sum}, and using once more that $j^2 \geq 2\lambda_k$ when $\ln j \in \supp \Theta_k$,  
 \begin{equation}
  \begin{split}
   | \widehat{K}_{\Theta_k,\lambda_k}(\xi) | 
   & \leq 
   8 \sum_{j\leq \sqrt{\omega_l}e^{-\kappa}} \frac{j^{3-4s}}{j^4+\xi^2} + 8 \sum_{j \geq \sqrt{\omega_l} e^{+\kappa}} \frac{j^{3-4s}}{j^4+\xi^2}
   \\ & \leq
   32 \sum_{j\leq \sqrt{\omega_l}e^{-\kappa}} \frac{j^{3-4s}}{\omega_l^2}
   + 8 \sum_{j \geq \sqrt{\omega_l} e^{+\kappa}} j^{-1-4s}
   \\ & \leq
   C \omega_l^{-2s} e^{-(4-4s)\kappa} + C \omega_l^{-2s} e^{-4s \kappa}
   \\ & \leq C \omega_l^{-2s} e^{-4s \kappa}
  \end{split}
 \end{equation}
 because $s\in(0,\frac12)$, which concludes the proof of \eqref{eq:Khat_L}.
\end{proof}

\begin{proposition}
 Let $T \in (0,T^\star]$. There exists a sequence $\varepsilon_n(T) \to 0$ such that, for every $y \in \ell^1$,
 \begin{equation}
  \sum_{k\in\N} \sup_{\tau \in [0,T]} {\langle k \rangle} \sum_{l\neq k} {\beta_l^2} |y_l| \omega_l^{2s} |Q_{k,\tau}(v_l,v_l)|
  \leq \varepsilon_n \|y\|_{\ell^1}
  .
 \end{equation}
\end{proposition}

\begin{proof}
 First, for $k,l \in \N$, Plancherel's formula and definitions \eqref{def:v.omega.lambda} and \eqref{tau4} yield the equality
\begin{equation}
 Q_{k,\tau}(v_l,v_l)
 = \frac{1}{2\pi} \int_\R  \widehat{K}_{\Theta_k,\lambda_k}(\xi) | \widehat{v}_{\omega_l,\tau_4,\lambda_l-\lambda_k}(\xi)|^2 \dd \xi
 .
\end{equation}
Using Fubini's theorem, to prove the result, it suffices to bound $\beta_T^2 (I_1 + I_2 + I_3)$ where we set
\begin{align}
 I_1 & := \sum_{l} {\beta_l^2} |y_l| \omega_l^{2s}  
\sum_{k \neq l, \omk > \omega_l}  
\sup_{\tau \in [0,T]} {\langle k \rangle}
\int_{\mathbb{R}} |\widehat{K}_{\Theta_k,\lambda_k}(\xi)| | \widehat{v}_{\omega_l,{\tau_4},\lambda_l-\lambda_k}(\xi)|^2 \dd \xi,
 \\
 I_2 & :=\sum_{l} {\beta_l^2} |y_l| \omega_l^{2s}  
\sum_{k \neq l, \omk < \omega_l} 
\sup_{\tau \in [0,T]} {\langle k \rangle}
\int_{|\xi-\omega_l|>\omega_l/2} 
|\widehat{K}_{\Theta_k,\lambda_k}(\xi)| | \widehat{v}_{\omega_l,{\tau_4},\lambda_l-\lambda_k}(\xi)|^2 \dd \xi,
\\
I_3 & :=\sum_{l} {\beta_l^2} |y_l| \omega_l^{2s}  
\sum_{k \neq l, \omk < \omega_l} 
\sup_{\tau \in [0,T]} {\langle k \rangle}
\int_{|\xi-\omega_l|<\omega_l/2} 
|\widehat{K}_{\Theta_k,\lambda_k}(\xi)| | \widehat{v}_{\omega_l,{\tau_4},\lambda_l-\lambda_k}(\xi)|^2 \dd \xi.
\end{align}

\paragraph{Bound on $I_1$.} Using estimate \eqref{khat2} from \cref{lemma:khat} and \eqref{v.loc} from \cref{lemma:v.TF}, there exists $C_1 > 0$ such that
\begin{equation}
 \begin{split}
  I_1 
  & \leq 
  \sum_{l} {\beta_l^2} |y_l| \omega_l^{2s}  
\sum_{k \neq l, \omk > \omega_l} {\langle k \rangle} \| \widehat{K}_{\Theta_k,\lambda_k} \|_{L^\infty} \sup_{\tau \in [0,T]} \| \widehat{v}_{\omega_l,{\tau_4},\lambda_l-\lambda_k} \|_{L^2}^2
  \\ & \leq 
  {\frac{C_1}{T}} \sum_{l}  |y_l| \omega_l^{2s}  
\sum_{k \neq l, \omk > \omega_l} {\langle k \rangle} \omk^{-2s} e^{2\lambda_kT^\star}.
 \end{split}
\end{equation}
Moreover, by construction (see \eqref{estimate.algo2}), when $\omega_l < \omk$,
\begin{equation}
 \omega_l \leq \omk e^{-\rho_1 L (k^2+n^2)}.
\end{equation}
Hence, for $l \in \N$, thanks to \eqref{tstar.k2}, one has
\begin{equation}
 \sum_{k \neq l, \omk > \omega_l} {\langle k \rangle} \omega_l^{2s} \omk^{-2s} e^{2\lambda_k T^\star}
 \leq 
 e^{-2s\rho_1Ln^2} \sum_{k\in\N} {\langle k \rangle} e^{- 2 s \rho_1 L k^2} e^{2\lambda_k T^\star} \leq C_s e^{-2s\rho_1Ln^2}
\end{equation}
from which we conclude that $I_1 / \|y\|_{\ell^1} \to 0$ as $n\to \infty$ for fixed $T > 0$.

\paragraph{Bound on $I_2$.} Using \eqref{v.loc} from \cref{lemma:v.TF}, there exists $C_2 > 0$ such that
\begin{equation}
 \int_{|\xi-\omega_l|>\omega_l/2} |\widehat{v}_{\omega_l,\tau_4,\lambda_l - \lambda_k}(\xi)|^2 \dd\xi \leq C_2 \omega_l^{-1} e^{2\lambda_k T^\star}.
\end{equation}
Combined with estimate \eqref{khat2} from \cref{lemma:khat},
\begin{equation}
 I_2 \leq C_2 \sum_{l} {\beta_l^2} |y_l| \omega_l^{2s-1} \sum_{k\neq l} {\langle k \rangle} \omk^{-2s} e^{2\lambda_kT^\star}
 \leq {\frac{C_2}{T}} \|y\|_{\ell^1} (\inf_l \omega_l)^{2s-1},
\end{equation}
thanks to \eqref{tstar.omk.sum}. From \eqref{n.omcmp} and \eqref{n.om1}, we conclude that $I_2 / \|y\|_{\ell^1} \to 0$ as $n \to \infty$.

\paragraph{Bound on $I_3$.} Using the $L^\infty$ estimate \eqref{eq:Khat_L} from \cref{Khat_L} on $\widehat{K}_{\Theta_k,\lambda_k}$ and the same $L^2$-estimate of the functions $v_{\omega_l,\tau_4,\lambda_l - \lambda_k}$ as for $I_1$,
we obtain
\begin{equation}
 I_3 \leq C_3 \sum_{l} {\beta_l^2} |y_l| \omega_l^{2s}  
\sum_{k \neq l, \omk < \omega_l} {\langle k \rangle}
\omega_l^{-2s} e^{-4s\rho_1L(k^2+n^2)} e^{2\lambda_k T^\star}
\leq {\frac{C_3}{T}} e^{-4s\rho_1 Ln^2} \|y\|_{l^1},
\end{equation}
because the sum in $k$ converges thanks to \eqref{tstar.k2},
so that $I_3/\|y\|_{\ell^1} \to 0$ as $n\to \infty$.
\end{proof}

\subsection{Estimate of non-diagonal terms}
\label{sec:ortho}

We can now turn to the estimate of the third term in the right-hand side of \eqref{couleurs}. 

\begin{proposition}
 Let $T \in (0,T^\star]$. There exists a sequence $\varepsilon_n(T) \to 0$ such that, for every $y \in \ell^1$,
 \begin{equation}
  \sum_{k\in\N} \sup_{\tau \in [0,T]} {\langle k \rangle} \sum_{l\neq l'} {\beta_l\beta_{l'}} |y_l|^{\frac12} |y_{l'}|^{\frac12} \omega_l^s \omega_{l'}^s   |Q_{k,\tau}(v_l,v_{l'})| 
  \leq \varepsilon_n \| y \|_{\ell^1}
  .
 \end{equation}
\end{proposition}

\begin{proof}
First, for every $k,l,l' \in \N$ with $l\neq l'$, thanks to estimate \eqref{khat2} from \cref{lemma:khat} and estimate \eqref{ortho} from \cref{lemma:ortho},
\begin{equation}
 \begin{split}
  |Q_{k,\tau}(v_l,v_{l'})|
  & 
  \leq \frac{1}{2\pi} \| \widehat{K}_{\Theta_k,\lambda_k} \|_{L^\infty} \int_\R | \widehat{v}_{\omega_l,\tau_4,\lambda_l-\lambda_k}(\xi) \widehat{v}_{\omega_{l'},\tau_4,\lambda_{l'}-\lambda_k}(\xi) | \dd \xi
  \\
  & \leq C \omk^{-2s} e^{2\lambda_k T^\star}\omega_l^{-s-\frac12(\frac12-s)} \omega_{l'}^{-s-\frac12(\frac12-s)}.
 \end{split}
\end{equation}
Moreover, $k \omk^{-2s} e^{2\lambda_k T^\star} \leq C$ thanks to \eqref{tstar.omk.sum}. Thus, 
\begin{equation}
 \sum_{k\in\N} \sup_{\tau \in [0,T]} {\langle k \rangle} \sum_{l\neq l'} {\beta_l\beta_{l'}} |y_l|^{\frac12} |y_{l'}|^{\frac12} \omega_l^s \omega_{l'}^s   |Q_{k,\tau}(v_l,v_{l'})| 
 \leq \frac{C}{T} \|y\|_{\ell^1} \sum_{l\in\N} \omega_l^{-(\frac12-s)}
 ,
\end{equation}
where the last sum over $l$ is finite and tends to zero as $n \to +\infty$ by \eqref{n.omsum}.
\end{proof}

\subsection{Estimates of the linear correction}
\label{sec:LVn}

To obtain a null controllability result, we need to bring the odd modes (which have a quadratic dependency on the control) to zero, but we also need to guarantee that the even modes (which correspond to the linearly controllable part) also vanish at the final time. Therefore, we must subtract from our candidate control $V_n(y)$ a correction $L(V_n(y))$ which lifts its moments so as to satisfy the condition \eqref{vj.moments}. The key point is that this correction is small (and will thus not modify too much the dynamics on the odd modes), because the moments of our oscillating controls are already small, as we prove in the following lemma.

\begin{lemma}
 \label{lemma:moments.v.omega.lambda}
 For every $T,\omega > 0$, $\lambda \geq 0$ and $j \in \N$,
 \begin{equation} \label{eq:moments.v.omega.lambda}
  \left| \int_0^T v_{\omega,\frac T4,\lambda}(t) e^{-j^2(T-t)} \dd t \right| 
  \leq 2 \omega^{-1} e^{-\frac{3}{4}j^2T}.
 \end{equation}
\end{lemma}

\begin{proof}
 We compute the integral using the definition \eqref{def:v.omega.lambda} of $v_{\omega,\frac T4,\lambda}$
 \begin{equation}
 \begin{split}
  \int_0^T v_{\omega,\frac T4,\lambda}(t) e^{-j^2(T-t)} \dd t
  & = e^{-j^2 T} \int_0^{\frac T4} \sin(\omega t) e^{(j^2-\lambda) t} \dd t
  \\
  & = \frac{e^{-j^2T}}{2i} \left[ \frac{e^{(i\omega-\lambda+j^2) t}}{i\omega-\lambda+j^2}
  \right]_0^{\frac T4} 
  - \frac{e^{-j^2T}}{2i} \left[ \frac{e^{(-i\omega-\lambda+j^2) t}}{-i\omega-\lambda+j^2}
  \right]_0^{\frac T4} 
 \end{split}
 \end{equation}
 and we conclude that \eqref{eq:moments.v.omega.lambda} holds.
\end{proof}

From the smallness of each individual moment, we deduce the existence of a corrector $L(V_n(y))$ and its smallness in $H^1_0(0,T)$, thanks to the results on the solvability of moment problems.

\begin{lemma}
 \label{lemma:LV}
 Let $T > 0$. There exists a sequence $\varepsilon_n(T) \to 0$ such that, for every $y \in \ell^1$, there exists a control $L(V_n(y)) \in H^1_0(0,T)$, such that $V_n(y)-L(V_n(y))$ is such that the linear order returns to zero and one has
 \begin{equation}
  \label{eq:LV.petit}
  \| L(V_n(y)) \|_{H^1} \leq \varepsilon_n \| y \|_{\ell^1}^{\frac12}.
 \end{equation}
\end{lemma}

\begin{proof}
 Obtaining $z_1(T) = 0$ is equivalent to solving the problem
 \begin{equation}
  \int_0^T V_n(y)(t) e^{-j^2(T-t)} \dd t
  = \int_0^T L(V_n(y))(t) e^{-j^2(T-t)} \dd t 
  .
 \end{equation}
 We apply \cref{thm:moments} (with $m=1$ and $\eta_1 = 3/4$).
 Thanks to estimate \eqref{eq:u-moments-size}, there exists $L(V_n(y)) \in H^1_0(0,T)$ such that $z_1(T) = 0$ and which satisfies
 \begin{equation}
  \| L(V_n(y)) \|_{H^1} 
  \leq 
  M_1 e^{2M_1/T} \sup_{j \in \N} e^{\frac{3}{4}j^2T} 
  \left| \int_0^T V_n(y)(t) e^{-j^2(T-t)} \dd t \right|
  .
 \end{equation}
 Moreover, thanks to estimate \eqref{eq:moments.v.omega.lambda} from \cref{lemma:moments.v.omega.lambda} and definition \eqref{V_n} of $V_n$, there holds
 \begin{equation}
  e^{\frac{3}{4}j^2T} \left| \int_0^T V_n(y)(t) e^{-j^2(T-t)} \dd t \right|
  \leq
  2 \beta_T \sum_{l\in\N} |y_l|^{\frac12} \omega_l^{s-1}
  \leq 
  2 \beta_T \|y\|_{\ell^1}^{\frac 12} \left( \sum_{l\in\N} \omega_l^{2s-2} \right)^{\frac12}.
 \end{equation}
 This concludes the proof, thanks to \eqref{n.omsum}.
\end{proof}

Eventually, we can now estimate the last two terms from the right-hand side of \eqref{couleurs}. 

\begin{proposition}
 \label{prop:LV.LV}
 Let $T \in (0,T^\star]$. There exists a sequence $\varepsilon_n(T) \to 0$ such that, for every $y \in \ell^1$,
 \begin{equation}
  \label{eq:LV.LV}
   \sum_{k\in\N} \sup_{\tau \in [0,T]} {\langle k \rangle} \Big( \left| Q_{k,\tau}\big(L(V_n(y)), L(V_n(y))\big) \right|+ \left| Q_{k,\tau}\big(L(V_n(y)), V_n(y)\big) \right| \Big)
   \leq \varepsilon_n \| y \|_{\ell^1}
  .
 \end{equation}
\end{proposition}

\begin{proof}
 Let $\tau \in [0,T]$. We introduce the notation $\psi_k(t) := \mathbf{1}_{[0,\tau]}(t) e^{\lambda_k t}$. By Parseval's formula and definition \eqref{def:Qk} of $Q_{k,\tau}$, one has, for every $a,b \in L^2(0,T)$,
 \begin{equation}
  \big| Q_{k,\tau}(a,b) \big| \leq 
  \frac{1}{2\pi} \| \widehat{K}_{\Theta_k,\lambda_k} \|_{L^\infty} \int_\R |\widehat{a\psi_k}(\xi)\widehat{b\psi_k}(\xi)| \dd \xi
  .
 \end{equation}
 Hence, thanks to estimate \eqref{khat2} from \cref{lemma:khat}, there exists $C > 0$ such that
 \begin{align}
  \big| Q_{k,\tau}(a,a) \big| 
  & \leq C \omk^{-2s} \| a \psi_k \|_{L^2}^2 
  \leq C \omk^{-2s} e^{2\lambda_k T^\star} \| a \|_{L^2}^2,
  \\
  \big| Q_{k,\tau}(a,b) \big| 
  & \leq C \omk^{-2s} \| a \psi_k \|_{H^s} \| b \psi_k \|_{H^{-s}} 
  .
 \end{align}
 Noting that $\psi_k = v_{0,\tau,-\lambda_k}$ (see \eqref{def:v.omega.lambda}) and applying estimate \eqref{v.loc} from \cref{lemma:v.TF} proves that there exists $C > 0$ such that
 \begin{equation}
  \| \psi_k \|_{H^s} \leq C e^{\lambda_k T^\star}.
 \end{equation}
 Thus, on the one hand, using a Kato-Ponce type or \emph{fractional Leibniz} inequality (for a recent presentation, see e.g.\ \cite[Theorem~1]{MR3200091}), there exists $C > 0$ such that
 \begin{equation}
 \| a \psi_k \|_{H^s} \leq C (\| a \|_{H^s} \| \psi_k \|_{L^\infty} + \| \psi_k \|_{H^s} \| a \|_{L^\infty})
 \leq C e^{\lambda_k T^\star} \| a\|_{H^1(0,T)}
 .
 \end{equation}
 On the other hand, thanks to \cref{lemma:Vn.Hs_new}, we obtain
 \begin{equation}
  \| V_n(y) \psi_k \|_{H^{-s}}
  \leq C e^{\lambda_k T^\star} \| y \|_{\ell^1}^{\frac12}.
 \end{equation}
 We conclude that
 \begin{align}
  \big| Q_{k,\tau}(L(V_n(y)),L(V_n(y))) \big| 
  & \leq C \omk^{-2s} e^{2\lambda_k T^*} \| L(V_n(y)) \|_{H^1(0,T)}^2, \\
  \big| Q_{k,\tau}(L(V_n(y)),V_n(y)) \big| 
  & \leq C \omk^{-2s} e^{2\lambda_k T^*} \| L(V_n(y)) \|_{H^1(0,T)} \| y \|_{\ell^1}^{\frac12}.
 \end{align}
 Combining these bounds with the estimate \eqref{eq:LV.petit} of $L(V_n(y))$ in $H^1(0,T)$ and the convergence of $\sum_k \langle k \rangle \omk^{-2s}e^{2\lambda_kT^\star}$ (see \eqref{tstar.omk.sum}) concludes the proof of \eqref{eq:LV.LV}.
\end{proof}

\subsection{Well-posedness and regularity of the solutions}
\label{sec:solution}

In the previous paragraphs, we constructed a control $v \in H^{-s}(0,T)$ to prove null controllability. For a general nonlinearity $\Gamma$ satisfying the assumption \eqref{def:Gamma} and a general control $v \in H^{-s}(0,T)$ it is \emph{a priori} not clear that the system \eqref{eq:z} is well posed. In \cref{Prop:WP+source}, we proved well-posedness for small controls in $L^\infty(0,T)$, with solutions $z \in Z$ (see \eqref{Z}).

In this paragraph, we prove that, for our particular nonlinearity $\Gamma_s$ (see \eqref{Gamma_s}) and for our particular control $v = \widetilde{V}_n(y)$, we can indeed construct a solution $z \in Z$ to \eqref{eq:z}.

\paragraph{Estimate for the linear part (even modes).} For the linear part, we can decompose $v = V_n(y) - L(V_n(y))$. Since $L(V_n(y)) \in H^1_0(0,T)$, well-posedness for this control is already known. Therefore, we only need to check that the part corresponding to the action of $V_n(y)$ makes sense for any $y \in \ell^1$. First, working as in the proof of \cref{lemma:moments.v.omega.lambda} and recalling that $\omega_l \geq 2\lambda_l$, we get that, for any $\tau\in(0,T)$,
\begin{equation}
 \left| \int_0^\tau e^{-j^2(\tau-t)} v_{\omega_l,\frac T4,\lambda_l}(t) \dd t \right|
 \leq \frac{C}{\omega_l+j^2}.
\end{equation}
Let $\kappa > 0$. Thanks to the definition \eqref{V_n} of $V_n(y)$,
\begin{equation}
 \begin{split}
 \left| \int_0^\tau e^{-j^2(\tau-t)} V_n(y)(t) \dd t \right|
 & \leq C \sum_{l\in\N} |y_l|^{\frac12} \frac{\omega_l^{s}}{\omega_l+j^2} 
 \\
 & \leq C \|y\|_{\ell^1}^{\frac12} \left(\sum_{l \in \N} \frac{\omega_l^{2s}}{(\omega_l+j^2)^2}\right)^{\frac12}
 \\ 
 & \leq C \|y\|_{\ell^1}^{\frac12} j^{-2+2s+\kappa},
 \end{split}
\end{equation}
because $\sum_{l \in \N} \omega_l^{-\kappa} < + \infty$ (see \eqref{n.omsum}). Thus, recalling \eqref{z1.T},
\begin{equation}
 \left| \left\langle z_1\left(\frac{\tau}{4}\right), \varphi_{2j} \right\rangle \right|
 \leq C \|y\|_{\ell^1}^{\frac12} j^{-\frac32-s+\kappa}
 .
\end{equation}
Setting $\kappa = \frac s2$ proves that the linear state $z_1$ actually belongs to $C^0([0,T];H^1_N(0,\pi))$ for our control.

\paragraph{Estimate for the quadratic part (odd modes).} For the quadratic part, \cref{prop:epsilon_n} and the explicit formula formula \eqref{z2.T.Qk} prove that there exists $M \in (0,+\infty)$ (which of course depends on the choice of $n$ and $y$) such that
\begin{equation}
 \sum_{k\in\N} \sup_{\tau\in[0,T]} e^{2\lambda_k\tau} | \langle z_2(\tau/4),\varphi_{2k+1} \rangle |
 \leq M
\end{equation}
for our particular choice of control $v$. In particular, the dominated convergence theorem can be used to prove that $z_2 \in C^0([0,T];L^2(0,\pi))$. Moreover,
\begin{equation}
 \begin{split}
  \int_0^{T} \| z_2(\tau/4) \|_{H^1}^2 \dd \tau
  & \leq \sum_{k\in\N} (2k+1)^2\int_0^T e^{-4 \lambda_k \tau} e^{4 \lambda_k \tau} |\langle z_2(\tau/4),\varphi_{2k+1} \rangle|^2 \dd \tau
  \\
  & \leq \sum_{k\in\N} (2k+1)^2 \left(\int_0^T e^{-4\lambda_k \tau} \dd \tau\right) \left(\sup_{\tau\in[0,T]} e^{2\lambda_k\tau} | \langle z_2(\tau/4),\varphi_{2k+1} \rangle |\right)^2
  \\
  & \leq  \sum_{k\in\N} \frac{(2k+1)^2}{4\lambda_k} \left(\sup_{\tau\in[0,T]} e^{2\lambda_k\tau} | \langle z_2(\tau/4),\varphi_{2k+1} \rangle |\right)^2
  \leq 8 M^2.
 \end{split}
\end{equation}
Hence $z_2 \in L^2((0,T);H^1_N(0,\pi))$. Therefore, we constructed a solution $z \in Z$ to \eqref{eq:z}.

\subsection{On the uniform cost estimate for small times}
\label{sec:cost1}

One of the surprising features of the nonlinear system we construct is that the cost estimate to control the odd modes in small time does not blow up as the time tends to zero. This is linked to the fact that the nonlinear system is ill-posed for controls with $H^{-s}$ regularity. We give below two  examples of uniform cost estimates in linear settings.

\paragraph{Linear systems in finite dimension.} Let $n \in \N^*$, $A \in \mathcal{M}_n(\R)$ and $b \in \R^n$. Assume that the pair $(A,b)$ satisfies the Kalman rank condition. Then, the scalar-input linear system $\dot{x} = Ax + bu$ is (in particular) null controllable in small-time. In \cite{MR923278}, it is shown that the control cost, measured in $L^2$ norm on $u$ blows up like $T^{-n+\frac12}$ as $T \to 0$. In \cite{MR1450356}, this analysis is extended to the case of $W^{m,p}$ norms with $m \in \N$ and $p \in [1,+\infty]$ and it is shown that the control cost blows up like $T^{-n-m+\frac{1}p}$. If one considers negative Sobolev norms, one can obtain non blowing-up costs.

\begin{proposition}
 There exists $C > 0$ such that, for every $T > 0$ and every $x^* \in \R^n$, there exists $u \in C^\infty([0,T])$ with $\|u\|_{W^{-n,\infty}} \leq C | x^* |$ such that the associated trajectory of $\dot{x} = Ax + bu$ for the initial state $x(0) = x^*$ and the control $u$ satisfies $x(T) = 0$. This property does not hold when the size of the control is measured in $W^{-n+1,\infty}$ norm.
\end{proposition}

\begin{proof}
 Since the pair $(A,b)$ is controllable, one can assume that it is in the control normal form for which, if $x = (x_1, \ldots x_n)$, one has $\dot{x}_1 = u$ and $\dot{x}_{j+1} = x_j$ for $j \in \{1,\ldots n-1\}$. Then, by definition of the $W^{-n,\infty}$ norm, $\|u\|_{W^{-n,\infty}} = \| x_n \|_{L^\infty}$. Let $x^* \in \R^n$. Using the flatness approach, we set $x_n(t) := \theta(t)$, where $\theta \in C^\infty([0,T])$ is such that $\theta^{(j)}(0) = x^*_{n-j}$ and $\theta^{(j)}(T) = 0$ for $1 \leq j \leq n$. It is clear that this can be done with $\|\theta\|_{L^\infty} \leq |x^*_n|$. 
 
 On the contrary, if we decompose the solution $x = \bar{x} + x^u$ where $\bar{x}$ is associated with the initial data and $x^u$ with the control. Then, $|x^u_n| \leq \int_0^T |x_{n-1}| 
 \leq T \|u\|_{W^{-n+1,\infty}}$, which prevents a uniform control cost in $W^{-n+1,\infty}$ norm for the controls.
\end{proof}

In this example, obtaining a uniform cost estimate in small time for $n$ directions only happens in $W^{-n,\infty}$ norm. In \cref{thm:infini}, we recover an infinite number of directions with a uniform cost estimate in small time in $H^{-s}$ norm for some fixed $s \in (0,\frac 12)$.

\paragraph{Linear system in infinite dimension.} We consider an abstract system governed by an infinite dimensional, time-dependent, ordinary differential equation $\dot{x}_k = f_k(t) u(t)$ for $k \in \N$, where
\begin{equation}
f_k(t) := \mathbf{1}_{\supp\Theta_k}\left(e^{\frac 1 t}\right),
\end{equation}
and $\Theta_k$ was constructed by Alg.\ \ref{algo}. Since the $\Theta_k$ are uniformly bounded in $L^\infty(\R)$ with disjoint supports, the system is well-posed in the following sense. For every $T > 0$, and every $u \in L^1(0,T)$, the state $x$ belongs to $C^0([0,T];\ell^1)$. Indeed, for every $0  \leq t_0 < t_1 \leq T$,
\begin{equation}
 \| x(t_1) - x(t_0) \|_{\ell^1}
 = \sum_{k \in \N} \left| \int_{t_0}^{t_1} f_k(t) u(t) \dd t \right|
 \leq 
 \int_{t_0}^{t_1} \left( \sum_{k \in \N} | f_k(t) | \right) |u(t)| \dd t
 \leq \| u \|_{L^1(t_0,t_1)}
 .
\end{equation}
This system is uniformly small-time null controllable.

\begin{proposition}
 For every $T > 0$ and every $x^* \in \ell^1$, there exists $u \in L^1(0,T)$ satisfying $\|u\|_{L^1} \leq \|x^*\|_{\ell^1}$ such that the associated solution with initial data $x^*$ satisfies $x(T)=0$.
\end{proposition}

\begin{proof}
 Let $T > 0$. We fix $n \geq 1$ large enough such that, for every $k\in\N$, 
 \begin{equation} 
  \supp \Theta_{k,\max(k,n)} \subset (e^{\frac1T},+\infty).
 \end{equation}
 For every $k\in\N$, we then define the scalar quantities
 \begin{equation}
  \nu_k := \int_0^T \mathbf{1}_{\supp\Theta_{k,\max(k,n)}}\left(e^{\frac 1 t}\right) \dd t > 0.
 \end{equation}
 Let $x^* \in \ell^1$. We set
 \begin{equation}
  u(t) := - \sum_{k \in\N} \frac{x^*_k}{\nu_k} \mathbf{1}_{\supp\Theta_{k,\max(k,n)}}\left(e^{\frac 1 t}\right).
 \end{equation}
 Then $\|u\|_{L^1} \leq \|x^*\|_{\ell^1}$ and, since the supports of the $\Theta_k$ are disjoint, $x(T)= 0$.
\end{proof}

\section{Perspectives}

In view of the results proved in this work, the following open questions seem interesting.

\begin{itemize}
 
 \item The integer-order drifts are related to Lie brackets of the vector fields defining the dynamics. Is it possible to find a geometric interpretation of the fractional drifts?
 
 \item We proved that one can recover the smooth small-time local null controllability with quadratic cost when one direction is lost at the linear order. We also proved that an infinite number of directions can be recovered with $H^{-s}$ controls. Is it possible to recover an infinite number of lost directions with smooth controls?
 
 \item This work concerns scalar-input control systems governed by parabolic equations. Can analogous results be obtained for time-reversible systems, like Schrödinger or Korteweg-de-Vries control systems? A difficulty might be that the quadratic kernels could be less regular.

\end{itemize}

\section*{Acknowledgements}

The authors wish to thank gratefully two anonymous referees for their careful reading of this work and their suggestions, which improved both the proofs and their exposition.

\bibliographystyle{plain}
\bibliography{bibliography}

\end{document}